%% file: main.tex
\documentclass[11pt,letterpaper]{article}
\input{preamble}
\title{Quillen's local parameter conjecture\\[3pt]
       in dimension four}
\input{author}
\date{}
\begin{document}
\maketitle
\begin{abstract}
\input{abstract}
\end{abstract}

\clearpage
\begingroup
\setlength{\parskip}{0pt}
\tableofcontents
\endgroup
\clearpage
\part{The completion criterion}
\input{sections/00_theorem_and_bridge}
\par\bigskip\needspace{12\baselineskip}
\part{Integral degree-two algebraic K-theory}
\input{sections/01_k2}
\par\bigskip\needspace{12\baselineskip}
\part{Hermitian theory and local motivic cohomology}
\input{sections/02_hermitian_and_support}
\input{sections/03_motivic_local_comparisons}
\par\bigskip\needspace{12\baselineskip}
\part{Global vanishing and projective modules}
\input{sections/04_global_mw4}
\input{sections/05_projective_conclusion}
\bigskip
\needspace{8\baselineskip}
\input{frontmatter/ai_disclosure}
\bigskip
\needspace{8\baselineskip}
\phantomsection
\addcontentsline{toc}{section}{References}
\begingroup
\linespread{1.03}\selectfont
\bibliographystyle{alpha}
\bibliography{references,references_hermitian,references_motivic,references_bridge,references_k2,references_context}
\endgroup
\end{document}

%% file: preamble.tex
\usepackage[T1]{fontenc}
\usepackage[utf8]{inputenc}
\usepackage{lmodern,amsmath,amssymb,amsthm,mathtools}
\usepackage[margin=1in,headheight=14pt]{geometry}
\usepackage{microtype,enumitem,longtable,booktabs,tabularx,xurl,needspace}
\usepackage[hidelinks]{hyperref}
\usepackage{bookmark,fancyhdr}
\hypersetup{pdftitle={Quillen's local parameter conjecture in dimension four},pdfauthor={Yunpeng Li},pdfsubject={The mixed-characteristic dimension-four local Quillen problem},pdfkeywords={projective modules, local Quillen problem, mixed characteristic, algebraic K-theory, Milnor-Witt K-theory},pdfdisplaydoctitle=true,bookmarksopen=true,bookmarksopenlevel=1}
\newcommand{\proofstep}[1]{\par\medskip\needspace{3\baselineskip}\noindent\textit{#1}\quad\ignorespaces}
\setlist{itemsep=2pt,topsep=5pt}\allowdisplaybreaks[2]
\numberwithin{equation}{section}
\newtheorem{theorem}{Theorem}[section]
\newtheorem{proposition}[theorem]{Proposition}
\newtheorem{lemma}[theorem]{Lemma}
\newtheorem{corollary}[theorem]{Corollary}
\theoremstyle{definition}\newtheorem{definition}[theorem]{Definition}
\theoremstyle{remark}\newtheorem{remark}[theorem]{Remark}
\newcommand{\Z}{\mathbb Z}\newcommand{\Q}{\mathbb Q}
\newcommand{\F}{\mathbb F}\newcommand{\A}{\mathbb A}
\newcommand{\KM}{K^{\mathrm M}}\newcommand{\KMW}{K^{\mathrm{MW}}}
\newcommand{\Khat}{\widehat K^{\mathrm M}}
\newcommand{\CH}{\mathrm{CH}}\newcommand{\Zar}{\mathrm{Zar}}
\newcommand{\et}{\mathrm{\acute et}}\newcommand{\mot}{\mathrm{mot}}
\newcommand{\ind}{\mathrm{ind}}
\newcommand{\Fil}{\operatorname{Fil}}\newcommand{\gr}{\operatorname{gr}}

\newcommand{\fm}{\mathfrak m}
\newcommand{\xto}[1]{\xrightarrow{#1}}

\DeclareMathOperator{\Spec}{Spec}\DeclareMathOperator{\Frac}{Frac}
\DeclareMathOperator{\Pic}{Pic}\DeclareMathOperator{\Br}{Br}
\DeclareMathOperator{\im}{im}\DeclareMathOperator{\coker}{coker}

\DeclareMathOperator{\Um}{Um}\DeclareMathOperator{\GW}{GW}\DeclareMathOperator{\KO}{KO}
\DeclareMathOperator{\SL}{SL}\DeclareMathOperator{\GL}{GL}
\DeclareMathOperator{\ESp}{ESp}\DeclareMathOperator{\Sp}{Sp}
\DeclareMathOperator{\hofib}{hofib}
\DeclareMathOperator{\diag}{diag}
\DeclareMathOperator{\res}{res}\DeclareMathOperator{\cor}{cor}
\DeclareMathOperator{\Div}{div}
\DeclareMathOperator{\Sing}{Sing}
\newcommand{\colim}{\mathop{\mathrm{colim}}}

\newcommand{\Hyp}{\mathsf H}
\newcommand{\Hypq}{\mathsf H_{\mathrm q}}
\newcommand{\hsym}{h_{\mathrm{sym}}}
\newcommand{\Ddual}{\mathbb D}
\newcommand{\Bclass}{\mathrm B}
\newcommand{\NFr}{N_{\mathrm{Fr}}}
\DeclareMathOperator{\mult}{mult}

%% file: author.tex
\author{Yunpeng Li\\[5pt]
  \small School of Data Science\\
  \small The Chinese University of Hong Kong, Shenzhen, China\\[3pt]
  \small \href{mailto:liyunpeng@cuhk.edu.cn}{\nolinkurl{liyunpeng@cuhk.edu.cn}}}

%% file: abstract.tex
Let $R$ be a Noetherian regular local ring of mixed characteristic
and dimension four, and let $r$ be a member of a regular system of
parameters. We prove that every finitely generated projective module over
$R[1/r]$ is free. We also prove that $\Sp_4(R[1/r])$ acts transitively on unimodular
columns of length four and that every determinant orientation of a
rank-two projective module is induced by a basis. The argument reduces freeness
to two vanishing results: the degree-three Zariski cohomology of the
naive degree-four Milnor--Witt sheaf and the alternating Witt group of $R[1/r]$.
We establish these results using integral degree-two algebraic
$K$-theory, supported hermitian $K$-theory, and local motivic cohomology.
A finite-rank symplectic comparison relates the image of the completion
obstruction to Milnor--Witt multiplication and yields the required
symplectic frames.

%% file: sections/00_theorem_and_bridge.tex
\section{Introduction}\label{main:setup}
\input{frontmatter/context}

Throughout, rings are commutative with identity and projective modules are finitely generated. A regular local ring is Noetherian. A \emph{regular parameter} in a regular local ring $(R,\fm)$ is an element of $\fm\setminus\fm^2$; it can be included in a regular system of parameters.

\begin{theorem}[The four-dimensional local Quillen theorem]\label{main:quillen}
Let $(R,\fm)$ be a regular local ring of mixed characteristic, with $\dim R=4$, and let $r$ be a regular parameter. Put $A=R[1/r]$. Every projective $A$-module is free. Moreover, $\Sp_4(A)$ is transitive on unimodular columns of length four. A rank-two projective $A$-module with any chosen determinant orientation admits a basis inducing that orientation.
\end{theorem}

Here mixed characteristic means $\operatorname{char}R=0$ and $\operatorname{char}(R/\fm)=p>0$.
The parameter open $\Spec A$ has dimension at most three, and all its
residue fields are infinite, even when $R/\fm$ is finite.
The elementary argument is recorded in Step~1 of the proof of
Theorem~\ref{global:naive_mw4}.
These facts supply the dimension and residue-field hypotheses of
Theorem~\ref{bridge:criterion}.

\subsection{Two vanishing theorems}

The proof rests on two vanishing theorems:
\begin{equation}\label{main:arithmetic_endpoints}
 H^3_{\Zar}(\Spec A,\mathcal K^{\mathrm{MW}}_4)=0,
 \qquad W_{\mathrm{alt}}(A)=0.
\end{equation}
Here $\mathcal K^{\mathrm{MW}}_4$ is the sheaf associated with the \emph{naive unit-symbol presentation}, and $W_{\mathrm{alt}}$ is the stable Witt group of perfect alternating forms. The first vanishing is Theorem~\ref{global:naive_mw4}; the second is Theorem~\ref{gw:global_walt}. Part~II establishes the integral degree-two $K$-theory results used in their proofs, and Part~III develops the local motivic and supported hermitian calculations.

Theorem~\ref{bridge:criterion} converts the first vanishing into a symplectic completion theorem. It compares the obstruction at the characteristic-zero generic point and then extends the resulting cohomological vanishing over the affine scheme.

\input{frontmatter/roadmap}

\subsection{Unit symbols and symplectic conventions}

\begin{definition}\label{notation:mw_definition}
For a ring $B$, $K_*^{\mathrm{MW}}(B)$ denotes the graded unit-symbol presentation with generators $[a]$ of degree one for $a\in B^\times$ and $\eta$ of degree $-1$, and relations
\[
 [ab]=[a]+[b]+\eta[a][b],\qquad [a][1-a]=0,
 \qquad \eta[a]=[a]\eta,\qquad \eta(\eta[-1]+2)=0,
\]
where the Steinberg relation is imposed when both $a$ and $1-a$ are units. The associated small-Zariski sheaf in degree $n$ is denoted $\mathcal K_n^{\mathrm{MW}}$. For local rings with infinite residue field, Schlichting's equivalent presentation in degrees at least two is identified with this one by the map fixing the degree-one symbols \cite[Theorem~4.18]{SchlichtingEuler}.
\end{definition}

An alternating pairing on a finite projective $B$-module $P$ satisfies
$\langle v,v\rangle=0$ for every $v\in P$. It is perfect if the
induced map $P\to P^*$ is an isomorphism. In characteristic two,
skew-symmetry alone does not imply the alternating condition.
Over a ring $B$, we write $\Hyp$ for the free hyperbolic alternating plane.
The form on $\Hyp^2$ is the block sum of two matrices
$\left(\begin{smallmatrix}0&1\\-1&0\end{smallmatrix}\right)$.
We order the standard coordinate vectors by hyperbolic pairs:
$(e_1,e_2)$ is the first pair and $(e_3,e_4)$ the second.
Grouping the first vectors before the second ones gives the usual
matrix $\left(\begin{smallmatrix}0&I_2\\-I_2&0\end{smallmatrix}\right)$;
see \cite[Chapter~XV, \S8]{LangAlgebra2002}.
A unimodular column $u$ in this space has a partner $v$ with $\langle u,v\rangle=1$. It determines the oriented rank-two alternating projective
\[
 P_u=u^\perp/Bu,\qquad P_u\perp \Hyp\simeq \Hyp^2.
\]
A symplectic basis of the complement, together with $u$ and its partner, completes $u$ to a symplectic matrix. Thus symplectic completion of $u$ is equivalent to freeness of the underlying module $P_u$.

\input{frontmatter/notation}

\section{A symplectic completion criterion}\label{bridge:section}

We first separate the passage from cohomology to vector bundles from the
arithmetic vanishing results. The criterion below applies to a rank-two
bundle together with a symplectic trivialization after adjoining one
hyperbolic plane. Its proof compares the image of the completion obstruction
at the generic point and then uses the dimension of the base to remove the
remaining sheaf-theoretic errors.

\begin{theorem}[Symplectic completion from the generic image]\label{bridge:criterion}
Let $B$ be an integral Noetherian ring with $X=\Spec B$ of dimension at most three. Suppose its fraction field has characteristic zero, every residue field of $B$ is infinite, and
\[
 H^3_{\Zar}(X,\mathcal K_4^{\mathrm{MW}})=0.
\]
Then every unimodular column of length four over $B$ is the first column of a symplectic matrix. More generally, every oriented rank-two projective $P$ with a specified isometry $P\perp \Hyp\simeq \Hyp^2$ is free with its orientation.
\end{theorem}

We use integral coefficients. The isometry in the theorem fixes a point
in the homotopy fibre of stabilization. We shall follow this point through
the comparison maps, so that triviality in the final classifying space
implies freeness of the given bundle.

\subsection{Top cohomology of a generically zero sheaf}

The comparison at the generic point will leave two sheaves with zero
generic stalk. The next lemma is the reason that these sheaves do not
contribute to the cohomological degree used in the proof.

\begin{lemma}[Top cohomology and codimension of support]\label{bridge:support}
Let $X$ be an integral Noetherian scheme of dimension at most $d$, with generic point $\xi$, and let $\mathcal E$ be an arbitrary abelian sheaf on its small Zariski site. If $\mathcal E_\xi=0$, then $H^d(X,\mathcal E)=0$ for $d>0$. More generally, if all stalks of codimension less than $c$ vanish, then $H^i(X,\mathcal E)=0$ for $i>d-c$.
\end{lemma}
\begin{proof}
Consider a local section $s\in\mathcal E(U)$ on a nonempty open $U$. Its nonzero-germ locus $S$ is closed in $U$. If its generic germ is zero, it vanishes on a nonempty open $V\subset U$. The closure $Z$ of $S$ in $X$ lies in $X\setminus V$, so is proper and has dimension at most $d-1$.

Writing $j:U\hookrightarrow X$, the section defines a morphism $j_!\Z_U\to\mathcal E$. The image of this morphism has zero stalk outside $S$, hence outside $Z$. It is therefore the closed pushforward of a sheaf on $Z$. Finite sums of these image sheaves have support on a finite union of proper closed subsets, still of dimension at most $d-1$. Such finite sums form a filtered system with union $\mathcal E$, since every germ is represented by one of the sections used.

Closed pushforward is exact and preserves cohomology. Grothendieck
vanishing for arbitrary abelian sheaves on a Noetherian topological
space bounds that cohomology by the dimension of its support.
Finally, cohomology commutes with filtered colimits on this site:
its opens and finite intersections are quasi-compact, and every cover
has a finite refinement. The vanishing, colimit, and closed-pushforward
facts are given in \cite[Chapter~III, Theorem~2.7, Proposition~2.9,
and Lemma~2.10]{Hartshorne1977}; see also
\cite[Tags~02UU, 0739]{Stacks} for the site-theoretic statements.
They prove the first assertion.

For the stronger assertion, the generic points of the closure of a section's nonzero-germ locus lie in its original open and in that locus. They have codimension at least $c$. Concatenating chains of specializations bounds the dimension of their closures by $d-c$. The same finite-section argument proves the stated cohomological bound.
\end{proof}

The proof allows the closed support to vary with the section. This is
useful for the image sheaves below, which need not be constructible.
We shall also use right exactness in top degree: on a space of cohomological
dimension at most three, an epimorphism of abelian sheaves induces an
epimorphism on $H^3$. This follows from the cohomology sequence of its
kernel, whose fourth cohomology vanishes.

\subsection{Plus constructions and the bundle-classifying space}

We need stabilization and forgetful maps that are compatible before
passing to homotopy groups. We therefore fix functorial plus constructions.
If $G$ is one of the matrix groups in use, let $\widetilde E_G(C)$ be the
largest perfect subgroup of
\[
 \ker\bigl(G(C)\longrightarrow\pi_0G(C[\Delta^\bullet])\bigr).
\]
Use the functorial plus construction relative to this subgroup. For $\GL_n$ this is Schlichting's $\widetilde E_n$ model \cite[Section~6, especially Lemmas~6.4--6.5]{SchlichtingEuler}. Compatible symplectic models and the inclusion $\SL_2=\Sp_2\to\GL_2$ give compatible maps. The image of a perfect polynomial-path subgroup lies in the corresponding subgroup of the target; thus these maps are functorial. Write, in the small Zariski homotopy category,
\[
 \mathcal B_2=\Bclass\Sp_2^+,
 \quad \mathcal B_4=\Bclass\Sp_4^+,
 \quad \mathcal F=\hofib(\mathcal B_2\to\mathcal B_4),
 \quad \mathcal N=\pi_3\mathcal F.
\]
At an ordinary local stalk $T$ with infinite residue field, symplectic elimination gives $\Sp_{2n}(T)=\ESp_{2n}(T)$. These groups are perfect. For $\Sp_2$, choose a unit $a$ with $a^2-1$ a unit and express root elements as commutators with $\diag(a,a^{-1})$. For $\Sp_4$, choose a torus element whose root characters are all different from one in the residue field; an infinite field supplies such an element. The same commutator calculation applies to every root. The commutator expressions involve finitely many entries and unit denominators, so they descend to a neighborhood and identify the stalks of the chosen plus models.

Consequently the stalks are the usual simply connected plus spaces. Schlichting's integral stability theorem gives
\[
 H_i(\Sp_4(T),\Sp_2(T);\Z)=0\quad(i<4).
\]
Relative Hurewicz identifies their first relative homotopy group with relative $H_4$; equivalently $\mathcal F$ is locally two-connected. Moreover
\begin{equation}\label{bridge:local_coefficients}
 \pi_2\mathcal B_2=\mathcal K_2^{\mathrm{MW}},
 \qquad
 \mathcal N_T=H_4(\Sp_4(T),\Sp_2(T);\Z).
\end{equation}
The first identification is the natural degree-two homology comparison
of \cite[Theorems~4.18 and~5.27]{SchlichtingEuler}; the relative homology
vanishing is \cite[Theorem~7.1]{SchlichtingSpStability}. Thus stabilization
has no obstruction below degree three. It is the image of this first
obstruction in the bundle-classifying space that we shall compute.

The block product induces the multiplication epimorphism
\begin{equation}\label{bridge:mu}
 \mult:\operatorname{Sym}^2_{\Z}\mathcal K_2^{\mathrm{MW}}
       \longrightarrow\mathcal K_4^{\mathrm{MW}}.
\end{equation}
Indeed the composite
\[
 H_2(\Sp_2)\otimes H_2(\Sp_2)\longrightarrow H_4(\Sp_4)
 \longrightarrow H_4(\Sp_4,\Sp_2)
 \longrightarrow H_4(\SL_4,\SL_3)
       =K_4^{\mathrm{MW}}
\]
is exactly multiplication, and is onto on the ordinary local stalks
\cite[Theorem~7.2 and its proof]{SchlichtingSpStability}.
Interchanging the two blocks is an inner symplectic automorphism; the cross-product sign is $(-1)^{2\cdot2}=1$. The map therefore factors through the indicated symmetric square and remains onto.

Set
\begin{equation}\label{bridge:Y}
 \mathcal Y=R_{\Zar}\Sing \Bclass\GL_2,
 \qquad
 c:\mathcal B_2\longrightarrow \Bclass\GL_2^+
                    \longrightarrow\mathcal Y.
\end{equation}
The second arrow is the comparison in Schlichting's Section~6. For any affine Noetherian $X$, the composite
\begin{equation}\label{bridge:bundle_detection}
 [X,\Bclass\GL_2]_{\Zar}\longrightarrow[X,\Bclass\GL_2^+]_{\Zar}
       \longrightarrow[X,\Sing \Bclass\GL_2]_{\Zar}
\end{equation}
is bijective \cite[Corollary~6.16]{SchlichtingEuler}. Thus the last target detects the underlying bundle. Schlichting constructs the comparison using categories of extended polynomial bundles and Quillen patching.

\subsection{The comparison at the generic point}

Let $F=\Frac(B)$. Over the generic field we use the spaces
\[
 Y_2=\Bclass\SL_2(F[\Delta^\bullet]),\qquad
 Y_4=\Bclass\Sp_4(F[\Delta^\bullet]).
\]
They are simply connected, since their groups of vertices are generated by elementary matrices, which admit polynomial paths to the identity. Naturality gives a map of pairs from the two plus spaces into $(Y_4,Y_2)$. Write $\rho_2$ for the map $\Bclass\Sp_2(F)^+\to Y_2$ induced by this map of pairs. The restriction of \eqref{bridge:Y} at $F$ is $\rho_2$ followed by $Y_2\to \Bclass\GL_2(F[\Delta^\bullet])$.

Hutchinson--Wendt give the fibre sequence and its first coefficient
\cite[Propositions~3.1, 3.2 and~3.5]{BridgeHW}:
\begin{equation}\label{bridge:field_fibre}
 \Sing(\A^4\setminus0)(F)\longrightarrow Y_2\longrightarrow Y_4,
 \qquad
 \pi_3\Sing(\A^4\setminus0)(F)=K_4^{\mathrm{MW}}(F),
\end{equation}
where the left-hand space is two-connected. Relative Hurewicz and the relative homotopy sequence therefore give
\begin{equation}\label{bridge:field_kernel}
 \pi_4(Y_4,Y_2)=H_4(Y_4,Y_2;\Z)=K_4^{\mathrm{MW}}(F),
 \qquad
 \ker(\pi_3Y_2\to\pi_3Y_4)=\im\partial.
\end{equation}
Here $\partial:K_4^{\mathrm{MW}}(F)\to\pi_3Y_2$ is the relative boundary
of \eqref{bridge:field_fibre}. To compare this boundary with the local
multiplication map, we use the top cell of a product of two spheres.
The calculation will determine the coefficient of the comparison as well
as its image.

For $a,b\in\pi_2\Bclass\Sp_2(F)^+$, the block-sum map supplies the diagram of pairs
\[
 \begin{array}{ccc}
 (S^2\times S^2,S^2\vee S^2)&\longrightarrow&(\Bclass\Sp_4(F)^+,\Bclass\Sp_2(F)^+)\\
 &&\downarrow\\[-2pt]
 &&(Y_4,Y_2).
 \end{array}
\]
The two axes are the two block inclusions; inner conjugacy identifies both with the chosen stabilization. The relative top cell has boundary $[a,b]$. Its relative Hurewicz class is the image of the absolute block cross product, with coefficient one. The comparison of \cite[Proposition~3.8(2)]{BridgeHW} identifies that cross product with $ab$. Thus
\begin{equation}\label{bridge:coefficient_one}
 \pi_3(\rho_2)([a,b])=\partial\mult(a\odot b).
\end{equation}
For $a=b$, the attaching map is the Whitehead square $[a,a]=2\gamma(a)$ for the universal quadratic map, with the same relative-cell normalization as in \eqref{bridge:coefficient_one}. Reversing the orientation of the top cell changes the common sign.

We describe the normalization on the generators. Hutchinson--Tao's general-position complex has generator
\[
 d(e_1,\ldots,e_n,\textstyle\sum_i a_i e_i),
\]
with alternating deletion signs. Its map $\epsilon_n$ is the iterated connecting map for the truncated complex, and its symbol map sends this generator to $[a_1]\cdots[a_n]$ \cite[Sections~3.1, 3.3 and Theorem~3.9]{BridgeHT}. Schlichting uses the identical complex and signs, the same incoming-edge map, and the same symbol formula \cite[Section~5.1, Proposition~5.7 and Proposition~5.21]{SchlichtingEuler}. The map is multiplicative for block frames. His Theorem~4.18 identifies the hatted presentation with the unit-symbol presentation by fixing $[a]$. This is why the local map \eqref{bridge:mu} and the field map in \cite[Section~3.3 and Proposition~3.8(2)]{BridgeHW} have the same normalization.

Since $\mult_F$ is onto, \eqref{bridge:field_kernel} and \eqref{bridge:coefficient_one} imply
\begin{equation}\label{bridge:generic_generation}
 \im\bigl(\operatorname{Sym}^2\KMW_2(F)\to\pi_3Y_2\bigr)
       =\ker(\pi_3Y_2\to\pi_3Y_4).
\end{equation}
\subsection{Vanishing of the obstruction image}

The generic calculation can now be expressed in terms of sheaves on $X$.
We distinguish the image of the relative obstruction, the part generated
by Whitehead products, and the relations imposed by Milnor--Witt
multiplication. Using the standard basepoint of $\mathcal Y$, define
\begin{align*}
 \mathcal I&=\im\bigl(\mathcal N\longrightarrow\pi_3\mathcal B_2
                                  \longrightarrow\pi_3\mathcal Y\bigr),\\
 \mathcal W&=\im\bigl(\operatorname{Sym}^2\mathcal K_2^{\mathrm{MW}}
                 \xrightarrow{[-,-]}\pi_3\mathcal B_2\longrightarrow\pi_3\mathcal Y\bigr),\\
 \mathcal R&=\im(\ker\mult\longrightarrow\mathcal W).
\end{align*}
The cross-effect brackets stabilize to zero in $\pi_3\mathcal B_4$: their stabilized wedge maps extend over the product. The relative sequence consequently gives
\begin{equation}\label{bridge:images}
 \mathcal W\subset\mathcal I,
 \qquad \mathcal K_4^{\mathrm{MW}}\twoheadrightarrow\mathcal W/\mathcal R.
\end{equation}
For the second map, both $\mathcal W$ and $\mathcal K_4^{\mathrm{MW}}$ are
quotients of the same symmetric square. Quotienting $\mathcal W$ by the
image $\mathcal R$ of the multiplication kernel makes its quotient map
factor through $\mathcal K_4^{\mathrm{MW}}$. Since the original map onto
$\mathcal W$ is surjective, the induced map onto $\mathcal W/\mathcal R$
is surjective as well.

Equation~\eqref{bridge:coefficient_one} shows $\mathcal R_F=0$.
Moreover the relative kernel maps, by the map of pairs, into
$\ker(\pi_3Y_2\to\pi_3Y_4)$. By \eqref{bridge:generic_generation} that kernel is generated by the cross-effects. Composing with $Y_2\to \Bclass\GL_2(F[\Delta^\bullet])$ proves
\begin{equation}\label{bridge:generic_defects}
 \mathcal R_F=0,
 \qquad (\mathcal I/\mathcal W)_F=0.
\end{equation}
Kernels and images commute with ordinary stalks. At the generic stalk, filtered localization commutes with the polynomial simplices and their homotopy groups, and Zariski fibrant replacement preserves the stalkwise homotopy type. Thus these are the generic stalks of the sheaves just defined.

Apply Lemma~\ref{bridge:support} to the two sheaves in \eqref{bridge:generic_defects}. It gives $H^3(X,\mathcal R)=H^3(X,\mathcal I/\mathcal W)=0$. The epimorphism in \eqref{bridge:images}, the hypothesis $H^3(X,\mathcal K_4^{\mathrm{MW}})=0$, and right exactness in top degree give $H^3(X,\mathcal W/\mathcal R)=0$. The two short exact image sequences now give
\begin{equation}\label{bridge:image_zero}
 H^3(X,\mathcal W)=H^3(X,\mathcal I)=0,
 \qquad H^3(X,\mathcal N)\longrightarrow H^3(X,\pi_3\mathcal Y)
                   \text{ is zero}.
\end{equation}
The map from the relative homotopy sheaf factors through $\mathcal I$.
Its map on third cohomology therefore factors through the zero group in
\eqref{bridge:image_zero}. This is the vanishing that will be used for
the framed bundle.

\subsection{The class of the framed bundle}

A specified isometry $P\perp \Hyp\simeq \Hyp^2$ makes the $\Sp_2$-torsor of oriented bases of $P$ into a framed object over the standard $\Sp_4$-torsor. After the compatible plus maps it therefore supplies an element of $\pi_0R\Gamma(X,\mathcal F)$, whose image in $\pi_0R\Gamma(X,\mathcal B_2)$ is the bundle class.

A Noetherian Zariski space of dimension at most three has homotopy dimension at most three and is Postnikov complete \cite[Theorem~3.12]{ClausenMathew}. Since $\mathcal F$ is locally two-connected, its third truncation is $K(\mathcal N,3)$ and
\[
 \pi_0R\Gamma(X,\mathcal F)=H^3(X,\mathcal N).
\]
Consider the composite $f:\mathcal F\to\mathcal B_2\to\mathcal Y$.
By truncation adjunction, the pointed mapping space from the two-connected
source to $\tau_{\leq2}\mathcal Y$ is contractible. Consequently the
composite of $\tau_{\leq3}\mathcal F\to\tau_{\leq3}\mathcal Y$ with the
second truncation has a coherent pointed nullhomotopy. The map therefore
factors through the fibre over the fixed standard basepoint:
\[
 K(\pi_3\mathcal Y,3)\longrightarrow\tau_{\leq3}\mathcal Y
                            \longrightarrow\tau_{\leq2}\mathcal Y.
\]
The induced map between Eilenberg--Mac Lane objects is precisely
$K(\pi_3f,3)$. Thus the map on section components factors as
\[
 H^3(X,\mathcal N)\longrightarrow H^3(X,\pi_3\mathcal Y)
              \longrightarrow\pi_0R\Gamma(X,\tau_{\leq3}\mathcal Y),
\]
and its first arrow is zero by \eqref{bridge:image_zero}.

The pullback over the fixed global basepoint of $\tau_{\leq2}\mathcal Y$ identifies the coefficient sheaf with $\pi_3\mathcal Y$. This fixes the coefficient system in the preceding factorization.

The fibre of $\mathcal Y\to\tau_{\leq3}\mathcal Y$ is locally three-connected. Homotopy dimension at most three makes all its derived-section fibres connected, so this map induces a bijection on section components. The class of $P$ is therefore the standard component in $[X,\Sing \Bclass\GL_2]_{\Zar}$. Equation~\eqref{bridge:bundle_detection} implies $P\simeq B^2$.

Finally a chosen orientation on a free rank-two module can be attained by rescaling one basis vector by a unit. For a column $u$, choose its partner $v$ and identify $P_u$ with the perpendicular complement of the plane $Bu+Bv$. Choose a basis of $P_u$ and rescale its second vector so that its alternating pairing is one. The resulting four vectors $(u,v,p_1,p_2)$ form a symplectic basis of $B^4$. Its first column is the original $u$. This proves Theorem~\ref{bridge:criterion}.

%% file: frontmatter/context.tex
Quillen's local parameter conjecture asserts that every finitely generated
projective module over $R[1/r]$ is free when $R$ is a regular local ring
and $r$ is a member of a regular system of parameters. Geometrically,
the problem concerns vector bundles on the principal open subset $D(r)$
of the local scheme $\Spec R$.
The assertion is known in all dimensions when $R$ contains a field,
by work of Bhatwadekar--Rao and Popescu; see
\cite[\S2.2.2]{CesnaviciusTorsorsSurvey}.
Gabber proved the case $\dim R\leq3$ in all characteristics, and Swan
gave an elementary proof~\cite{Swan1988}.
In mixed characteristic, known results in higher dimensions include the
case where $R$ is geometrically regular over a Dedekind subring containing
the chosen parameter~\cite[Theorem~1.3(2)]{CesnaviciusSmoothDivisor}.
This paper establishes the dimension-four assertion for arbitrary regular
local rings of mixed characteristic and arbitrary regular parameters,
without restrictions on ramification or residue characteristic.

The local parameter conjecture is related to the polynomial Bass--Quillen
problem by the dimension shift $Q_{n+1}\Rightarrow BQ_n$ in Swan's
notation~\cite[\S1]{Swan1988}. Here $Q_n$ denotes the local parameter
assertion in dimension $n$, and $BQ_n$ denotes freeness of projective
modules over a polynomial ring in one variable over a regular local
ring of dimension $n$.
The polynomial assertion is already known for regular local rings containing a field
and for unramified regular local rings of mixed characteristic; see
\cite[\S2.1.2]{CesnaviciusTorsorsSurvey}.
In dimension three, Rao proved it when the residue characteristic is
different from $2$ and $3$~\cite[Theorem~2]{Rao1988}.
Our result therefore implies the one-variable Bass--Quillen assertion
in dimension three, including the general ramified mixed-characteristic
cases with residue characteristic $2$ or $3$.

%% file: frontmatter/roadmap.tex
\subsection{Outline of the proof}\label{intro:roadmap}

The proof begins with the distinction between stable freeness and
freeness. A finite free resolution over $R$ shows that every projective
module over $A$ is stably free (Lemma~\ref{main:stable_free}). The remaining
rank-two problem is treated by alternating forms. Lemma~\ref{main:actual_frame}
constructs an isometry $P\perp\Hyp\simeq\Hyp^2$, and
Theorem~\ref{bridge:criterion} shows that its rank-two complement $P$ is
free. The two vanishings in \eqref{main:arithmetic_endpoints} supply the
hypotheses for these steps.

\paragraph{Symplectic completion.}
Under its hypotheses, Theorem~\ref{bridge:criterion} applies to an integral
affine scheme of dimension at most three. Given the first vanishing in
\eqref{main:arithmetic_endpoints} and an isometry $P\perp \Hyp\simeq \Hyp^2$,
it proves that $P$ is free. Local symplectic stability places the first
obstruction in degree three. Milnor--Witt multiplication describes its
generic image in a space that classifies the underlying vector bundle;
a support argument then makes the induced map on degree-three
cohomology zero.

\paragraph{Integral degree-two $K$-theory.}
Part~II proves generic injectivity and image purity for Quillen
$K$-theory in degree two. Integral Adams operations with support give
finite-exponent bounds for the kernel and the image-purity defect.
Divisibility arguments eliminate torsion of order prime to the residue characteristic, and
local motivic comparisons treat the residue-prime part. At the dyadic
prime, explicit Dennis--Stein symbols supply the required torsion lifts.

\paragraph{Hermitian and local cohomological calculations.}
Part~III develops the comparison between genuine symmetric and
homotopy-symmetric theories. Supported genuine lifting and the local
comparison range give $W_{\mathrm{alt}}(A)=0$ in
Theorem~\ref{gw:global_walt}. Local motivic calculations identify the
edge maps and the residue exactness for the Milnor--Witt sheaf. In
degree four, the field comparison has cokernel
$H^0_{\mot}(F,\Z(2))$, whose local extension controls the residue image.

\paragraph{Global vanishing and freeness.}
Part~IV proves the field-complex vanishing and passes to the sheaf
cohomology in Theorem~\ref{global:naive_mw4} by bounding the supports
of the augmentation kernel and cokernel and the remaining cohomology
sheaves. Stable freeness and alternating Witt vanishing then give a
rank-four symplectic frame. The completion criterion trivializes its
rank-two complement and completes the proof.

The following table summarizes the principal results and their inputs.
\begin{center}
\renewcommand{\arraystretch}{1.18}
\begin{tabularx}{\textwidth}{@{}>{\raggedright\arraybackslash}p{0.36\textwidth}>{\raggedright\arraybackslash}X>{\raggedright\arraybackslash}p{0.22\textwidth}@{}}
\toprule
Result & Principal ingredients & Location\\
\midrule
Integral degree-two generic exactness & Supported operations, finite coefficients, and torsion lifts & Theorem~\ref{k2:thm:main}\\
Genuine lifts of supported symmetric classes & Coconnective surgery and the Frobenius grade bound & Theorem~\ref{gw:supported_lift}\\
$W_{\mathrm{alt}}(A)=0$ & Local negative groups and descent for the comparison fibre & Theorem~\ref{gw:global_walt}\\
Local Milnor--Witt residue exactness & Low-degree symbol maps and motivic edges & Theorems~\ref{loc:mw4_h2}, \ref{loc:mw4_h1}\\
$H^3_{\Zar}(A,\mathcal K_4^{\mathrm{MW}})=0$ & Field-complex vanishing and support bounds & Theorem~\ref{global:naive_mw4}\\
Freeness of the rank-two complement & Symplectic framing and the image criterion & Theorem~\ref{bridge:criterion}; Lemma~\ref{main:actual_frame}\\
\bottomrule
\end{tabularx}
\end{center}

Section~\ref{main:conclusion} assembles the argument. The completion
criterion of Section~\ref{bridge:section} is independent of the calculations
in Parts~II and~III: it uses their output only through the first vanishing
in \eqref{main:arithmetic_endpoints}. The alternating Witt vanishing is
used to construct the frame to which that criterion applies.

%% file: frontmatter/notation.tex
\subsection{Notation and conventions}\label{intro:notation}

In the main theorem, $R$ is the four-dimensional regular local ring,
$r$ is the chosen parameter, and $A=R[1/r]$. We write $S=R/(r)$ when
working on the divisor and $X=\Spec A$ when working on the parameter
open. Auxiliary local rings are denoted $T$, with fraction field
$F=\Frac T$ and residue field $k$. For a regular scheme $Y$, the set of
codimension-$c$ points is $Y^{(c)}$; for a ring $T$, the abbreviation
$T^{(c)}$ means $(\Spec T)^{(c)}$. The generic point is denoted $\xi$.

We follow the unit-symbol notation for Milnor and Milnor--Witt theory
in \cite{SchlichtingEuler,DegliseMW}. The genuine and homotopy-symmetric
structures are distinguished as in \cite{GWCDHIII}; the square-bracket
shift in $\GW^{[j]}$ or $\KO^{[j]}$ records the shifted duality,
including its biduality sign. The conventions used throughout are as follows.

\begingroup
\renewcommand{\arraystretch}{1.08}
\begin{longtable}{@{}>{\raggedright\arraybackslash}p{0.30\textwidth}>{\raggedright\arraybackslash}p{0.65\textwidth}@{}}
\toprule
Notation & Meaning\\
\midrule
\endfirsthead
\toprule
Notation & Meaning (continued)\\
\midrule
\endhead
\bottomrule
\endfoot
$K_i(T)$ & Quillen $K$-theory, with its nonconnective extension in localization sequences.\\
$\KM_i(T)$, $\Khat_i(T)$ & Ordinary unit-symbol Milnor $K$-theory and improved Milnor $K$-theory.\\
$\KMW_i(T)$ & The degree-$i$ naive Milnor--Witt unit-symbol group of Definition~\ref{notation:mw_definition}.\\
$\mathcal K_i^{\mathrm{MW}}$ & The Zariski sheafification of $T\mapsto\KMW_i(T)$.\\
$H^i_{\mot}(T,\Z(j))$ & Motivic cohomology in cohomological degree $i$ and weight $j$. In particular the weight-two groups are written explicitly as $H^0_{\mot}(T,\Z(2))$, $H^1_{\mot}(T,\Z(2))$, and $H^2_{\mot}(T,\Z(2))$.\\
$q_{i,T}$ & The weight-two edge $K_i(T)\to H^{4-i}_{\mot}(T,\Z(2))$, for $i=2,3,4$.\\
$\GW_i^{[j]}$, $\KO_i^{[j]}$ & Genuine symmetric and homotopy-symmetric Grothendieck--Witt groups, respectively.\\
$Q^{\mathrm{gs}},Q^{\mathrm s},Q^{\mathrm q}$;
$L^{\mathrm{gs}},L^{\mathrm s},L^{\mathrm q}$ & The genuine symmetric, homotopy-symmetric, and pure quadratic structures and their $L$-theories.\\
$\Hyp(P)$, $\Hyp$ & The alternating hyperbolic space on $P\oplus P^*$ and the free alternating plane $\Hyp=\Hyp(T)$. A power $\Hyp^m$ is an $m$-fold orthogonal sum.\\
$\Hypq$, $\hsym$ & The positive-quadratic hyperbolic space used in the formation argument, and the symmetric Grothendieck--Witt class $\hsym=\langle1\rangle+\langle-1\rangle$.\\
$f$, $h^{[j]}$ & Forgetful and hyperbolic maps. In $h^{[j]}$ the superscript specifies the target duality shift.\\
$\eta$, $\eta_{\mathrm g}$ & The Milnor--Witt generator of degree $-1$ and the genuine Bott class. The symbol comparison sends $\eta$ to $-\eta_{\mathrm g}$.\\
$\mu_n$, $\mult$ & The degree-$n$ hermitian symbol comparison and the Milnor--Witt multiplication map used in the completion criterion. In \'etale cohomology, $\mu_n$ instead denotes the sheaf of $n$th roots of unity.\\
$W(k)$, $I(k)$ & The bilinear Witt ring and its fundamental ideal; $I^a(k)=W(k)$ for $a\le0$.\\
$W_{\mathrm{alt}}(A)$ & The stable Witt group of perfect alternating forms over $A$.\\
$K_2^{\mathrm{nr}}(T)$, $J_2(T)$, $P_2(T)$ & The unramified generic subgroup, generic kernel, and image-purity cokernel in degree two; see Section~\ref{k2:sec:main}.\\
$\omega_x$ & The exceptional line $\operatorname{Ext}^{c}_{\mathcal O_{Y,x}}(\kappa(x),\mathcal O_{Y,x})$ for $x\in Y^{(c)}$, with its induced duality.\\
$C_{\mathrm{MW},4}(Y)$ & The compatible weight-four field complex of Definition~\ref{loc:mw_row}; its cochain degree is codimension.\\
$C_{\mathrm W}(Y)$, $C_{I,j}(Y)$ & The full bilinear Witt row and its filtered row with degree-$c$ coefficient $I^{j-c}(\kappa(x),\omega_x)$.\\
$\Fil^c_{\mathrm{con}}$ & The codimension-of-support filtration in coniveau theory.\\
$\Fil^i_{\mot}K(T)$,
$\Fil^i_{\mathrm{lis}}K(T)$ & The motivic filtration and the filtration obtained by realizing the classical filtration on a simplicial local resolution; see Lemma~\ref{loc:lisse_window}.\\
\end{longtable}
\endgroup

For a complex of sheaves $\mathcal C^\bullet$, we write
$\mathcal H^i(\mathcal C^\bullet)$ for its cohomology sheaf and
$\mathbb H^i(X,\mathcal C^\bullet)$ for hypercohomology.
Unadorned $H^i(X,\mathcal A)$ denotes Zariski sheaf cohomology;
$H^i(C^\bullet)$ denotes the cohomology of a complex of abelian groups.
Finite-coefficient groups and integral quotients are distinguished:
$H^i_{\mot}(T,\Z/m(j))$ denotes cohomology with finite coefficients,
whereas $H^i_{\mot}(T,\Z(j))/m$ is a quotient of an integral group.
For an abelian group $M$, $M[m]=\ker(m:M\to M)$.

We use homological shifts for perfect complexes:
$H_i(E[a])=H_{i-a}(E)$. The derived category is $\mathrm D(T)$,
its derived duality is $\Ddual_T E=\operatorname{RHom}_T(E,T)$,
and $\mathrm H M$ denotes an Eilenberg--Mac Lane spectrum.
The category of spectra is $\mathsf{Sp}$, whereas $\Sp_{2n}$ denotes
the symplectic group. The classifying-space functor is written
$\Bclass$; the elementary symplectic subgroup is $\ESp_{2n}$.

For a rank-two projective $B$-module $P$, an orientation is an
isomorphism $\theta:\det P=\bigwedge^2P\xrightarrow{\sim}B$.
It defines the perfect alternating pairing
$b_\theta(v,w)=\theta(v\wedge w)$.
A frame in the completion argument is a specified isometry
$P\perp\Hyp\simeq\Hyp^2$.

%% file: sections/01_k2.tex
\section{The first two integral Gersten positions in degree two}
\label{k2:sec:main}
We study the kernel of generic restriction and the obstruction to
extending an unramified generic class. The groups defined below measure
these two defects. Integral Adams operations give finite-exponent bounds;
divisibility and torsion-lifting arguments then show that both groups vanish.

Throughout this part, $K_i$ denotes Quillen algebraic $K$-theory,
$\KM_i$ denotes the ordinary unit-symbol Milnor group, and
$\Khat_i$ denotes improved Milnor $K$-theory.
For a Noetherian regular local domain $T$ with fraction field $F$, put
\begin{align*}
 K_2^{\mathrm{nr}}(T)&=\ker\left(K_2(F)\xto{\partial}
       \bigoplus_{\mathfrak p\in(\Spec T)^{(1)}}
                  \kappa(\mathfrak p)^\times\right),\\
 J_2(T)&=\ker\bigl(K_2(T)\to K_2(F)\bigr),\qquad
 P_2(T)=\coker\bigl(K_2(T)\to K_2^{\mathrm{nr}}(T)\bigr).
\end{align*}
Thus $J_2(T)$ measures the failure of generic injectivity, while
$P_2(T)$ measures the failure of an unramified generic class to extend
over $T$. The map defining $P_2(T)$ is well-defined by localization
at height-one primes. Each component of $\partial$ is the tame
boundary for the discrete valuation ring (DVR) $T_{\mathfrak p}$.
We use the convention $\partial\{u,\pi\}=\bar u$, with $u$ a unit
and $\pi$ a uniformizer. The opposite uniform sign gives the same
kernels and images.

\begin{theorem}[The first two integral Gersten positions]\label{k2:thm:main}
For every Noetherian regular local domain $T$, the sequence
\[
 0\longrightarrow K_2(T)\longrightarrow K_2(\Frac T)
 \xto{\partial}
 \bigoplus_{\mathfrak p\in(\Spec T)^{(1)}}
              \kappa(\mathfrak p)^\times
\]
is exact at its first two nonzero terms. Equivalently,
$J_2(T)=P_2(T)=0$.
\end{theorem}

The equal-characteristic case is the Quillen--Panin Gersten theorem
\cite[Theorem A]{Panin}. We prove the mixed-characteristic $(0,p)$ case by treating generic
injectivity and image purity separately.

\subsection{Integral Adams operations with support}
\begin{proposition}[Integral Adams operations with support]\label{k2:prop:supportadams}
Let $X$ be a finite-dimensional Noetherian regular scheme. On the
codimension-$c$ field term of the Quillen coniveau spectral sequence,
the integral Adams operation $\psi^k$, for an integer $k\geq1$, is
\[
       k^c\psi^k_{\kappa(x)}.
\]
The operation preserves the exact couple and agrees with the usual
operation on its abutment. In particular, on a field $K_j$ term for
$j=0,1,2$ its scalar is $k^{c+j}$.
\end{proposition}
\begin{proof}
Here $K_m^Z(C)$ denotes the $m$-th $K$-group of the category of perfect complexes over $C$
whose restriction to $\Spec C\setminus Z$ is acyclic.
Soul\'e constructs integral operations with closed supports and their
coniveau action in \cite[Sections 4.1--4.3 and Theorem 4(i),
pp.~521--522]{Soule}. We compute their action by localizing at $x$;
d\'evissage then identifies the local-support group with the
$K$-theory of $\kappa(x)$.

Fix $x\in X^{(c)}$ and an affine neighborhood $\Spec C$ of $x$.
Let $\mathfrak p\subset C$ be the prime corresponding to $x$.
The local-support group is the filtered colimit
\[
 \colim_{s\notin\mathfrak p}K_m^{V(\mathfrak p C_s)}(C_s)
       =K_m^{\{x\}}(C_{\mathfrak p}).
\]
The ideal $\mathfrak p$ is finitely generated. The open complement
of $V(\mathfrak p)$ has a finite principal-open cover. Continuity of perfect-complex
$K$-theory, finite Zariski descent, and the description of supported
$K$-theory as a fibre prove the displayed continuity statement; see
\cite[Proposition 3.20]{TT}. The supported operations are natural for
these restrictions.

The closed immersion
$\Spec\kappa(x)\to\Spec C_{\mathfrak p}$ is regular, with trivial
normal bundle of rank $c$. Apply the integral supported
Riemann--Roch formula of \cite[Theorem 3, p.~517; Lemma 2(i),
p.~515]{Soule}, taking the base scheme to be
$\Spec C_{\mathfrak p}$ itself. The ambient scheme and its closed
point are regular schemes of finite type over that base, satisfying
the hypotheses of the theorem. Under d\'evissage,
\[
 \psi^k(i_*a)=i_*\bigl(\theta^k(N)\psi^k(a)\bigr).
\]
For a trivial rank-$c$ normal bundle, the defining product formula for
$\theta^k$ gives $\theta^k(N)=k^c$ integrally. This computes the
operation on the field term. Its scalar on field $K_0,K_1,K_2$
follows respectively from rank, the power operation on units, and
Matsumoto symbols.
\end{proof}

For $T$ of dimension $d$ we use the finite coniveau spectral sequence
\begin{equation}\label{k2:eq:coniveau}
 E_1^{c,q}=\bigoplus_{x\in(\Spec T)^{(c)}}
 K_{-c-q}(\kappa(x))\ \Longrightarrow\ K_{-c-q}(T),
 \qquad d_s:(c,q)\mapsto(c+s,q-s+1).
\end{equation}
In total degree two, the generic term is $K_2(F)$ and its first
differential is the tame map $\partial$. Thus
$E_2^{0,-2}=K_2^{\mathrm{nr}}(T)$, while the terminal edge
$E_\infty^{0,-2}$ is the image of $K_2(T)\to K_2(F)$.
The first positive filtration stage is $J_2(T)$, and $P_2(T)$ is
the quotient of the unramified edge by the terminal edge image.
The filtration has length at most $d$.

We first obtain annihilators for the two defects. The same Adams
operation acts with different weights on the generic group and on the
higher-codimension terms that can obstruct its extension. Subtracting
these scalar actions gives a finite annihilator on each graded piece;
the finite length of the coniveau filtration then gives an annihilator
for the whole defect.

\begin{lemma}[Finite-exponent bounds for the two defects]\label{k2:lem:bounds}
There are nonzero integers
\[
 N_d=\prod_{c=1}^{d}4(2^c-1),\qquad
 M_d=\prod_{s=2}^{d}4(2^{s-1}-1)
\]
such that $N_dJ_2(T)=0$ and $M_dP_2(T)=0$.
An empty product is one.
\end{lemma}
\begin{proof}
Ordinary unit symbols generate $K_2$ of every local ring:
combine the ordinary-to-improved surjection with
$\Khat_2=K_2$ in \cite[Propositions 4.2.8(3), 4.3.3]{KerzThesis}.
Multiplicativity of Adams operations therefore gives
$\psi^2=4$ on $K_2(T)$, including for small finite residue fields.

The kernel $J_2(T)$ is the first positive coniveau filtration stage.
Its codimension-$c$ graded piece is a subquotient of field $K_2$ in
that codimension. Proposition~\ref{k2:prop:supportadams} gives scalar
$2^{c+2}$ there, whereas the same operation on the whole kernel has
scalar four. Thus $4(2^c-1)$ kills this graded piece. Multiplying the
finitely many factors proves the assertion for $J_2$.

The unramified edge $E_2^{0,-2}=K_2^{\mathrm{nr}}(T)$ has scalar four and no incoming
differentials. Its only possible outgoing higher maps have targets
\[
 E_s^{s,-s-1},\qquad 2\leq s\leq d.
\]
Each target is a subquotient of a codimension-$s$ field $K_1$ term,
with scalar $2^{s+1}$. Its differential image is consequently killed
by $4(2^{s-1}-1)$. The terminal edge is precisely the image of $K_2(T)\to K_2(F)$.
The descending filtration of $K_2^{\mathrm{nr}}(T)/\im(K_2(T)\to K_2(F))$ proves the second bound.
In dimension at most one it is simply DVR image exactness.
\end{proof}

\subsection{Prime-to-residue divisibility}
Every regular Noetherian semilocal domain is factorial.
Its height-one prime ideals are locally principal, hence invertible,
because its regular localizations are factorial.
Invertible modules over a semilocal ring are free of rank one, so
these prime ideals are principal; see \cite[Exercise~4.13,
Corollary~10.6, and Theorem~19.19]{Eisenbud1995}.
The definitions of $J_2$, $K_2^{\mathrm{nr}}$, and $P_2$ apply to these
semilocal domains as well. We use this slightly broader setting because
adjoining a root of unity may split the ring into several components.

\begin{lemma}[Prime-to-residue divisibility with roots of unity]\label{k2:lem:tamedivision}
Let $C$ be a regular Noetherian semilocal domain, and let $\ell$ be a
prime invertible in $C$. If $C$ contains a primitive $\ell$-th root of
unity, both $J_2(C)$ and $P_2(C)$ are $\ell$-divisible.
\end{lemma}
\begin{proof}
The norm-residue theorem first provides division modulo classes that
extend from the ring. We then correct the resulting generic torsion or
residues, so that the divided class belongs to the same kernel or
unramified subgroup as the original class.

Write $L=\Frac C$ and fix a primitive root $\zeta$.
Hoobler's natural norm-residue isomorphism
\cite[Theorem 3]{Hoobler} is
\[
 K_2(D)/\ell\cong H^2_{\et}(D,\mu_\ell^{\otimes2})
\]
for each semilocal ring $D$ in which $\ell$ is invertible.
The root untwists $\mu_\ell^{\otimes2}$ to $\mu_\ell$.
Kummer theory, $\Pic(C)=0$, and Brauer generic injectivity
\cite[Theorem 1.2]{Cesnavicius} therefore show that
$K_2(C)/\ell\to K_2(L)/\ell$ is injective.

Take $a\in J_2(C)$ and write $a=\ell b$ using that injection.
Then $b_L\in K_2(L)[\ell]$. The field torsion theorem
\cite[Theorem 14.2, pp.~331--332]{MerkurjevSuslin} expresses this as
$\{\zeta,u\}$. Sums of such symbols combine into one because
$\mu_\ell$ is cyclic. The class is integrally unramified. Its tame
residue at a height-one prime is $\bar\zeta^{v(u)}$, so every
valuation of $u$ is divisible by $\ell$. Factoriality gives
$u=\epsilon h^\ell$, with $\epsilon\in C^\times$.
The symbol $\{\zeta,\epsilon\}$ is defined over $C$ and is killed by
$\ell$ there. Subtracting it from $b$ gives $b_0\in J_2(C)$ with
$a=\ell b_0$. Its generic restriction is zero, while multiplication by
$\ell$ is unchanged. We have therefore divided $a$ within $J_2(C)$.

For $x\in K_2^{\mathrm{nr}}(C)$, DVR localization lifts $x$ over every height-one
local ring. Hoobler's theorem and the root untwisting identify its
mod-$\ell$ class with a Brauer class extending at each height-one
point. Brauer purity extends it over $C$; generic injection makes its
extension $\ell$-torsion. Kummer theory and $\Pic(C)=0$, followed by
Hoobler, supply $y\in K_2(C)$ with
$x-y_L=\ell z$ for some $z\in K_2(L)$.
The residues of $z$ are $\ell$-th roots of unity, with finite support.
Choose for each nontrivial component the corresponding global root
$\zeta_{\mathfrak p}$ and a generator $\pi_{\mathfrak p}$ of its prime.
The $\ell$-torsion symbol
$\{\zeta_{\mathfrak p},\pi_{\mathfrak p}\}$ has that residue and no
others. Subtract their sum from $z$. The corrected $z_0$ is
integrally unramified and $\ell z_0=\ell z$.
Thus $[x]=\ell[z_0]$ in $P_2(C)$.
\end{proof}

\begin{corollary}[Removal of prime-to-residue torsion]\label{k2:cor:primeaway}
If $T$ has mixed characteristic $(0,p)$ and $\ell$ is a prime
different from $p$, then
$J_2(T)=\ell J_2(T)$ and $P_2(T)=\ell P_2(T)$.
Consequently each of these groups is killed by a power of $p$.
\end{corollary}
\begin{proof}
The algebra $T'=T[z]/\Phi_\ell(z)$ is finite \'etale and free of rank
$\ell-1$. For $\ell=2$ it is $T$ itself. Its regular semilocal domain
components contain the required primitive root, which remains
primitive in every residue field. Apply
Lemma~\ref{k2:lem:tamedivision} componentwise.

Restriction and Quillen transfer preserve both generic kernels and
integral unramified subgroups, the latter by their compatibility with
DVR residues and norms. They also preserve the images of the ring
$K_2$ groups, so they act on the purity defects. Their composite is
multiplication by $\ell-1$ because $T'$ is free of that rank.
For $x$ in either group write $\res(x)=\ell y$ and put $z=\cor(y)$.
Then $(\ell-1)x=\ell z$ and $x=\ell(x-z)$, proving divisibility.

Factor either finite annihilator of Lemma~\ref{k2:lem:bounds} as
$p^a m$ with $(p,m)=1$. Repeated divisibility gives $G=mG$ for the
relevant group $G$, and $p^a mG=0$ then implies $p^aG=0$.
\end{proof}

\section{The ordinary local motivic edge in weight two}
\label{k2:sec:motivic}
We use Bouis's mixed-characteristic motivic complexes. Their weight-two
cohomology in degree two is denoted $H^2_{\mot}(C,\Z(2))$.

\subsection{The presentation coequalizer}
We use ordinary Milnor groups defined by unit symbols, multilinearity, and Steinberg relations. We denote the ordinary and improved Milnor groups by $\KM_n$ and $\Khat_n$, respectively.

The local resolution used to construct motivic cohomology gives a
coequalizer in the degree of interest. To compare it with unit symbols,
we must know that a symbol lifted to the resolution is independent of
its chosen unit lifts. The following lemma proves this directly from
the simplicial identities.

\begin{lemma}[The ordinary symbol coequalizer]\label{k2:lem:symbolcoequalizer}
Let $P_\bullet\to T$ be a simplicial resolution of a local ring by local rings. Suppose its augmentation and its face maps in degrees zero and one are local surjections. Then for every $n\geq1$,
\[
 \operatorname{coeq}\bigl(\KM_n(P_1)\rightrightarrows\KM_n(P_0)\bigr)
 \cong\KM_n(T).
\]
\end{lemma}
\begin{proof}
Let $C$ denote the coequalizer. Unit symbols lift along $P_0\to T$, so $C\to\KM_n(T)$ is onto. We construct its inverse by lifting the units of a symbol. If $u,v\in P_0^\times$ have the same image, then $u-v$ is in the image of $d_0-d_1$ on the additive group of $P_1$: this is the formula for $\pi_0$ of the underlying simplicial abelian group. Choose $z$ with $d_0z-d_1z=u-v$ and put
\[
 w=z+s_0(v-d_1z).
\]
Then $d_0w=u$ and $d_1w=v$. Since $d_0$ is local, $w$ is a unit. A symbol with $w$ in the chosen coordinate and degeneracies of the other lifted units in every other coordinate identifies the two lifts in $C$. This proves independence in every coordinate.

For a product of units choose the product of their lifts; this gives multilinearity. If $a$ and $1-a$ are units in $T$, any lift $\widetilde a$ has both $\widetilde a$ and $1-\widetilde a$ invertible, since the augmentation is local. The Steinberg relation therefore already holds upstairs. This construction is well-defined and gives the inverse of the natural map.
\end{proof}

\begin{theorem}[The integral weight-two edge for local rings]\label{k2:thm:alllocalq2}
For every local ring $T$, the natural multiplicative weight-two edge induces an isomorphism
\[
 q_{2,T}:K_2(T)\xrightarrow{\sim}H^2_{\mot}(T,\Z(2)).
\]
It agrees with the unit-symbol map and is natural in ring homomorphisms between local rings.
\end{theorem}
\begin{proof}
We first identify the edge when the residue field is large enough for
ordinary and improved Milnor theory to agree. The coequalizer then gives
surjectivity for every local ring. For a small residue field, injectivity
will descend from two finite extensions of coprime degrees.

\proofstep{1. The local weight-two edge.}
The local resolution in \cite[proof of Lemma 2.19]{BouisWeibel} is a simplicial resolution $P_\bullet\to T$ by local ind-smooth $\Z$-algebras, with henselian-surjective face maps. Its augmentation to $\pi_0(P_\bullet)=T$ is surjective and hence local. Every term has the same residue field as $T$. The local left-Kan-extension theorem and the displayed coequalizer there give
\begin{equation}\label{k2:eq:motcoequalizer}
 H^2_{\mot}(T,\Z(2))=
 \operatorname{coeq}\bigl(H^2_{\mot}(P_1,\Z(2))
 \rightrightarrows H^2_{\mot}(P_0,\Z(2))\bigr).
\end{equation}

On a local essentially smooth $\Z$-algebra the weight-two comparison is $H^2_{\mot}=K_2$ by \cite[Lemma 2.16]{BouisWeibel}; continuity gives the same on the ind-smooth terms. Its identification on products of weight-one unit classes is the ordinary $K_2$ symbol, by multiplicativity of the filtration. Kerz identifies improved Milnor $K_2$ with Quillen $K_2$ on every local ring and identifies ordinary with improved Milnor $K_2$ once the residue field is sufficiently large \cite[Proposition 4.2.8(3),(5)]{KerzThesis}. Hence for sufficiently large residue fields, \eqref{k2:eq:motcoequalizer} is the symbol coequalizer of Lemma~\ref{k2:lem:symbolcoequalizer}, proving the assertion.

\proofstep{2. Surjectivity of the local edge.}
For an arbitrary residue field, surjectivity follows from the coequalizer: the map from the $P_0$ term in \eqref{k2:eq:motcoequalizer} is onto, and its natural identification with $K_2(P_0)$ factors through $q_{2,T}$. It remains to prove injectivity for a small finite residue field $\F_{p^f}$.

\proofstep{3. Injectivity by coprime extensions.}
Choose two sufficiently large coprime positive integers $a,b$, each coprime to $f$. A monic lift to $\Z_{(p)}$ of an irreducible polynomial of degree $a$ over $\F_p$ gives a finite free local \'etale algebra $V_a$ with residue field $\F_{p^a}$. Its discriminant is a unit, so it is \'etale. Put
\[
 T_a=T\otimes_{\Z_{(p)}}V_a.
\]
Its closed residue algebra is $\F_{p^f}\otimes_{\F_p}\F_{p^a}=\F_{p^{fa}}$, a field. Every maximal ideal of this finite algebra lies over the maximal ideal of $T$, so $T_a$ is local. It is free of rank $a$, with large residue field. Construct $T_b$ similarly.

If $x\in\ker q_{2,T}$, its restriction to either extension is zero by naturality and the proved large-residue injectivity. Quillen transfer composed with restriction is multiplication by $a$, respectively $b$: the extension module is free of that rank, and the projection formula applies. Thus $ax=bx=0$, whence $x=0$. This proves the theorem for all local rings.
\end{proof}

\subsection{Finite coefficients and arithmetic constants}
For the remainder of this part, let $T$ be a Noetherian regular local
domain of mixed characteristic $(0,p)$, with fraction field $F$,
and put $Y=\Spec T[1/p]$.

\begin{lemma}[The finite-coefficient comparison]\label{k2:lem:finitecomparison}
The comparison gives natural isomorphisms
\[
 H^a_{\mot}(T,\Z/p(2))\cong H^a_{\et}(Y,\mu_p^{\otimes2}),\qquad a=0,1,
\]
and an injection in degree two. Generic restriction on the degree-one
group is injective.
\end{lemma}
\begin{proof}
A characteristic-zero domain is $\Z$-flat, and a regular Noetherian
ring is $F$-smooth by \cite[Theorem 4.15]{BhattMathew}.
The fibre in \cite[Corollary 5.12]{BouisMotivic}, applied in weight
two, has cohomology only in degrees at least three. Its long exact
sequence gives exactly these comparisons. A torsor under the finite \'etale group sheaf $\mu_p^{\otimes2}$ on the normal
integral scheme $Y$ which is generically trivial has a generic
section. The component selected by that section is finite birational;
normality makes it the base itself. The torsor is therefore trivial,
which proves generic injectivity on $H^1$.
\end{proof}

We will also use the natural weight-two edge
\begin{equation}\label{k2:eq:k3edge}
 q_{3,T}:K_3(T)\longrightarrow H^1_{\mot}(T,\Z(2)).
\end{equation}
The first two filtration triangles identify
$\pi_3\Fil^2_{\mot}K(T)$ with $K_3(T)$. Indeed, the weight-zero and
weight-one graded layers have vanishing $\pi_3$ and $\pi_4$.
For weight zero, this follows from \cite[Example~3.8]{BouisMotivic} and
\cite[Corollary~2.12]{BouisWeibel}; for weight one, it follows from
\cite[Example~2.15]{BouisWeibel}. The long exact sequences therefore
make the maps from the second filtration stage isomorphisms in the
degree under consideration.
Projection to the weight-two graded layer gives
\eqref{k2:eq:k3edge}. Multiplicativity and naturality of the filtration
identify its field version with the quotient
$K_3(F)\to K_3^{\ind}(F)=H^1_{\mot}(F,\Z(2))$.

\begin{lemma}[Lifting the arithmetic constants]\label{k2:lem:constants}
Restriction is surjective in
\[
 H^1_{\mot}(T,\Z(2))/p\longrightarrow
 H^1_{\mot}(F,\Z(2))/p.
\]
\end{lemma}
\begin{proof}
Represent a given field class over a finitely generated
characteristic-zero subfield $F_0\subset F$. Its algebraic constant
field $K$, the algebraic closure of $\Q$ in $F_0$, is a number field.
The finite-level clause of \cite[Theorem 4.13, p.~336]{Levine} gives
\[
 K_3^{\ind}(K)/p\xrightarrow{\sim}K_3^{\ind}(F_0)/p,
\]
including for $p=2$. Thus the class modulo $p$ comes from $K_3(K)$.

Every algebraic integer of $K$ is integral over $T$, because its monic
polynomial over $\Z$ is also over $T$. Normality therefore puts
$\mathcal O_K$ inside $T$. Contracting the maximal ideal gives a prime
$\mathfrak p_K$ over $p$, and locality extends the map to
$\mathcal O_{K,\mathfrak p_K}\subset T$.
DVR localization gives a surjection
\[
 K_3(\mathcal O_{K,\mathfrak p_K})\longrightarrow K_3(K),
\]
since its next term is $K_2$ of a finite field, zero by
\cite{QuillenFinite}. The resulting class of $K_3(T)$ maps under
\eqref{k2:eq:k3edge} to the required finite-level motivic class.
\end{proof}

The preceding two lemmas complement one another. A class killed by
$p$ in the integral generic kernel comes from
$H^1_{\mot}(T,\Z/p(2))$ by the Bockstein. The generic restriction
of that degree-one class can be removed using a lifted constant
class, after which the degree-one generic injectivity in
Lemma~\ref{k2:lem:finitecomparison} applies.

\begin{proposition}[Absence of residue-prime torsion in the kernel]\label{k2:prop:noptorsion}
The kernel of $H^2_{\mot}(T,\Z(2))\to H^2_{\mot}(F,\Z(2))$ has no $p$-torsion.
\end{proposition}
\begin{proof}
If $z$ belongs to that kernel and $pz=0$, the integral coefficient
sequence supplies $a\in H^1_{\mot}(T,\Z/p(2))$ with $\beta(a)=z$.
The generic class $a_F$ has zero Bockstein, so it is the reduction of
a class in $H^1_{\mot}(F,\Z(2))/p$. Lift the latter by
Lemma~\ref{k2:lem:constants} and subtract its reduction from $a$.
The Bockstein remains $z$, while the new class is generically zero.
Lemma~\ref{k2:lem:finitecomparison} makes that new class zero. Thus $z=0$, proving that the integral generic kernel has trivial
$p$-torsion.
\end{proof}

\begin{theorem}[Integral generic injectivity]\label{k2:thm:genericinjection}
For every Noetherian regular local domain of mixed characteristic,
$K_2(T)\to K_2(F)$ is injective.
\end{theorem}
\begin{proof}
By Corollary~\ref{k2:cor:primeaway}, $J_2(T)$ is killed by a power of
the residue prime $p$. The natural isomorphism $q_{2,T}$ of
Theorem~\ref{k2:thm:alllocalq2}, together with the field isomorphism
$q_{2,F}$, identifies this kernel with the corresponding motivic kernel.
Proposition~\ref{k2:prop:noptorsion} rules out $p$-torsion in it.
Successive multiplication by $p$ therefore kills no nonzero element,
and a group of bounded $p$-power exponent with this property is zero.
\end{proof}

\section{Image purity at the residue prime}
\label{k2:sec:image}
Generic injectivity settles $J_2(T)$; extension of unramified generic
classes requires the separate assertion $P_2(T)=0$.
Corollary~\ref{k2:cor:primeaway} gives $P_2(T)$ a bounded $p$-power
exponent. Any nonzero group with this property has a nonzero element
killed by $p$, so it suffices to prove $P_2(T)[p]=0$.
For odd $p$, we combine the divided-Chern calculation with unramified
torsion lifting. For $p=2$, the quotient comparison and the
Dennis--Stein lifts provide the corresponding steps.

\subsection{Root-free lifting of unramified torsion for odd primes}
\begin{lemma}[Unramified torsion at an odd residue prime]\label{k2:lem:oddtorsion}
Assume $p$ is odd. Every element of $K_2(F)[p]\cap K_2^{\mathrm{nr}}(T)$
extends to $K_2(T)$.
\end{lemma}
\begin{proof}
\proofstep{1. The torsion class.}
For $\tau$ in the indicated group, choose
$\theta\in H^1_{\et}(F,\mu_p^{\otimes2})=H^1_{\mot}(F,\Z/p(2))$ whose integral
Bockstein is $q_{2,F}(\tau)$. Such a choice exists by the coefficient
sequence and the field comparison.
Use the finite \'etale primitive-root torsor
\[
 Y'=\Spec\bigl(T[1/p][z]/\Phi_p(z)\bigr)\longrightarrow Y.
\]
Its group is $G=(\Z/p)^\times$, its degree is $p-1$, and it may be
disconnected. Work on each regular integral component with function
field $F'$ and tautological primitive root $\zeta$. Untwisting by
$\zeta$ writes $\theta_{F'}=\zeta\cup[u]$ for a Kummer class $[u]$.
Multiplicativity of the integral coefficient triangle gives
\begin{equation}\label{k2:eq:bocksteinsymbol}
       \beta(\zeta\cup[u])=\{\zeta,u\}.
\end{equation}
Here the weight-one Bockstein sends the finite-coefficient root to the
integral unit $\zeta$, and $[u]$ lifts integrally to a unit class.
This verifies the symbol normalization.

\proofstep{2. Extension of the Kummer class.}
At each height-one point of the component, the tame boundary of the
right side is $\bar\zeta^{v(u)}$. It is trivial since $\tau$ is
integrally unramified and finite \'etale restriction preserves this
property. The root remains primitive, so $p$ divides every valuation.
Write $\Div(u)=pD$. Regularity makes $D$ a Cartier divisor.
The Kummer pair consisting of $\mathcal O(D)$ and its $p$-th-power
trivialization defined by $u$ extends $[u]$ to the component of $Y'$;
see \cite[tag 03PK]{Stacks}. A change to the inverse divisor and
trivialization gives the same assertion under the opposite Kummer
convention.

\proofstep{3. Descent of the extended class.}
Generic restriction on this Kummer $H^1$ is injective, either by the
normal torsor argument or by the divisor description. Thus the
extension $\theta'$ of $\theta_{F'}$ is unique on all of $Y'$, and it is
$G$-invariant because its generic class comes from $F$. Group
cohomology in positive degree with $\F_p$-coefficients vanishes for
the group of order $p-1$. Hochschild--Serre therefore identifies
\[
 H^1_{\et}(Y,\mu_p^{\otimes2})=H^1_{\et}(Y',\mu_p^{\otimes2})^G.
\]
The descended class restricts to $\theta$: restriction to the generic
primitive-root torsor is itself injective by the same argument.
Lemma~\ref{k2:lem:finitecomparison} lifts it to
$\theta_T\in H^1_{\mot}(T,\Z/p(2))$.
Its integral Bockstein extends $q_{2,F}(\tau)$ in $H^2_{\mot}(T,\Z(2))$.
Surjectivity of $q_{2,T}$ lifts that Bockstein to $K_2(T)$, and the field
isomorphism $q_{2,F}$ shows that its image is exactly $\tau$.
The possible $\Pic(Y')[p]$ contribution was retained throughout.
\end{proof}

\subsection{A divided Chern class with the Picard term retained}
\begin{lemma}[A divided Chern class retaining the Picard term]\label{k2:lem:dividedchern}
Suppose $p$ is odd, $\alpha\in K_2^{\mathrm{nr}}(T)$, and $b\in K_2(T)$ satisfy
$b_F=p\alpha$. Then the coefficient reduction of $q_{2,T}(b)$ is zero in
$H^2_{\mot}(T,\Z/p(2))$.
\end{lemma}
\begin{proof}
We shall test the coefficient class after adjoining a primitive root
of unity. On that cover, the Chern class of the difference between the
ring class and the divided generic class is controlled by Adams
operations. Removing supports of codimension at least two makes this
calculation accessible; Picard invariance and Brauer injectivity then
bring the vanishing back to the whole cover.

\proofstep{1. The primitive-root cover.}
Consider the same degree-$(p-1)$ torsor $Y'\to Y$, and let $S$ be an
integral component, with function field $L$. Since all rational primes
other than $p$ are already units of $T$, the scheme $S$ is a regular
affine Noetherian $\Q$-scheme. The restricted class $\alpha_L$ is
integrally unramified at all its height-one points.

\proofstep{2. The normalized Chern map.}
For a quasi-compact open $V\subset S$ use the ordinary integral
higher Chern homomorphism, normalized on unit products,
\[
 c_V:K_2(V)\longrightarrow H^2_{\et}(V,\mu_p^{\otimes2}),
 \qquad c_V(\{u,v\})=[u]\cup[v].
\]
It is additive and obeys the integral identity
\begin{equation}\label{k2:eq:chernadams}
                  c_V(\psi^2x)=4c_V(x).
\end{equation}
For the product normalization see \cite[Section 2, equation (2.2),
Lemma 2.3, and Section 3]{WeibelChern}; for the integral Adams
identity and its finite \'etale coefficient examples see
\cite[Section 4, pp.~25 and 29]{Schneider}.
Schneider constructs these homomorphisms for smooth quasi-projective
schemes over a field (p.~14). The integral identity
$c_{i,j}\psi^m=m^j c_{i,j}$ on p.~25 applies to the root-of-unity
coefficient complexes listed on p.~29. A common sign change of
$c_{2,2}$ gives the displayed unit-product normalization and preserves
the Adams identity. Thus $c_V$ is a homomorphism from integral
$K$-theory to torsion cohomology.

The coordinate ring of $S$ is geometrically regular over the perfect
field $\Q$, so these maps extend to $S$ by continuity.
Popescu's theorem \cite[tag 07GC]{Stacks} expresses it as a filtered
colimit of smooth $\Q$-algebras. A quasi-compact open is a finite union
of principal opens and descends to a finite stage. Continuity of
$K$-theory \cite[Proposition 3.20]{TT} and of finite \'etale cohomology
\cite[Expos\'e VI, Theorem 8.7.3]{SGA4} carries the Chern maps,
naturality, and \eqref{k2:eq:chernadams} to these opens.

\proofstep{3. The codimension-one calculation.}
Put
\[
 C(S)=\colim_{\operatorname{codim}Z\geq2}K_2(S-Z),\qquad
 H(S)=\colim_{\operatorname{codim}Z\geq2}H^2_{\et}(S-Z,\mu_p^{\otimes2}).
\]
The colimits run over closed subsets of codimension at least two. Quillen localization gives
\begin{equation}\label{k2:eq:skeleton}
 \bigoplus_{P\in S^{(1)}}K_2(\kappa(P))\longrightarrow C(S)
 \longrightarrow K_2(L)\longrightarrow
 \bigoplus_{P\in S^{(1)}}\kappa(P)^\times .
\end{equation}
It is the exact localization couple after deeper supports are removed.
Let $J(S)=\ker(C(S)\to K_2(L))$. Exactness identifies this group with
the image of the first arrow. Proposition~\ref{k2:prop:supportadams}
therefore gives the scalar $\psi^2=2\cdot2^2=8$ on the whole kernel:
one factor comes from codimension one, and the other from field degree
two. The colimit removes all supports of codimension at least two. Passing the Chern maps to
it yields $8c(x)=c(\psi^2x)=4c(x)$ for $x\in J(S)$.
Since the target is killed by the odd prime $p$, this forces $c(x)=0$.

\proofstep{4. Vanishing of the divided Chern class.}
Lift $\alpha_L$ to $a\in C(S)$ by \eqref{k2:eq:skeleton}.
Then $b_S-pa\in J(S)$, so the image of $c_S(b_S)$ in $H(S)$ is zero.
We verify that $H^2_{\et}(S,\mu_p^{\otimes2})\to H(S)$ is injective.
The root identifies $\mu_p^{\otimes2}$ with $\mu_p$. Removing a closed subset of
codimension at least two from a regular scheme preserves Weil
divisors and principal divisors, hence its Picard group. In
\[
 0\longrightarrow\Pic(S)/p\longrightarrow H^2_{\et}(S,\mu_p)
       \longrightarrow\Br(S)[p]\longrightarrow0,
\]
a class killed on such an open has generically zero Brauer component.
Brauer generic injection kills that component already on $S$.
The remaining $\Pic(S)/p$ component is zero by Picard invariance.
Thus restriction to every such open, and therefore to the filtered
colimit, is injective. We conclude $c_S(b_S)=0$ on every component of
$Y'$.

\proofstep{5. The local motivic edge.}
Finally ordinary unit products generate $K_2(T)$, including for small
finite residue fields. Both the normalized Chern map and the
motivic-to-\'etale class of $b$ on $Y$ send these generators to the
same Kummer cup products. These computations identify the two global classes. Pullback to $Y'$ kills this class.
Trace composed with restriction is multiplication by $p-1$, invertible
on the $p$-torsion target. Thus the class is zero on $Y$.
The degree-two injection of Lemma~\ref{k2:lem:finitecomparison}
proves the asserted motivic vanishing.
\end{proof}

\begin{proposition}[Image purity at an odd residue prime]\label{k2:prop:oddimage}
For an odd residue prime $p$, one has $P_2(T)[p]=0$.
\end{proposition}
\begin{proof}
Represent a class killed by $p$ by $\alpha\in K_2^{\mathrm{nr}}(T)$ and choose
$b\in K_2(T)$ with $b_F=p\alpha$.
Lemma~\ref{k2:lem:dividedchern} and the integral coefficient sequence
give $q_{2,T}(b)=pm$ for some $m\in H^2_{\mot}(T,\Z(2))$.
Lift $m$ to $c\in K_2(T)$ using the surjectivity of $q_{2,T}$.
The field isomorphism gives $p(\alpha-c_F)=0$.
The difference remains integrally unramified, so
Lemma~\ref{k2:lem:oddtorsion} lifts it to $K_2(T)$.
Thus $\alpha\in \im(K_2(T)\to K_2(F))$.
\end{proof}

\subsection{The dyadic divided class and explicit torsion lifts}
\begin{lemma}[The dyadic quotient comparison]\label{k2:lem:dyadicquotient}
If $p=2$, the map $H^2_{\mot}(T,\Z(2))/2\to H^2_{\mot}(F,\Z(2))/2$ is injective.
\end{lemma}
\begin{proof}
Bouis's local vanishing \cite[Corollary 2.10]{BouisWeibel} gives
$H^3_{\mot}(T,\Z(2))=0$; the coefficient triangle identifies
$H^2_{\mot}(T,\Z(2))/2$ with $H^2_{\mot}(T,\Z/2(2))$.
Lemma~\ref{k2:lem:finitecomparison} injects this into
$H^2_{\et}(Y,\mu_2^{\otimes2})$.
The regular local ring $T$ is a unique factorization domain (UFD), and its localization is again a UFD, so
$\Pic(Y)=0$. The global root $-1$ untwists
$\mu_2^{\otimes2}$ to $\mu_2$ on $Y$, where two is invertible.
Kummer theory and Brauer generic injection give
\[
 H^2_{\et}(Y,\mu_2^{\otimes2})\cong\Br(Y)[2]
          \hookrightarrow\Br(F)[2]
          \cong H^2_{\et}(F,\mu_2^{\otimes2}).
\]
Naturality with the motivic coefficient sequence proves the claim.
\end{proof}

For the dyadic argument we also need to lift the torsion left after
dividing a ring class. At primes above two, the usual tame residue
condition does not force an even valuation. The Dennis--Stein symbols
in the next proof supply the required integral lifts at exactly those
primes.

\begin{lemma}[Dyadic two-torsion lifts]\label{k2:lem:dyadictorsion}
In mixed characteristic $(0,2)$ every class in
$K_2(F)[2]\cap K_2^{\mathrm{nr}}(T)$ has a lift in $K_2(T)[2]$.
\end{lemma}
\begin{proof}
The field torsion theorem writes the class as $\tau=\{-1,u\}$.
Factor in the UFD $T$:
\[
       u=\epsilon\prod_{\mathfrak p}\pi_{\mathfrak p}^{n_{\mathfrak p}},
       \qquad \epsilon\in T^\times.
\]
If $2\notin\mathfrak p$, the unramified tame value
$(-1)^{n_{\mathfrak p}}$ forces $n_{\mathfrak p}$ to be even, so that
prime contributes zero to the two-torsion symbol. At a prime dividing
two, put $a=-2/\pi_{\mathfrak p}\in T$.
The Dennis--Stein symbol
\[
        d_{\mathfrak p}=\langle-2/\pi_{\mathfrak p},
                               \pi_{\mathfrak p}\rangle
\]
is defined because $1+a\pi_{\mathfrak p}=-1$ is a unit. The
Dennis--Stein identity with invertible second entry gives, over $F$,
\[
       (d_{\mathfrak p})_F=\{-1,\pi_{\mathfrak p}\}.
\]
We use the convention and relations (D1)--(D4) in
\cite[Section 2.1]{LudersMorrowK2}; they are stated for arbitrary rings.
The Dennis--Stein relation with a fixed second entry yields
\[
 2\langle a,\pi\rangle
   =\langle 2a+a^2\pi,\pi\rangle=\langle0,\pi\rangle=0,
\]
since $a\pi=-2$. Thus
$\{-1,\epsilon\}+\sum_{\mathfrak p\ni2}n_{\mathfrak p}d_{\mathfrak p}$
is the required two-torsion lift. Negative exponents are allowed in
this abelian-group sum.
\end{proof}

\begin{proposition}[Image purity at residue characteristic two]\label{k2:prop:dyadicimage}
When the residue characteristic is two, $P_2(T)[2]=0$.
\end{proposition}
\begin{proof}
If $\alpha\in K_2^{\mathrm{nr}}(T)$ represents a class killed by two, choose
$b\in K_2(T)$ with $b_F=2\alpha$.
Lemma~\ref{k2:lem:dyadicquotient} shows that $q_{2,T}(b)=2m$ in $H^2_{\mot}(T,\Z(2))$.
Choose $c\in K_2(T)$ with $q_{2,T}(c)=m$.
The field isomorphism implies $2(\alpha-c_F)=0$.
This difference is integrally unramified; its lift supplied by
Lemma~\ref{k2:lem:dyadictorsion}, added to $c$, lifts $\alpha$.
\end{proof}

\begin{proof}[Proof of Theorem~\ref{k2:thm:main}]
For mixed characteristic, generic injectivity is
Theorem~\ref{k2:thm:genericinjection}. The purity defect has bounded
$p$-power exponent by Corollary~\ref{k2:cor:primeaway} and has no
$p$-torsion by Proposition~\ref{k2:prop:oddimage} or
\ref{k2:prop:dyadicimage}. It is therefore zero.
The equal-characteristic case follows from \cite[Theorem A]{Panin},
so all cases in the theorem are covered.
\end{proof}

\section{Low-dimensional consequences and coefficient boundaries}
\label{k2:sec:consequences}
The following elementary consequences will be used in the subsequent
Milnor--Witt calculations. They concern Quillen differentials.

\begin{lemma}[The low-dimensional degree-two row]\label{k2:lem:lowrow}
If $T$ is regular local of dimension at most three, the augmented row
\[
 0\to K_2(T)\to K_2(F)\to
 \bigoplus_{T^{(1)}}\kappa(x)^\times\to
 \bigoplus_{T^{(2)}}\Z\to0
\]
is exact. Also $\CH^2(T)=0$ for a regular local $T$ in any dimension.
\end{lemma}
\begin{proof}
The first two nonzero positions are Theorem~\ref{k2:thm:main}.
In \eqref{k2:eq:coniveau}, $\CH^2(T)=E_2^{2,-2}$ has zero positive
$K_0(T)=\Z$ limiting piece. Its sole possible higher incoming
map is $d_2$ from unramified $K_1$, namely $T^\times$; those units
extend and are permanent cycles. All higher outgoing targets have
negative field $K$-degree. Thus $\CH^2(T)=0$ without a dimension bound.

The row's degree-one cohomology is $E_2^{1,-2}$, a positive filtration
piece of $K_1(T)=T^\times$ in the limit, hence zero there. There are no
incoming higher differentials. In dimension at most three its only
possible outgoing higher target is $E_2^{3,-3}=\CH^3(T)$.
In dimension three this is the terminal Chow group, which is zero:
a DVR curve cut out by two members of a regular system of parameters
has a uniformizer whose divisor is the closed point with coefficient one. In smaller dimensions the target
is absent. This proves degree-one exactness as well.
\end{proof}

We will use a finite-quotient consequence of the integral row. Its
proof lifts a cocycle modulo the coefficient integer and corrects its
integral boundary. Torsion-freeness of the terminal divisor group is
what makes the correction possible.

\begin{corollary}[Finite-quotient image exactness]\label{k2:cor:quotientimage}
For $T$ as in Lemma~\ref{k2:lem:lowrow} and any integer $m\geq1$,
\[
 \im\bigl(K_2(T)/m\to K_2(F)/m\bigr)=
 \ker\left(K_2(F)/m\to
       \bigoplus_{T^{(1)}}\kappa(x)^\times/\kappa(x)^{\times m}\right).
\]
Only image exactness is asserted.
\end{corollary}
\begin{proof}
Lift a class of the kernel to $x\in K_2(F)$ and choose a finitely
supported unit cochain $b$ with $\partial x=mb$, in additive notation.
The next group is a direct sum of copies of $\Z$, so
$m\partial b=0$ implies $\partial b=0$.
Lemma~\ref{k2:lem:lowrow} supplies $y\in K_2(F)$ with $\partial y=b$.
The class $x-my$ is integrally unramified and extends by
Theorem~\ref{k2:thm:main}. The reverse inclusion follows from
localization.
\end{proof}

\begin{remark}\label{k2:rem:scope}
At an odd residue prime, the divided-Chern argument starts with
$b_F=p\alpha$ and an integrally unramified class $\alpha$.
The Kummer calculation uses $\Pic(Y')/p$, while the torsion-lifting
calculation uses $\Pic(Y')[p]$. At the dyadic prime,
Lemma~\ref{k2:lem:dyadicquotient} gives the motivic quotient injection
used to divide the class. These two constructions supply the
residue-prime steps in the integral theorem.
\end{remark}

%% file: sections/02_hermitian_and_support.tex
\section{Hermitian theories and the integral field symbols}
\label{gw:foundations}
The final argument uses alternating forms on projective modules.
The coniveau calculations, on the other hand, use homotopy-symmetric
forms on perfect complexes. We begin by defining the two theories and
the maps between them. The unit classes constructed in this section
will identify the low-degree field coefficients; localization will
then define the corresponding residue maps. Throughout, commutative rings carry the trivial involution; duality signs are displayed explicitly. Write $\operatorname{Perf}(B)$ for the stable category of perfect $B$-complexes. We use homological conventions:
\[
 H_i(X[a])=H_{i-a}(X),\qquad \Ddual_B X=\operatorname{RHom}_B(X,B).
\]
All rings called regular are Noetherian and of finite Krull dimension.

\subsection{The genuine and homotopy-symmetric structures}

For a perfect complex $X$, let $\operatorname{Bil}_B(X,X)$ be its bilinear mapping spectrum, with the transposition action of $C_2$. For the coefficient $B$ it is the spectrum underlying $\operatorname{RHom}_B(X,\Ddual_B X)$. The homotopy-symmetric and pure quadratic structures are
\[
 Q_B^{\mathrm{s}}(X)=\operatorname{Bil}_B(X,X)^{hC_2},\qquad
 Q_B^{\mathrm{q}}(X)=\operatorname{Bil}_B(X,X)_{hC_2}.
\]
There is also a sequence of genuine structures $Q_B^{\geq m}$. To specify it, put $E_B=(\mathrm H B)^{tC_2}$, with the derived $B$-module structure induced by the $\mathrm H\Z$-linear Tate diagonal. Then
\begin{equation}\label{gw:cutoff_pullback}
 Q_B^{\geq m}(X)=Q_B^{\mathrm{s}}(X)
 \mathop{\times}_{\operatorname{RHom}_B(X,E_B)}
 \operatorname{RHom}_B(X,\tau_{\geq m}E_B).
\end{equation}
This is the pullback of \cite[Recollection R.3]{GWCDHIII}, using \cite[Lemma 3.2.5]{GWCDHI}. The genuine symmetric structure is $Q_B^{\mathrm{gs}}=Q_B^{\geq0}$. On projective modules in degree zero it is the ordinary abelian group of symmetric bilinear forms; its extension to perfect complexes is the derived quadratic functor of \cite[Propositions 4.2.18 and 4.2.22]{GWCDHI}. The Tate diagonal determines the $B$-module structure on $E_B$ used
in this pullback. In the supported lifting argument, this structure
accounts for the Frobenius action on the modules appearing in the
comparison cofibre.

A genuine symmetric form includes a homotopy-symmetric form and a
compatible lift of its Tate image through the connective truncation
in \eqref{gw:cutoff_pullback}.

A Poincar\'e object for $Q$ is a pair $(X,q)$, with $q\in\pi_0Q(X)$ whose underlying duality map is an equivalence. The group $L_n(\mathcal C,Q)$ is the cobordism group of Poincar\'e objects for $Q[-n]$. Consequently
\begin{equation}\label{gw:l_shift}
 L_n(\mathcal C,Q[a])=L_{n-a}(\mathcal C,Q).
\end{equation}
Here shifting includes the biduality sign. Set
\[
 \begin{aligned}
 \GW^{[a]}(B)&=\GW(\operatorname{Perf}(B),Q_B^{\mathrm{gs}}[a]),\\
 \KO^{[a]}(B)&=\GW(\operatorname{Perf}(B),Q_B^{\mathrm{s}}[a]),\\
 L^{\mathrm{gs}}(B)&=L(B,Q_B^{\mathrm{gs}}),\qquad L^{\mathrm{s}}(B)=L(B,Q_B^{\mathrm{s}}).
 \end{aligned}
\]
We reserve the notation $\GW$ in square-bracket shifts for the genuine theory. For $a=2$ its nonnegative groups are the classical symplectic Grothendieck--Witt groups. The natural comparison is denoted $\GW^{[a]}\to\KO^{[a]}$.

We use two fundamental fibre sequences. The first is the genuine Bott sequence
\begin{equation}\label{gw:bott}
 \GW^{[a]}(B)\xrightarrow{f}K(B)\xrightarrow{h}
 \GW^{[a+1]}(B)\xrightarrow{\eta_{\mathrm g}}\Sigma\GW^{[a]}(B),
\end{equation}
where $f$ forgets the form and $h$ is hyperbolic. The second is
\begin{equation}\label{gw:fundamental}
 K(\mathcal C;\Ddual)_{hC_2}\longrightarrow\GW(\mathcal C,Q)
 \longrightarrow L(\mathcal C,Q).
\end{equation}
These are \cite[Corollaries 4.3.1 and 4.4.13]{GWCDHII}; the shifted formulation is also given in \cite[(1.1) and Section 3]{GWSchlichtingComparison}. In \eqref{gw:bott} the boundary is multiplication by the integral class $\eta_{\mathrm g}\in\GW_{-1}^{[-1]}(\Z)$. Products satisfy
\[
 xy=(-1)^{ij}\langle-1\rangle^{ab}yx
 \quad\bigl(x\in\GW_i^{[a]},\ y\in\GW_j^{[b]}\bigr).
\]
Put $\hsym=\langle1\rangle+\langle-1\rangle$. The hyperbolic projection formula, \cite[Corollary 7.5.6]{GWCDHI}, gives
\begin{equation}\label{gw:projection}
 h(f(x))=\hsym x,\qquad \eta_{\mathrm g}\hsym=0.
\end{equation}
In \eqref{gw:bott}, $h$ has target shift $a+1$; in
\eqref{gw:projection}, it has the same target shift as $x$.
We write $h^{[j]}$ when the target shift $j$ needs to be specified.
The shifted version of \eqref{gw:fundamental} has final term $\Sigma^aL^{\mathrm{gs}}(B)$. The suspension identity follows from bordism invariance. Pure quadratic and homotopy-symmetric $L$-theories are four-periodic \cite[Corollary R.10]{GWCDHIII}.

\subsection{An integral unit class}

To construct the Milnor--Witt comparison, we need a class for each
unit whose forgetful image and Bott boundary have prescribed values.
The multiplicative and Steinberg relations will then ensure that the
assignment factors through the unit-symbol presentation.

\begin{lemma}[The integral hermitian unit class]\label{gw:unit_class}
For every commutative ring $B$ and $u\in B^\times$ there is a natural class $[u]_{\mathrm g}\in\GW_1^{[1]}(B)$ such that
\begin{align}
 f[u]_{\mathrm g}&=u,& \eta_{\mathrm g}[u]_{\mathrm g}&=1-\langle u\rangle,\label{gw:unit_normalization}\\
 [uv]_{\mathrm g}&=[u]_{\mathrm g}+\langle u\rangle[v]_{\mathrm g},&[u^2]_{\mathrm g}&=\hsym[u]_{\mathrm g}.\label{gw:unit_laws}
\end{align}
Moreover $[u]_{\mathrm g}[1-u]_{\mathrm g}=0$ whenever $u$ and $1-u$ are units.
\end{lemma}

\begin{proof}
\proofstep{1. The universal unit class.}
We construct a universal class, determine it by two coordinates, and
prove its Steinberg relation by transfer from a quadratic algebra.
Specialization will give the required class over every ring.
Let $U=\Z[t,t^{-1}]$. The projective-line class $\beta$ of \cite[Proposition 3.2]{GWSchlichtingComparison} is constructed from the cone of $\mathcal O(-1)\to\mathcal O$; its forgetful class is $[\mathcal O]-[\mathcal O(-1)]$. Its restrictions to the two affine charts vanish. The projective-line formula and Mayer--Vietoris therefore give
\[
 \begin{split}
 0\to\GW_1^{[1]}(\Z)&\to\GW_1^{[1]}(U)
       \xrightarrow{\delta}\GW_0^{[0]}(\Z)\to0,\\
 0\to\GW_0^{[0]}(\Z)&\to\GW_0^{[0]}(U)
       \xrightarrow{\delta}\GW_{-1}^{[-1]}(\Z)\to0.
 \end{split}
\]
Evaluation at $t=1$ splits the left arrows. The homotopy invariance needed in these degrees follows from \cite[Theorem 1.2]{GWSchlichtingComparison} on rings smooth over $\Z$ and homotopy invariance of $\KO$; $\GW_0^{[1]}(\Z[x])=0$ also follows from its Lemma 2.4.

Define $[t]_{\mathrm g}$ by $\delta[t]_{\mathrm g}=\langle1\rangle$ and $[t]_{\mathrm g}|_{t=1}=0$. The forgetful projective-line class fixes the Bass orientation, giving $f[t]_{\mathrm g}=t$. One can also detect this equality after extension to $\Q$, since $K_1(U)=U^\times$ injects into $K_1(\Q[t,t^{-1}])$. For the second formula, both $\eta_{\mathrm g}[t]_{\mathrm g}$ and $1-\langle t\rangle$ are reduced at $1$, and their difference is detected by
\[
 \widetilde{\GW}_0(U)=\GW_{-1}^{[-1]}(\Z)=W(\Z).
\]
The last equality is \eqref{gw:bott}, since $K_{-1}(\Z)=0$. The map $W(\Z)\to W(\Q)$ is injective: intersect a rational Lagrangian of a unimodular integral lattice with that lattice. The intersection is saturated, hence a direct summand; its perpendicular equals it, since their quotient is torsion-free of rank zero. The same argument after hyperbolic stabilization handles every Witt-zero rational class. Over $\Q$ the equality $\eta_{\mathrm g}[t]_{\mathrm g}=1-\langle t\rangle$ with this Bass orientation is \cite[Lemma 6.9]{GWKSW}. Injection proves it over $\Z$. Specializing $t\mapsto u$ proves \eqref{gw:unit_normalization}.

\proofstep{2. The unit relations.}
The two coordinates in \eqref{gw:unit_normalization} determine the
class over the universal rings used below. To see this, suppose that
$\mathrm{SK}_1(B)=0$ and set
$I_{\mathrm{GW}}(B)=\ker(\GW_0(B)\to K_0(B))$. Bott exactness and the identity $fh(u)=u^2$ give a natural map
\begin{equation}\label{gw:degree_one_detection}
 (f,\eta_{\mathrm g}):\GW_1^{[1]}(B)\longrightarrow
 B^\times\mathop{\times}_{B^\times/B^{\times2}} I_{\mathrm{GW}}(B).
\end{equation}
The map from $I_{\mathrm{GW}}$ to square classes is the map on cokernels induced by $f$ in the two sequences with left arrows $h$ and squaring. The two exact sequences show that
\eqref{gw:degree_one_detection} is surjective. To prove injectivity,
let a class lie in the kernel of $(f,\eta_{\mathrm g})$. Bott exactness writes it as
$h(u)$, and its forgetful coordinate gives $u^2=1$. Multiplication by $u$ is an isometry of $\langle1\rangle$; its class in $\GW_1^{[0]}(B)$ forgets to $u$, so Bott exactness gives $h(u)=0$. Thus \eqref{gw:degree_one_detection} is an isomorphism. Apply it to $\Z[s^{\pm1},t^{\pm1}]$, whose $\mathrm{SK}_1$ is zero by the $K$-theory fundamental theorem. Comparing the two coordinates gives \eqref{gw:unit_laws}, and specialization gives those identities over every $B$.

\proofstep{3. The integral Steinberg relation.}
For the Steinberg relation put
\[
 B_0=\Z[a,a^{-1},(1-a)^{-1}],\qquad
 C=B_0[t]/(t^2-a),\qquad \lambda(1)=1,\quad\lambda(t)=0.
\]
The Frobenius pairing $(x,y)\mapsto\lambda(xy)$ has matrix $\operatorname{diag}(1,a)$, so restriction of scalars, together with
\[
 \operatorname{Res}\operatorname{Hom}_C(P,C)\longrightarrow
 \operatorname{Hom}_{B_0}(\operatorname{Res}P,B_0),
 \qquad \phi\longmapsto\lambda\phi,
\]
is a duality-preserving functor on projectives. The natural map on their quadratic form functors extends to the genuine derived structures by \cite[Propositions 4.2.18 and 4.2.22]{GWCDHI}. It therefore defines a genuine transfer $\lambda_*$ in all shifts. It commutes with $f$ and the Bott boundary. The tensor identity
\[
 \operatorname{Res}((C\otimes_{B_0}P)\otimes_C Q)
   =P\otimes_{B_0}\operatorname{Res}Q
\]
and its form pairing give the projection formula $\lambda_*(\operatorname{res}(x)y)=x\lambda_*(y)$, also after deriving and shifting. This construction uses finite projectivity and the displayed perfect pairing, so remains valid after arbitrary base change.

Direct calculation gives
\[
 \lambda_*\langle1\rangle=\langle1\rangle+\langle a\rangle,
 \qquad
 \lambda_*\langle1-t\rangle=\langle a\rangle+\langle1-a\rangle.
\]
Indeed the second Gram matrix is $\left(\begin{smallmatrix}1&-a\\-a&a\end{smallmatrix}\right)$, diagonalized by the basis $(t,1+t)$. The determinant of multiplication by $1-t$ is $1-a$. Hence $\lambda_*[1-t]_{\mathrm g}$ has coordinates $1-a$ and $1-\langle1-a\rangle$ in \eqref{gw:degree_one_detection}. Here $\mathrm{SK}_1(B_0)=0$: localization at the disjoint divisors $a=0,1$ in $\Z[a]$ gives $K_1(B_0)=K_1(\Z)\oplus\Z^2$, generated by $-1,a,1-a$. Thus
\[
 \lambda_*[1-t]_{\mathrm g}=[1-a]_{\mathrm g}.
\]
In $C$, both $t$ and $1-t$ are units, and ordinary $K$-theory satisfies the Steinberg relation. Equations \eqref{gw:projection} and \eqref{gw:unit_laws} imply
\[
 0=\hsym[t]_{\mathrm g}[1-t]_{\mathrm g}=[t^2]_{\mathrm g}[1-t]_{\mathrm g}
   =\operatorname{res}([a]_{\mathrm g})[1-t]_{\mathrm g}.
\]
Applying $\lambda_*$ gives $[a]_{\mathrm g}[1-a]_{\mathrm g}=0$. Every pair of complementary units is a specialization of $B_0$, proving the assertion.
\end{proof}

\subsection{The canonical field comparison through degree three}

The presentation-defined Milnor--Witt ring has generators $[u]$ in degree one and $\eta$ in degree minus one. Its relations are $[u][1-u]=0$ for complementary units, $\eta[u]=[u]\eta$, and
\[
 [uv]=[u]+[v]+\eta[u][v],\qquad
 \eta(\eta[-1]+2)=0.
\]
Lemma~\ref{gw:unit_class} defines a graded-ring map
\begin{equation}\label{gw:mu_definition}
 \mu:\KMW_*(B)\longrightarrow\GW_*^{[*]}(B),
 \qquad [u]\longmapsto[u]_{\mathrm g},\qquad \eta\longmapsto-\eta_{\mathrm g}.
\end{equation}
The product rule gives $\eta_{\mathrm g}[u]_{\mathrm g}=-\langle-1\rangle[u]_{\mathrm g}\eta_{\mathrm g}$; on multiples of $\eta_{\mathrm g}$, the factor $-\langle-1\rangle$ is $1$ by $\hsym\eta_{\mathrm g}=0$. This proves the commutation relation. The remaining relation maps to
\[
 \eta_{\mathrm g}^2[-1]_{\mathrm g}-2\eta_{\mathrm g}
 =\eta_{\mathrm g}(1-\langle-1\rangle-2)=-\eta_{\mathrm g}\hsym=0.
\]
The multiplicative and Steinberg relations were proved above. In particular $1+\eta[u]$ maps to the form $\langle u\rangle$.

\begin{theorem}[The field comparison through degree three]\label{gw:field_symbols}
For every field $k$ and $0\leq n\leq3$, the canonical maps
\[
 \mu_n:\KMW_n(k)\xrightarrow{\ \sim\ }\GW_n^{[n]}(k)
       \xrightarrow{\ \sim\ }\KO_n^{[n]}(k)
\]
are isomorphisms. On every unit symbol their forgetful and Witt coordinates satisfy
\begin{equation}\label{gw:field_coefficients}
 f\mu_n([u_1]\cdots[u_n])=\{u_1,\ldots,u_n\}\in K_n(k),
 \qquad
 \eta_{\mathrm g}^n\mu_n([u_1]\cdots[u_n])
     =\prod_i(1-\langle u_i\rangle)\in\GW_0(k).
\end{equation}
The first coordinate is the forward Milnor-to-Quillen product, with coefficient one. The second is the bilinear Pfister form after projection to $W(k)$.
\end{theorem}

\begin{proof}
The case in which two is invertible follows from the stated symbol
comparison. We give the remaining characteristic-two argument in
three stages: the Milnor--Witt pullback handles degrees one and two;
relative symplectic homology handles degree three over infinite fields;
and the finite-field calculation completes the proof.

\proofstep{1. The field comparisons and the remaining case.}
Degree zero is $\KMW_0(k)=\GW_0(k)$. If $\operatorname{char}k\ne2$, the assertion for the symbol map is \cite[Theorems 4.1.2 and 4.3.1]{AsokFaselKO}, with the sign of the Bott generator translated as in \eqref{gw:mu_definition}. Its field statement includes finite fields. We now prove the characteristic-two case. Comparison from genuine to homotopy-symmetric theory is an isomorphism on these diagonals: the field case of \cite[Theorem 1.2]{GWSchlichtingComparison} gives an isomorphism for $i\geq a-1$ in shift $a$. The formulas \eqref{gw:field_coefficients} already follow integrally from \eqref{gw:unit_normalization} and products.

\proofstep{2. The pullback in characteristic two.}
For a characteristic-two field, the bilinear Milnor conjecture and the Milnor--Witt pullback give
\begin{equation}\label{gw:field_pullback}
 \KMW_n(k)=\KM_n(k)\mathop{\times}_{\KM_n(k)/2}I^n(k),
 \qquad I^0(k)=W(k).
\end{equation}
Here $I$ is the fundamental ideal of the bilinear Witt ring; see \cite[Theorems 3.1 and 3.11]{GWCarlier} and \cite[Corollary 2.3.7]{DegliseMW}. We normalize the Witt coordinate of a degree-$n$ symbol as $\prod(1-\langle u_i\rangle)$. In characteristic two all Witt groups are killed by two, because the diagonal is a Lagrangian in $b\perp b$; thus the two signs for $\eta$ give identical maps in the following exact sequences.

Rank and Witt class identify $\GW_0(k)$ with $\Z\times_{\Z/2}W(k)$, so $I_{\mathrm{GW}}(k)=I(k)$. Formula \eqref{gw:degree_one_detection} is exactly \eqref{gw:field_pullback} for $n=1$, with the discriminant coordinate fixed by \eqref{gw:unit_normalization}. This proves the degree-one assertion. There is an exact sequence
\begin{equation}\label{gw:mw_bott}
 \KM_n(k)\xrightarrow{h_{\mathrm{MW}}}\KMW_n(k)\xrightarrow{\eta}
 \KMW_{n-1}(k)\xrightarrow{f}\KM_{n-1}(k)\to0,
\end{equation}
where, in the pullbacks, $h_{\mathrm{MW}}(x)=(2x,0)$ and $\eta(x,y)=(0,y)$. For $n=2$, compare this with \eqref{gw:bott}. The left comparison is Matsumoto's theorem and the right comparisons are the degree-one result and $K_1(k)=k^\times$. The projection formula \eqref{gw:projection} makes the diagram commute, giving surjectivity of $\mu_2$. Its composite with $(f,\eta_{\mathrm g})$ sends $(x,y)$ to $(x,(1,y))$, where $1$ denotes the identity in $k^\times$; the inclusion $I^2\subset I$ makes this composite injective. This proves the degree-two comparison.

\proofstep{3. The low symplectic Hurewicz maps.}
Suppose next that $k$ is infinite of characteristic two. The Hurewicz comparisons of \cite[Theorems 9.1 and 9.2]{SchlichtingSarwar} give isomorphisms $\alpha_2,\alpha_3$ from the genuine groups to the corresponding Milnor--Witt groups through
\[
 \begin{split}
 \GW_2^{[2]}(k)&\longrightarrow H_2(\Sp(k))
       \longleftarrow H_2(\SL_2(k))\longrightarrow\KMW_2(k),\\
 \GW_3^{[3]}(k)&\longrightarrow H_3(\SL(k),\Sp(k))
       \longleftarrow H_3(\SL_3(k),\SL_2(k))\longrightarrow\KMW_3(k).
 \end{split}
\]
The genuine interpretation follows from the symplectic model in shift two and the Bott fibre sequence: $\GW_3^{[3]}=\pi_3(\Bclass\SL^+,\Bclass\Sp^+)$. These group-homological theorems apply over the infinite field $k$. The stabilization square in \cite[(10.8)]{SchlichtingSarwar} identifies the degree-two forgetful map with the forward Milnor map. The relative boundary calculation of \cite[Theorem 5.37 and Lemma 5.40(2)]{SchlichtingEuler} gives
\begin{equation}\label{gw:alpha_boundary}
 \eta\alpha_3=\alpha_2\eta_{\mathrm g}.
\end{equation}
The hypotheses used for that calculation are the Milnor--Witt pullback and the bilinear Milnor conjecture, available here by \eqref{gw:field_pullback}; the module-boundary calculation applies over $k$.

\proofstep{4. The hyperbolic comparison.}
Consider $\beta=\alpha_3h\circ(\KM_3(k)\to K_3(k))$. Its Milnor coordinate is $2x$: this is the coefficient $(3-1)!$ in \cite[Theorem 4.1(a)]{GWNS}, in precisely the normalization used in \cite[Lemma 10.7]{SchlichtingSarwar}. Its Witt coordinate is zero by Bott exactness and \eqref{gw:alpha_boundary}, since $I^3\to I^2$ is injective. Thus $\beta=h_{\mathrm{MW}}$. If an opposite overall orientation is used for the auxiliary $\alpha_3$, replace it by its negative; \eqref{gw:alpha_boundary} is unchanged since its image is two-torsion. The normalization of $\mu_3$ remains the one fixed in \eqref{gw:mu_definition}.

\proofstep{5. Identification of the symbol map.}
Set $\varphi=\alpha_3\mu_3$ and $\psi=\alpha_2\mu_2$. The latter is an isomorphism preserving the Milnor projection, since the degree-two Hurewicz coefficient is $(2-1)!=1$. Naturality and the projection formula give a commuting diagram between two copies of \eqref{gw:mw_bott}, for $n=3$, with vertical arrows $\operatorname{id},\varphi,\psi,\operatorname{id}$. The map $\varphi$ is the identity on
$\im h_{\mathrm{MW}}=\ker\eta$. On the quotient, identified with
$\im\eta=\ker f$, it induces the isomorphism coming from $\psi$.
Applying the short five lemma to these kernel--quotient sequences shows
that $\varphi$ is an isomorphism. Since $\alpha_3$ is an isomorphism,
the same follows for $\mu_3$.

\proofstep{6. Finite fields of characteristic two.}
Finally let $k=\F_q$, $q$ even. Then $I(k)=0$ and $\KM_n(k)=0$ for $n\geq2$, so $\KMW_3(k)=0$. Apply \cite[Proposition 3.1.4]{GWCDHIII}:
\[
 \KO^{[a]}(\F_q)\simeq K(\F_q;\Ddual[a])^{hC_2}.
\]
Quillen's computation \cite{QuillenFinite} makes every positive $K$-group finite of odd order. For the fibre $E$ of $K(\F_q)\to \mathrm H\Z$, multiplication by two is therefore invertible, so homotopy fixed points have positive homotopy equal to invariants. The $\mathrm H\Z$ term, with either sign action, contributes no positive homotopy. On $K_3(\F_q)$ ordinary duality is $\psi^{-1}=+1$. To check its sign integrally, for each odd prime $\ell$ use the finite-field \(\ell\)-adic \(K\)-theory comparison: degree three is $H^1_{\et}(\F_q,\Z_\ell(2))$, and duality acts on weight two by $(-1)^2=1$. This comparison and its Adams compatibility are the finite-field argument in \cite[Sections 2.3.1 and 2.6.1, and Proposition 2.9]{GWFGV}; cohomological dimension one leaves no other weight.

The finite odd-order group is detected by its odd-primary completions. A duality shift by $a$ acts additionally by $(-1)^a$, by additivity. Thus for $a=3$ the invariants of $K_3(\F_q)$ are zero. We obtain $\KO_3^{[3]}(\F_q)=0$, and field comparison gives $\GW_3^{[3]}(\F_q)=0$. This includes $\F_2$ and completes the proof.
\end{proof}

\begin{remark}\label{gw:field_scope}
The field pullback \eqref{gw:field_pullback}, with the normalized Pfister coordinate, is the uniform formulation of \cite[Corollary 2.3.7]{DegliseMW}. Theorem~\ref{gw:field_symbols} identifies the diagonal groups $\GW_n^{[n]}$ and $\KO_n^{[n]}$ for $0\leq n\leq3$. The shifted degree-zero group $\GW_0^{[2]}$, used later, is the symplectic Grothendieck--Witt group.
\end{remark}

\section{Support theories and low Witt residues}
\label{gw:supports}

\subsection{Purity at a regular local support}

The field identifications will be used inside coniveau complexes.
We therefore fix the supported theory before transporting any boundary
to symbols. This also fixes the exceptional line and the duality shift
at each point.

For a closed set $Z\subset Y$, the notation $\operatorname{Perf}_Z(Y)$ means perfect complexes whose homology is supported on $Z$, with the \emph{restriction of the ambient} Poincar\'e structure. In particular, the genuine structure on the support category is obtained by restricting the ambient cutoff.

We use localization and purity in the following form. For a Noetherian
$Y$, write $\operatorname{Perf}(Y)^{\geq p}$ for perfect complexes
whose homology is supported in codimension at least $p$.
The superscript refers to support, rather than homological degree.
The support filtration has Poincar\'e--Verdier quotients
\begin{equation}\label{gw:coniveau_quotient}
 \operatorname{Perf}(Y)^{\geq p+1}\longrightarrow
 \operatorname{Perf}(Y)^{\geq p}\longrightarrow
 \bigoplus_{x\in Y^{(p)}}\operatorname{Perf}_{x}(\mathcal O_{Y,x}),
\end{equation}
with the restricted structures \cite[Proposition 4.5.1]{GWCHN}. For a closed immersion $i:Z\hookrightarrow Y$ of finite-dimensional regular Noetherian schemes and a line bundle $\mathcal L$ with involution, the functor with duality
\[
 (\operatorname{Perf}(Z),Q^{\mathrm{s}}_{i^!\mathcal L}[a])
 \longrightarrow
 (\operatorname{Perf}_Z(Y),Q^{\mathrm{s}}_{\mathcal L}[a]|_Z)
\]
induces equivalences on Grothendieck--Witt and $L$-theory, for every $a\in\Z$ \cite[Theorem 5.2.1]{GWCHN}. If $i$ has codimension $p$, then
\begin{equation}\label{gw:exceptional_line}
 i^!\mathcal O_Y=\det(N_{Z/Y})[-p]
\end{equation}
with its induced biduality \cite[Remark 5.1.10]{GWCHN}. Ordinary and Karoubi-localizing theories agree on these categories; their bounded Noetherian support hearts have vanishing negative $K$-groups \cite[Remark 5.2.4]{GWCHN}.

In particular, for a point $x$ of codimension $p$ in a regular $Y$, apply this theorem to the closed point of the \emph{regular local scheme} $\Spec\mathcal O_{Y,x}$. Set
\[
 \omega_x=\operatorname{Ext}^p_{\mathcal O_{Y,x}}
       (\kappa(x),\mathcal O_{Y,x}).
\]
This is the determinant line in \eqref{gw:exceptional_line}, computed on the regular local scheme $\Spec\mathcal O_{Y,x}$.

The finite support tower gives the symmetric coniveau spectral sequence
\begin{equation}\label{gw:l_coniveau}
 E_1^{p,q}=\bigoplus_{x\in Y^{(p)}}L^{\mathrm{s}}_{-q}(\kappa(x),\omega_x)
 \ \Longrightarrow\ L^{\mathrm{s}}_{-p-q}(Y),
 \qquad d_s:(p,q)\longmapsto(p+s,q-s+1).
\end{equation}
The shift $[-p]$ in purity cancels the codimension contribution to
the total degree: \eqref{gw:l_shift} turns supported degree $-p-q$
into field degree $-q$. The line $\omega_x$ retains the orientation
data of the support.
The $q=0$ row is denoted $C_{\mathrm W}(Y)$. Its coefficient is the full bilinear Witt group, including at characteristic-two points. Odd field $L$-groups vanish; the even groups are the Witt groups with the appropriate positive or negative duality sign \cite[Corollary 1.3.8]{GWCDHIII}. Negative-sign bilinear forms are alternating when two is invertible, so that group is zero there. In characteristic two the two bilinear signs agree.

For a one-dimensional $k$-vector space $\omega$, the twisted symbol group is
\[
 \KMW_n(k,\omega)=\KMW_n(k)\otimes_{\Z[k^\times]}\Z[\omega^\times],
\]
where $u\in k^\times$ acts on $\KMW_n(k)$ by $\langle u\rangle$ and on nonzero vectors by scalar multiplication. The analogous construction gives the Witt and Grothendieck--Witt coefficients. The map $\mu$ is $\GW_0(k)$-linear by its product construction, so Theorem~\ref{gw:field_symbols} descends to these twists.
For a chosen nonzero $e\in\omega$, the map
$x\mapsto x\otimes[e]$ identifies the untwisted coefficient group
with its $\omega$-twist. Replacing $e$ by $ue$ multiplies this
identification by $\langle u\rangle$. The tensor product therefore
records the change of basis in the coefficient line.

The same support tower for $\KO^{[a]}$ has coefficients
\begin{equation}\label{gw:ko_coniveau}
 {}^aE_1^{p,q}=\bigoplus_{x\in Y^{(p)}}
 \KO_{-p-q}^{[a-p]}(\kappa(x),\omega_x)
 \ \Longrightarrow\ \KO_{-p-q}^{[a]}(Y).
\end{equation}
In particular its diagonal row has coefficients $\KO_{a-p}^{[a-p]}$. Theorem~\ref{gw:field_symbols} supplies the canonical Milnor--Witt identification in every nonnegative field degree at most three. Every differential mentioned below is the differential of \eqref{gw:coniveau_quotient}, transported through these specified equivalences. Naturality of the forgetful map and of Bott multiplication holds on the entire tower, before taking any field coefficients. Connecting maps carry the usual suspension signs when a graded operation is moved past a boundary; these signs are retained in the integral maps.

\subsection{Coefficient one and the low filtered squares}

\begin{lemma}[Compatibility of the low filtered residues]\label{gw:residue_squares}\label{gw:low_residues}
The supported differentials preserve the low fundamental-ideal filtration. In the positions where the indicated coefficients occur, they satisfy
\[
 \partial(I^2)\subset I,\qquad
 \partial(I^3)\subset I^2,\qquad
 \partial(I^4)\subset I^3.
\]
On the associated low quotients the canonical maps
\[
 e_1:I/I^2\simeq\KM_1/2,
 \qquad e_2:I^2/I^3\simeq\KM_2/2,
 \qquad W/I\simeq\Z/2
\]
intertwine the Witt differential and the forward Quillen tame differential, modulo two. At a DVR $V$ with uniformizer $\pi$, residue field $k$, and $u\in V^\times$, a compatible trivialization of the exceptional line gives
\begin{equation}\label{gw:dvr_coefficient_one}
 \partial\langle\pi u\rangle=\langle\bar u\rangle,
 \qquad \partial\langle u\rangle=0.
\end{equation}
The corresponding diagonal unit-symbol boundary has coefficient one.
\end{lemma}

\begin{proof}
\proofstep{1. The normalized DVR boundary.}
First fix the DVR normalization. Formula \eqref{gw:dvr_coefficient_one} is \cite[Lemma 2.2.9]{GWCDHIII}, which applies to every local Dedekind ring. We describe its representative in the support category. For $a=\pi u$, the integral lattice $V$ in the rank-one generic form has dual lattice $a^{-1}V$. The boundary is represented by $a^{-1}V/V$ with pairing
\[
 ([x],[y])\longmapsto[axy]\in\Frac(V)/V.
\]
Purity and the normal-line orientation from $\pi$ identify this with $\langle\bar u\rangle$. Unimodular unit forms extend and have zero boundary. Likewise the Bass class of $\pi$ has boundary the rank-one class $\langle1\rangle$ of the one-term residue module; this is the cone of multiplication by $\pi$, with the same orientation as the projective-line class defining $[t]_{\mathrm g}$. The product formula for localization then sends a uniformizer symbol multiplied by extending unit symbols to the product of their residue symbols, with coefficient one. The comparison carries this calculation to the homotopy-symmetric support theory.

\proofstep{2. Incidence maps in the support sequence.}
For an arbitrary incidence, compute the boundary in \eqref{gw:ko_coniveau}. At the source field $E$, represent $w\in I^2(E)$ by a pair $(x,w)\in\KMW_2(E)$ in the field pullback, choosing $x\in\KM_2(E)$ lifting $e_2(w)$. The canonical identification $\mu_2$, the forgetful map, and Bott multiplication identify the relevant maps of the support towers, with their shifted boundary conventions. At every target field the boundary lies in $\KMW_1$ and its Milnor coordinate is the Quillen tame boundary of $x$.

Applying Bott multiplication identifies its Witt coordinate with the full Witt boundary of $w$, with the suspension sign prescribed by the shifted support tower. This sign has no effect on membership in a fundamental-ideal power, and disappears on the associated graded, which is killed by two. The degree-one field pullback therefore gives
\[
 e_1(\partial w)=\partial(e_2(w)).
\]
If $w\in I^3(E)$, take the pair $(0,w)$ instead. Its boundary has zero Milnor coordinate, so its Witt component belongs to $I^2$. Similarly, the degree-one pair $(0,w)$ for $w\in I^2$ has boundary in $\KMW_0=\GW_0$ of rank zero, hence Witt component in $I$. For a general degree-one pair the rank of the forgetful boundary is the valuation, giving the $e_1$/rank-parity square. The identical argument with $\mu_3$ and $(0,w)$ for $w\in I^4$ gives $\partial w\in I^3$.

\proofstep{3. The filtered residue squares.}
For example, the arrow from codimension one to codimension two is obtained by putting $\KMW_2$ and $\KMW_1$ into the corresponding positions of the shift-three tower. Both natural transformations exist on that tower before any quotient is taken, so the preceding diagram is the ambient incidence map. The argument applies equally to any adjacent codimensions in the low tail. Finally changing a value-line trivialization multiplies a Witt coefficient by $\langle v\rangle$; this acts trivially on the associated graded, since $\langle v\rangle\equiv1\pmod I$. Thus the squares and the coefficient-one conclusions have the stated line-twisted interpretation.
\end{proof}

For later use write $C_{I,j}(Y)$ for the filtered row whose codimension-$p$ term is $I^{j-p}(\kappa(x),\omega_x)$, with $I^a=W$ for $a\leq0$, whenever its indicated low-tail differentials are used. Lemma~\ref{gw:residue_squares} supplies precisely those differentials and graded squares. Exactness of these filtered rows will be established in the cohomological degrees used below.

\section{Local genuine vanishing and the comparison window}

\subsection{Positive quadratic formations over a local ring}

The additional local vanishing in degree minus three reduces to a
positive quadratic formation group. Its proof is an elimination
argument: we use quadratic isometries with zero formation class to
put both Lagrangians in standard position.

\needspace{7\baselineskip}
\begin{lemma}[Positive quadratic formations over a local ring]\label{gw:positive_formations}
For every commutative local ring $T$,
\[
 L_1^{\mathrm{q}}(T;+)=0.
\]
This is the undecorated projective pure quadratic group with positive sign.
\end{lemma}

\begin{proof}
\proofstep{1. Quadratic isometries.}
The formation model agrees with pure quadratic $L$-theory: see \cite[Remark 1.2.25]{GWCDHIII} and \cite[Proposition 5.2]{GWRanickiSurgery}. A formation has two direct-summand Lagrangians in a nonsingular quadratic module. Since finite projectives over $T$ are free and a metabolic quadratic module is hyperbolic \cite[Theorem 1.1]{GWRanickiFoundations}, put the first Lagrangian in standard position
\[
 \Lambda_0=T^n\oplus0\subset \Hypq=T^n\oplus(T^n)^*,
 \qquad q(x,y)=x^{\mathsf t}y.
\]
The associated bilinear pairing is
$b((x,y),(x',y'))=x^{\mathsf t}y'+(x')^{\mathsf t}y$.
A Lagrangian is a direct summand $L=L^\perp$ on which $q$ vanishes.
Here $\Hypq$ carries this positive quadratic structure;
$\Hyp$, used in the completion argument, carries an alternating form.
For an isometry $g$, let $\nu(g)=[\Hypq;\Lambda_0,g\Lambda_0]$. The formation cocycle relation \cite[Lemma 3.3]{GWRanickiFoundations} gives $\nu(gh)=\nu(g)+\nu(h)$. We use the following three types of isometries. Each has zero formation
class for the reason indicated:
\begin{enumerate}
\item $D_u(x,y)=(ux,u^{-\mathsf t}y)$ preserves $\Lambda_0$.
\item $L_B(x,y)=(x,y+Bx)$, with $B^{\mathsf t}=-B$ and zero diagonal, is a quadratic isometry, since $x^{\mathsf t}Bx=0$. The Lagrangians $\Lambda_0,L_B\Lambda_0$ have the common complement $0\oplus(T^n)^*$, so their formation is a boundary.
\item Exchanging one hyperbolic pair $e_i,f_i$ preserves the positive quadratic function and gives the sum of a complementary-plane formation and equal-Lagrangian formations, hence zero.
\end{enumerate}
The last two boundary assertions are the complementary and common-complement relations of \cite[Lemmas 2.1 and 2.2]{GWRanickiFoundations}.

\proofstep{2. Reduction to the standard Lagrangian.}
These transformations act transitively on Lagrangians. Choose a basis vector $v=(x,y)$ of a Lagrangian $\Lambda_1$. It is unimodular in $\Hypq$, because $\Lambda_1$ is a direct summand. Since $T$ is local, some coordinate is a unit. After a pair exchange if necessary, a coordinate of $x$ is a unit; a transformation $D_u$ puts $x=e_1$. Quadratic isotropy gives $y_1=0$ exactly. Set
\[
 B_{j1}=-y_j,\qquad B_{1j}=y_j\quad(j>1),
 \qquad B_{ij}=0\text{ otherwise}.
\]
Then $L_B(e_1,y)=e_1$. Every other basis vector of the transformed $\Lambda_1$ has zero $f_1$-coordinate by orthogonality to $e_1$. Subtracting its $e_1$-coordinate puts it in the smaller hyperbolic module on the remaining pairs. The remaining summand is a direct summand of the smaller hyperbolic
module. A retraction is obtained by restricting a retraction onto
$\Lambda_1$ and projecting away its first basis vector. Its orthogonal
complement in that smaller module is itself, so it is a Lagrangian. Induction reduces it to the standard one. Thus some product $g$ of the three types takes $\Lambda_1$ to $\Lambda_0$, and $\nu(g)=0$ proves $[\Hypq;\Lambda_0,\Lambda_1]=0$.

\proofstep{3. Preservation of the quadratic form.}
The zero diagonal of $B$ gives $x^{\mathsf t}Bx=0$, so the transformations preserve the quadratic function $q(x,y)=x^{\mathsf t}y$.
\end{proof}

\begin{proposition}[Local genuine negative groups]\label{gw:genuine_local_negative}
For every commutative local ring $T$,
\[
 L_{-1}^{\mathrm{gs}}(T)=L_{-2}^{\mathrm{gs}}(T)=L_{-3}^{\mathrm{gs}}(T)=0.
\]
\end{proposition}

\begin{proof}
The first two equalities are \cite[Lemma 2.1]{GWSchlichtingComparison}. To recall its algebraic content, local symplectic elimination gives $\GW_1^{[2]}(T)=0$ and $K_0(T)\xrightarrow{h}\GW_0^{[2]}(T)$ an isomorphism, with both groups $\Z$ when symplectic rank is normalized by hyperbolic planes. The Bott sequence gives $\GW_0^{[1]}(T)=0$, and the surjections $\GW_0^{[a]}(T)\to L_{-a}^{\mathrm{gs}}(T)$ give the assertions in degrees $-1,-2$. For the additional degree, $Q^{\mathrm{gs}}=Q^{\geq0}$ is $0$-quadratic. Hence \cite[Corollary 1.2.12]{GWCDHIII}, with $m=0$, gives
\[
 L_{-3}^{\mathrm{gs}}(T)=L_{-3}^{\mathrm{q}}(T;+)=L_1^{\mathrm{q}}(T;+)=0.
\]
The middle equality uses only four-periodicity of \emph{pure quadratic} theory, \cite[Corollary R.10]{GWCDHIII}; the final equality is Lemma~\ref{gw:positive_formations}.
\end{proof}

\subsection{A comparison range from two-torsion-freeness}

\begin{proposition}[The two-torsion-free comparison range]\label{gw:comparison_window}
Let $B$ be a commutative coherent ring of finite global dimension $d$, with two a nonzerodivisor. Then
\[
 L_n^{\mathrm{gs}}(B)\longrightarrow L_n^{\mathrm{s}}(B)
\]
is injective for $n\geq d-4$ and bijective for $n\geq d-3$. The genuine-to-homotopy-symmetric comparison
\[
 \GW_i^{[a]}(B)\longrightarrow\KO_i^{[a]}(B)
\]
is injective for $i-a\geq d-4$ and bijective for $i-a\geq d-3$.
\end{proposition}

\begin{proof}
For the trivial involution with positive sign, $\widehat H^1(C_2;B)=B[2]=0$. Formula \eqref{gw:cutoff_pullback}, or \cite[Remark R.4]{GWCDHIII} with $m=-1$, therefore identifies $Q_B^{\geq0}$ with $Q_B^{\geq-1}$. The latter is $3$-symmetric by \cite[Example 1.1.4]{GWCDHIII}. Corollary 1.3.9 of that paper says that an $s$-symmetric structure over a coherent ring of global dimension $d$ compares injectively in degrees $n\geq d-2s+2$ and bijectively in degrees $n\geq d-2s+3$. Substituting $s=3$ gives the claimed bounds. The square induced by \eqref{gw:fundamental} is cartesian, since the two structures have the same underlying duality and thus the same $K_{hC_2}$ term. Its comparison fibres agree. Formula \eqref{gw:l_shift} therefore transfers the range to the shifted $\GW$ comparison exactly as stated.
\end{proof}

\section{Coconnective surgery on a support category}
\label{gw:supported_surgery_section}

We now study the lifting of homotopy-symmetric classes with support
to genuine classes. The obstruction to lifting is measured by a linear
mapping spectrum. Rather
than trying to lift an arbitrary representative, we first perform
surgery so that its homology lies in a range where the obstruction
vanishes. The resulting lifting theorem will give the local symmetric
vanishings in degrees $-1$ and $-2$.

\subsection{The linear obstruction}

Let $T$ be regular local, with two a nonzerodivisor. Put
\[
 E_- =\tau_{<0}(\mathrm H T)^{tC_2},\qquad \NFr=\operatorname{Fr}_*(T/2),
 \qquad t\cdot\bar a=\overline{t^2a}.
\]
Taking the cofibre of \eqref{gw:cutoff_pullback} gives the comparison sequence
\begin{equation}\label{gw:linear_comparison}
 Q_T^{\mathrm{gs}}(X)\longrightarrow Q_T^{\mathrm{s}}(X)
       \longrightarrow\operatorname{RHom}_T(X,E_-).
\end{equation}
Tate cohomology, with the diagonal action specified there, gives
\begin{equation}\label{gw:tate_groups}
 \pi_{-2\ell}E_-=\NFr\quad(\ell\geq1),\qquad
 \pi_jE_-=0\quad\text{otherwise}.
\end{equation}
Indeed scalar multiplication on the two inputs acts on the diagonal by its square; modulo two this is additive. We compute the lifting obstruction using the Postnikov tower of $E_-$.

The comparison cofibre has Frobenius-twisted homotopy modules.
The following grade estimate is the algebraic input that will force
maps into those modules to vanish. The codimension bound loses one
because the coefficient module is supported on the divisor defined
by two.

\begin{lemma}[The Frobenius-twisted grade estimate]\label{gw:frobenius_grade}
If a finite $T$-module $M$ has support of codimension at least $b\geq1$, then
\[
 \operatorname{Ext}_T^a(M,\NFr)=0\qquad(a<b-1).
\]
\end{lemma}

\begin{proof}
The assertion is immediate if $2\in T^\times$. Otherwise $T/2$ is Cohen--Macaulay. Put $J=\operatorname{Ann}_T(M)$. For every prime $\mathfrak p\supset J+(2)$,
\[
 \operatorname{height}_{T/2}(\mathfrak p/(2))
 =\dim(T_{\mathfrak p}/2)=\operatorname{height}_T\mathfrak p-1\geq b-1.
\]
Thus the image of $J$ has grade at least $b-1$ in $T/2$. Choose $t_1,\ldots,t_{b-1}\in J$ whose images are a regular sequence on $T/2$. On $\NFr$ these scalars act by $t_1^2,\ldots,t_{b-1}^2$, still a regular sequence; after the first $i$ scalars the underlying additive quotient is $(T/2)/(t_1^2,\ldots,t_i^2)$.

We apply the annihilator regular-sequence argument to $\NFr$. For one regular scalar $t\in J$, every homomorphism $M\to \NFr$ is zero. For a longer sequence use
\[
 0\longrightarrow \NFr\xrightarrow{t_1}\NFr\longrightarrow \NFr/t_1\NFr\longrightarrow0.
\]
The scalar $t_1$ acts as zero on every $\operatorname{Ext}_T^a(M,\NFr)$ because it annihilates $M$. Induction applied to the remaining sequence on $\NFr/t_1\NFr$, followed by the long exact sequence, makes this zero action injective for $a<b-1$. Hence those Ext groups vanish.
\end{proof}

\begin{theorem}[Supported genuine lifting]\label{gw:supported_lift}
Let $T$ be a finite-dimensional regular local ring with two a nonzerodivisor. Let $c\geq1$, and let $\mathcal C\subset\operatorname{Perf}(T)$ be a full stable, duality-invariant subcategory closed under ordinary homological truncations, all of whose homology has support of codimension at least $c$. The map for the \emph{restricted ambient structures}
\[
 L_0(\mathcal C,Q_T^{\mathrm{gs}}[c])\longrightarrow
 L_0(\mathcal C,Q_T^{\mathrm{s}}[c])
\]
is surjective.
\end{theorem}

\begin{proof}
Starting with a symmetric Poincar\'e representative, we first construct
an isotropic truncation and perform surgery. The resulting coconnective
representative is still supported in the prescribed category.
Self-duality then improves the codimension bound on each negative
homology module. These improved bounds, together with
Lemma~\ref{gw:frobenius_grade}, will annihilate its lifting obstruction.

\proofstep{1. The isotropic truncation.}
Put $\Ddual_c =\Ddual_T [c]$. If $L\in\mathcal C$ is concentrated in homological degrees at least one, the grade formula gives
\[
 \Ddual_T L\in \mathrm D(T)_{\leq-c-1},\qquad \Ddual_c L\in \mathrm D(T)_{\leq-1}.
\]
This follows from the finite homology tower and $\operatorname{Ext}_T^a(H_iL,T)=0$ for $a<c$. Hence
\[
 \operatorname{RHom}_T(L,\Ddual_c L)\in\mathsf{Sp}_{\leq-2}.
\]
The shifted homotopy-symmetric form spectrum is the homotopy-fixed-point spectrum of this bilinear spectrum, with the shifted transposition action. Homotopy fixed points preserve this upper truncation. Therefore $Q_T^{\mathrm{s}}[c](L)$ has contractible infinite-loop space.

\proofstep{2. The supported surgery construction.}
Take a Poincar\'e representative $(X,q)$ of a class on the right and use the truncation map $f:L=\tau_{\geq1}X\to X$. The restriction $f^*q$ has a contractible space of nullhomotopies, providing the full isotropic surgery datum. Algebraic surgery \cite[Definition 2.4.1 and Proposition 2.4.3]{GWCDHII}, also \cite[Proposition 1.1.13]{GWCDHIII}, gives a Poincar\'e cobordism from $X$ to
\begin{equation}\label{gw:surgery_formula}
 Y=\operatorname{fib}\bigl(\operatorname{cofib}(L\to X)\longrightarrow \Ddual_c L\bigr).
\end{equation}
Equivalently, put $\chi(f)=\operatorname{fib}(X\to \Ddual_c L)$; the chosen isotropic nullhomotopy supplies $L\to\chi(f)$, and $Y=\operatorname{cofib}(L\to\chi(f))$. All these objects remain in $\mathcal C$. Since $\operatorname{cofib}(L\to X)=\tau_{\leq0}X$ and $\Ddual_c L\leq-1$, the long exact homology sequence gives
\begin{equation}\label{gw:surgered_bounds}
 Y\leq0,\qquad Y\simeq \Ddual_T Y[c].
\end{equation}
The negative homology of $Y$ is controlled by the following support estimate.

\proofstep{3. Support bounds from self-duality.}
Write $M_j=H_{-j}Y$ for $j\geq0$. At a prime $\mathfrak p$ of height $h$, the regular local ring $T_{\mathfrak p}$ has global dimension $h$. Since $Y_{\mathfrak p}\leq0$, its dual is concentrated in degrees at least $-h$. Self-duality in \eqref{gw:surgered_bounds} therefore gives $Y_{\mathfrak p}\geq c-h$. If $h<c+j$, its homology in degree $-j$ is zero. Thus
\begin{equation}\label{gw:deepening_support}
 \operatorname{codim}\operatorname{Supp}M_j\geq c+j.
\end{equation}
We apply this support bound to the homology modules in the Postnikov tower of $Y$.

\proofstep{4. Vanishing of the obstruction groups.}
By Lemma~\ref{gw:frobenius_grade}, $\operatorname{Ext}^a_T(M_j,\NFr)=0$ for $a<c+j-1$. The Postnikov pieces of $E_-[c]$ are $\NFr[c-2\ell]$, $\ell\geq1$, and
\begin{equation}\label{gw:postnikov_ext}
 \pi_0\operatorname{RHom}_T(M_j[-j],\NFr[c-2\ell])
   =\operatorname{Ext}_T^{c+j-2\ell}(M_j,\NFr)=0.
\end{equation}
For a nonnegative exponent this is the strict grade inequality $c+j-2\ell<c+j-1$; negative Ext is zero as well.

\proofstep{5. The finite Postnikov calculation.}
Since $Y$ is bounded, the obstruction can be computed with a finite Postnikov section. Choose $m$ with $Y\geq m$, and then $k<m$. For $V=\tau_{\leq k}(E_-[c])$, orthogonality gives $\operatorname{RHom}_T(Y,V)\leq k-m\leq-1$. Its $\pi_0$ and $\pi_1$ vanish, so the map from the finite retained Postnikov section $\tau_{\geq k+1}(E_-[c])$ induces an isomorphism on degree-zero maps out of $Y$. Both that section and $Y$ have finite Postnikov towers. For each triangle in either finite Postnikov tower, the exact sequence
on $\pi_0$ shows that vanishing for the two outside mapping spectra
implies vanishing for the middle one. Applying this argument successively
to the two towers and using \eqref{gw:postnikov_ext} gives
\begin{equation}\label{gw:obstruction_zero}
 \pi_0\operatorname{RHom}_T(Y,E_-[c])=0.
\end{equation}
The finite Postnikov calculation therefore proves the vanishing of the obstruction group.

\proofstep{6. Lifting the Poincar\'e structure.}
Shift \eqref{gw:linear_comparison} by $c$. Equation~\eqref{gw:obstruction_zero} lifts the symmetric structure on $Y$ to $Q_T^{\mathrm{gs}}[c](Y)$. Comparison preserves the underlying duality map, so this lift is still nondegenerate. Its genuine supported class maps to the class of $X$, because the passage from $X$ to $Y$ was a supported cobordism. This proves surjectivity.
\end{proof}

\subsection{Local symmetric groups in degrees minus one and minus two}

\begin{theorem}[Local symmetric negative groups]\label{gw:local_negative}
For every finite-dimensional regular local ring $T$ with two a nonzerodivisor,
\[
 L^{\mathrm{s}}_{-1}(T)=L^{\mathrm{s}}_{-2}(T)=0.
\]
In addition $L^{\mathrm{gs}}_{-1}(T)=L^{\mathrm{gs}}_{-2}(T)=L^{\mathrm{gs}}_{-3}(T)=0$.
\end{theorem}

\begin{proof}
We reduce the symmetric classes to classes with support and then apply
Theorem~\ref{gw:supported_lift}. In degree minus one, torsion support
already has the required codimension. In degree minus two, we must first
remove its height-one part.

\proofstep{1. The genuine and generic groups.}
The genuine assertions are Proposition~\ref{gw:genuine_local_negative}. Let $F=\Frac(T)$. Both $L^{\mathrm{s}}_{-1}(F)$ and $L^{\mathrm{s}}_{-2}(F)$ vanish: the former is odd-field vanishing and the latter is the Witt group of alternating forms over a field of characteristic different from two.
If $\dim T=0$, then $T=F$ and these vanishings prove the symmetric assertions.
For the remaining steps, assume $\dim T\geq1$.

\proofstep{2. Classes in degree minus one.}
The flat localization $T\to F$ is a Poincar\'e--Verdier quotient. Indeed its essential image has full $K_0$ image, since $[T]$ maps to the generator of $K_0(F)=\Z$; apply \cite[Proposition 1.4.9 and Corollary 1.4.10]{GWCDHII}. The $L$-theory localization sequence is supplied by \cite[Corollary 4.4.5]{GWCDHII}. Every class under consideration thus comes from the torsion-perfect category. Its representative has finitely many finite homology modules, so it belongs to $\operatorname{Perf}_{V(f)}(T)$ for one nonzero nonunit $f\in T$. In positive dimension, a supporting element can be multiplied by a nonzero nonunit if necessary. This category is stable, duality invariant and closed under truncations. It has support of codimension at least one. Apply Theorem~\ref{gw:supported_lift} with $c=1$, then push to $L^{\mathrm{gs}}_{-1}(T)=0$. This proves the degree-$-1$ assertion.

\proofstep{3. The height-one part in degree minus two.}
For a degree-$-2$ class, choose a nonzero nonunit $f\in T$ supporting a lift, as in the preceding step. We first remove its height-one support. Let $\mathfrak p_1,\ldots,\mathfrak p_s$ be the height-one prime divisors of $f$, and put
\[
 U=T\setminus\bigcup_i\mathfrak p_i,\qquad B=U^{-1}T.
\]
Finite prime avoidance shows that $B$ is regular semilocal of dimension one with maximal ideals $\mathfrak p_iB$. The category $\operatorname{Perf}_{V(f)}(B)$ has the bounded heart of finite-length $B$-modules. On that heart $\Ddual_B [1]$ is exact, since only $\operatorname{Ext}^1_B(-,B)$ occurs. Thus the odd-group assertion of \cite[Corollary 1.3.8]{GWCDHIII} gives
\[
 L^{\mathrm{s}}_{-2}(\operatorname{Perf}_{V(f)}(B),Q_B^{\mathrm{s}})
   =L_{-1}(\operatorname{Perf}_{V(f)}(B),Q_B^{\mathrm{s}}[1])=0.
\]
To identify the supported localization, consider its image on $K_0$. Its target after idempotent completion is the displayed category, and
\[
 K_0(\operatorname{Perf}_{V(f)}(B))
   =\bigoplus_i\Z[\kappa(\mathfrak p_i)].
\]
Each generator is the restriction of the perfect $T$-module $T/\mathfrak p_i$. The $K_0$ image is therefore full, so the supported localization is a Verdier quotient before idempotent completion. The target Poincar\'e structure is the restriction of $Q_B^{\mathrm{s}}$: this is the supported form of flat localization in \cite[Proposition 1.4.9]{GWCDHII}. Concretely, bilinear complexes of perfect objects are bounded, so each degree of their standard $C_2$-resolution involves finitely many coefficients, which can be localized. Scalar denominator maps preserve the support; on diagonal forms they act by squares, and inverting a scalar or its square gives the same localization. Thus the induced structure is the restricted symmetric structure.

\proofstep{4. Lifting from codimension-two support.}
The kernel $\mathcal C$ consists of perfect complexes supported on $V(f)$ which vanish after tensoring with $B$. Their homology has no height-one support, and hence is supported in codimension at least two. The kernel is closed under truncations and duality. Exactness lifts the degree-$-2$ class to
\[
 L_0(\mathcal C,Q_T^{\mathrm{s}}[2]).
\]
Apply Theorem~\ref{gw:supported_lift} with $c=2$, and then push its lift to $L^{\mathrm{gs}}_{-2}(T)=0$. The original class is zero, as required.
\end{proof}

\section{Equal-characteristic-two generic injectivity}

\begin{theorem}[Generic injectivity in characteristic two]\label{gw:equalchar_injection}
Let $S$ be a finite-dimensional regular local ring of characteristic two and let $F=\Frac(S)$. The restriction map
\[
 L^{\mathrm{s}}_{-2}(S)\longrightarrow L^{\mathrm{s}}_{-2}(F)
\]
is injective. Also $L^{\mathrm{s}}_{-1}(S)=0$. These statements hold for ordinary local rings, including finite residue fields.
\end{theorem}

\begin{proof}
We first obtain generic injectivity from motivic representability over
a smooth local algebra. We then descend a given class and its generic
nullhomotopy through a smooth approximation of the regular local ring.
This proves the assertion for the ordinary local ring itself.

\proofstep{1. The smooth local case.}
We first work over a local ring essentially smooth over $\F_2$. The motivic representability result \cite[Corollary 8.1.8]{GWCHN}, using its Theorem 6.3.1, identifies evaluation of the motivic spectrum $KW$ on a regular smooth scheme $V$ with its homotopy-symmetric spectrum $L^{\mathrm{s}}(V)$. In particular let $\mathcal E$ be the pointed motivic space obtained by the right adjoint from motivic spectra to pointed motivic spaces, applied to the ordinary simplicial suspension $\Sigma^2KW$. Then
\[
 \pi_0\mathcal E(V)=\pi_0\Sigma^2L^{\mathrm{s}}(V)=L^{\mathrm{s}}_{-2}(V).
\]
Here $\Sigma^2$ denotes the ordinary simplicial suspension of the spectrum.

\proofstep{2. Generic injectivity on section spaces.}
Bachmann's local-section theorem \cite[Proposition 2.9]{GWBachmann} states that for every pointed motivic space $\mathcal E$ over a field $k$, a connected smooth $k$-scheme $V$, a point $v$, and the generic point $\xi$, the map
\[
 \pi_0\mathcal E(V_v)\longrightarrow\pi_0\mathcal E(\xi)
\]
has trivial fibre. Here $V_v$ denotes the ordinary local scheme. The proposition applies with $k=\F_2$. For a group-valued $\pi_0$, trivial fibre means injectivity. Continuity through ordinary open neighborhoods now proves the theorem for regular local rings essentially smooth over $\F_2$. Thus generic restriction is injective on the evaluation of the spectrum at each of these local schemes.

\proofstep{3. Smooth approximation.}
To pass to arbitrary $S$, the map $\F_2\to S$ is regular, since $\F_2$ is perfect and $S$ is regular. Popescu's theorem \cite[Tag 07GC]{Stacks} expresses $S$ as a filtered colimit of smooth $\F_2$-algebras. Localize the stages at the inverse image of the maximal ideal of $S$. We obtain a filtered system of regular local domains $B_i$, essentially smooth over $\F_2$, with local maps $B_i\to S$ and colimit $S$.

\proofstep{4. Descent of a vanishing class.}
Let $\alpha\in L^{\mathrm{s}}_{-2}(S)$ vanish in $L^{\mathrm{s}}_{-2}(F)$. Since $F$ is the filtered localization of $S$ at its nonzero elements, $\alpha$ vanishes over $S[1/a]$ for some nonzero $a\in S$. Descend together the class $\alpha$, the element $a$, and its vanishing after localization. At a sufficiently late stage this gives
\[
 \alpha_i\in L^{\mathrm{s}}_{-2}(B_i),\qquad a_i\in B_i,
 \qquad \alpha_i|_{B_i[1/a_i]}=0,
\]
with $a_i$ mapping to $a$. The element $a_i$ is nonzero because its image is nonzero. Since $B_i$ is a domain, $\alpha_i$ is generically zero, and the essentially smooth case gives $\alpha_i=0$. Hence $\alpha=0$.

\proofstep{5. Continuity of forms and cobordisms.}
The continuity used here is continuity of Poincar\'e categories and their $L$-theory, \cite[Corollary 4.4.5]{GWCDHII}. Perfect complexes, duality maps, and the finite cobordism data descend to a stage. For the homotopy-symmetric structure on a bounded perfect complex, any specified homotopy degree of the $C_2$-resolution uses only finitely many terms of its bounded bilinear complex; this verifies the required continuity of form and nullhomotopy data. Consequently the form and cobordism data descend along the filtered system.

\proofstep{6. The degree-minus-one group.}
Repeat the same argument with $\Sigma KW$ in place of $\Sigma^2KW$. It proves injection of $L^{\mathrm{s}}_{-1}(S)$ into $L^{\mathrm{s}}_{-1}(F)$, whose target is zero by odd-field vanishing. This proves the last assertion.
\end{proof}

\section{Global alternating Witt vanishing on the parameter open}

For a ring $A$, let $W_{\mathrm{alt}}(A)$ be the classical Witt group of nondegenerate alternating forms on finite projective $A$-modules, modulo hyperbolic forms. The distinction between alternating and merely skew-symmetric is retained at two.

\begin{theorem}[Alternating Witt vanishing on the parameter open]\label{gw:global_walt}
Let $R$ be a characteristic-zero regular local ring of dimension four and let $r$ belong to a regular system of parameters. Set $A=R[1/r]$. Then
\[
 L^{\mathrm{s}}_{-2}(A)=L^{\mathrm{gs}}_{-2}(A)=W_{\mathrm{alt}}(A)=0.
\]
In particular the conclusion holds in mixed characteristic $(0,p)$ for every residue prime $p$.
\end{theorem}

\begin{proof}
We first prove vanishing in homotopy-symmetric theory by localization
along the parameter divisor. We then compare the genuine and symmetric
theories on the open scheme. The shifted fundamental sequence identifies
the resulting genuine group with the alternating Witt group.

\proofstep{1. Symmetric divisor localization.}
Put $S=R/(r)$, a regular local domain. The regular immersion $i:\Spec S\hookrightarrow\Spec R$ has exceptional line
\[
 i^!R=\det(N_{S/R})[-1]\simeq S[-1],
\]
where the parameter $r$ fixes the normal-line trivialization. Homotopy-symmetric purity in \eqref{gw:exceptional_line} gives
\begin{equation}\label{gw:divisor_support_shift}
 L_n(\operatorname{Perf}_{V(r)}(R),Q_R^{\mathrm{s}}|)
       \simeq L^{\mathrm{s}}_{n+1}(S).
\end{equation}
In particular the localization boundary out of degree $-2$ lands in support degree $-3$, which becomes degree $-2$ on $S$. Its exact segment is
\begin{equation}\label{gw:divisor_segment}
 L^{\mathrm{s}}_{-2}(R)\longrightarrow L^{\mathrm{s}}_{-2}(A)
       \xrightarrow{\partial}L^{\mathrm{s}}_{-2}(S).
\end{equation}
This is the homotopy-symmetric localization map. One may use its Karoubi-localizing version, which agrees with ordinary theory here. Alternatively the ordinary Verdier quotient suffices: $K_0(R)\to K_0(A)$ is onto, since every projective $A$-module descends to a finite $R$-module and a finite free resolution over $R$ localizes to a free resolution over $A$.

\proofstep{2. The characteristic-two boundary.}
The left term of \eqref{gw:divisor_segment} is zero by Theorem~\ref{gw:local_negative}. If $2$ remains a nonzerodivisor in $S$, the same theorem kills the right term. This includes characteristic zero and every odd-characteristic possibility for $S$. Otherwise, since $S$ is a domain, it has characteristic two. Let $x\in L^{\mathrm{s}}_{-2}(A)$ and $\beta=\partial x$. Localize the divisor square at the height-one prime $(r)$. Its ambient DVR $V=R_{(r)}$ has fraction field $F=\Frac(R)$ and residue field $\kappa=\Frac(S)$. Naturality, with the same parameter $r$ orienting both exceptional lines, gives
\[
 \beta|_\kappa=\partial_V(x|_F).
\]
But $L^{\mathrm{s}}_{-2}(F)=0$, since $F$ has characteristic zero. Thus $\beta$ is generically zero on $S$; Theorem~\ref{gw:equalchar_injection} gives $\beta=0$. Exactness of \eqref{gw:divisor_segment} again yields $x=0$. Thus $L^{\mathrm{s}}_{-2}(A)=0$ in both cases.

\proofstep{3. The global comparison fibre.}
It remains to compare this with genuine theory globally. On $X=\Spec A$, let
\[
 \mathcal F_{\mathrm{cmp}}=\operatorname{fib}(L^{\mathrm{gs}}\longrightarrow L^{\mathrm{s}})
\]
be the fibre of the Zariski sheaves of spectra. Zariski descent for the Karoubi-localizing theories is \cite[Corollary 4.3.5]{GWCHN}; regularity identifies them with the ordinary theories. By filtered localization, the stalk at $x$ is the comparison fibre evaluated at $T=\mathcal O_{X,x}$. Every such $T$ has dimension at most three and is two-torsion-free. Proposition~\ref{gw:comparison_window} gives an isomorphism in all degrees at least zero and an injection in degree $-1$. The fibre sequence therefore gives
\[
 \pi_n\mathcal F_{\mathrm{cmp}}(T)=0\qquad(n\geq-1).
\]
Both symmetric groups in degrees $-1,-2$ vanish by Theorem~\ref{gw:local_negative}, and the three genuine groups in degrees $-1,-2,-3$ vanish by Proposition~\ref{gw:genuine_local_negative}. The next two pieces of the fibre sequence give $\pi_{-2}\mathcal F_{\mathrm{cmp}}(T)=\pi_{-3}\mathcal F_{\mathrm{cmp}}(T)=0$. Hence every homotopy sheaf of $\mathcal F_{\mathrm{cmp}}$ in degrees at least $-3$ is zero.

Now $\dim X\leq3$: the maximal ideal of $R$ is its only height-four prime and contains $r$. The finite-dimensional Zariski descent spectral sequence
\[
 H^p_{\Zar}(X,\pi_q\mathcal F_{\mathrm{cmp}})\Longrightarrow
       \pi_{q-p}\mathcal F_{\mathrm{cmp}}(X)
\]
has only $0\leq p\leq3$. For total degree at least $-3$, every relevant $q$ is at least $-3$, so all entries vanish. Finite cohomological dimension gives convergence and therefore $\pi_n\mathcal F_{\mathrm{cmp}}(X)=0$ for $n\geq-3$. The comparison $L^{\mathrm{gs}}_{-2}(A)\to L^{\mathrm{s}}_{-2}(A)$ is an isomorphism, proving its vanishing. Thus the local comparison range gives the global comparison in the
degree needed for alternating forms.

\proofstep{4. Identification of the alternating Witt group.}
Finally identify the genuine endpoint with the alternating Witt group. The finite-resolution argument above gives $K_0(A)=\Z$, with generator $[A]$: rank supplies the inverse. Shift-two duality acts trivially on this group. The $K$-theory spectrum is connective, so its homotopy-orbit spectral sequence has
\[
 \pi_0\bigl(K(A;\Ddual_A[2])_{hC_2}\bigr)=K_0(A)_{C_2}=K_0(A),
 \qquad
 \pi_{-1}\bigl(K(A;\Ddual_A[2])_{hC_2}\bigr)=0.
\]
The degree-zero edge of the shifted fibre sequence \eqref{gw:fundamental} is therefore
\[
 K_0(A)\xrightarrow{h}\GW_0^{[2]}(A)
        \longrightarrow L^{\mathrm{gs}}_{-2}(A)\longrightarrow0.
\]
Here $h$ sends a projective module to its alternating hyperbolic form, and $\GW_0^{[2]}$ is the classical symplectic Grothendieck--Witt group. Its cokernel is exactly $W_{\mathrm{alt}}(A)$. Thus $W_{\mathrm{alt}}(A)=L^{\mathrm{gs}}_{-2}(A)=0$.
\end{proof}

\section{Codimension-three lifting and the full Witt row}
\paragraph{The support-filtration consequence.}
The vanishing conclusions of Theorem~\ref{gw:codim3_lift} follow
from the supported lifting theorem. For
$\mathcal C=\operatorname{Perf}(R)^{\geq3}$ and $c=3$, that theorem gives
\[
 L_0(\mathcal C,Q_R^{\mathrm{gs}}[3])\twoheadrightarrow
 L_0(\mathcal C,Q_R^{\mathrm{s}}[3])=L^{\mathrm{s}}_{-3}(\mathcal C).
\]
Both sides map naturally to the corresponding ambient groups. The
genuine ambient group is $L^{\mathrm{gs}}_{-3}(R)=0$ by
Proposition~\ref{gw:genuine_local_negative}. Hence the image of
$L^{\mathrm{s}}_{-3}(\mathcal C)$ in $L^{\mathrm{s}}_{-3}(R)$ is zero; this image is
$\Fil^3_{\mathrm{con}}L^{\mathrm{s}}_{-3}(R)$.
In the dimension range stated below, the coniveau argument at the end
of the proof identifies $H^3(C_{\mathrm W}(R))$ with its corresponding limiting
graded piece. A graded piece of the zero filtration group is zero.
The theorem below gives the stronger representative statement. Its
proof is included because it describes how the genuine structure is
obtained on a finite module, while the preceding argument suffices for
the support-filtration consequence.

\begin{theorem}[Codimension-three lifting and Witt-row vanishing]\label{gw:codim3_lift}
Let $R$ be a characteristic-zero regular local ring of dimension at most four. Every class in
\[
 L^{\mathrm{s}}_{-3}(\operatorname{Perf}(R)^{\geq3},Q_R^{\mathrm{s}}|)
\]
has a representative whose underlying complex is a finite module in the ordinary heart, and whose symmetric structure lifts to the restricted ambient genuine structure. Consequently
\[
 \Fil^3_{\mathrm{con}} L^{\mathrm{s}}_{-3}(R)=0,\qquad H^3(C_{\mathrm W}(R))=0.
\]
Here $C_{\mathrm W}(R)$ is the full bilinear Witt row, with the determinant lines and coniveau differential of \eqref{gw:l_coniveau}.
\end{theorem}

\begin{proof}
\proofstep{1. A representative in the supported heart.}
Let $\mathcal C=\operatorname{Perf}(R)^{\geq3}$. Its ordinary bounded $t$-structure is well defined: regularity makes all finite homology modules perfect, and truncation preserves the support condition. Duality preserves the support condition by localization and biduality. For a heart module $M$,
\[
 \operatorname{Ext}^i_R(M,R)=0\quad\text{unless }3\leq i\leq4.
\]
With $\Ddual'=\Ddual_R [3]$ and $Q'=Q_R^{\mathrm{s}}[3]|_{\mathcal C}$, the finite homology tower gives
\[
 \Ddual'(\mathcal C_{\geq0})\subset\mathcal C_{\leq0},\qquad
 \Ddual'(\mathcal C_{\leq0})\subset\mathcal C_{\geq-1}.
\]
Apply \cite[Proposition 1.3.1]{GWCDHIII} to this bounded $t$-structured stable category. Choose $n=0$, duality width $d'=1$, $a=0$, $b=1$, and symmetry parameter $s\geq1$. The required inequalities are
\[
 a\geq d'-1,\quad b\geq d',\quad b\geq a,\quad n+a\text{ even},
 \quad a\geq-n+2d'-2s,
\]
and all hold. Since a homotopy-symmetric structure is $s$-symmetric for every $s$, the theorem gives a connective Poincar\'e representative of every class in $L_0(\mathcal C,Q')=L^{\mathrm{s}}_{-3}(\mathcal C)$. Its self-duality and the first inclusion make it also coconnective. It is therefore a finite module $M$ in the heart, with the required support. If $\dim R<3$, the category is zero; in dimension three the same width-one bounds remain valid.

\proofstep{2. The genuine lifting obstruction.}
We verify coherent genuine lifting on this heart representative. Put $N=\operatorname{Fr}_*(R/2)$ with its squared scalar action. Lemma~\ref{gw:frobenius_grade}, with $b=3$, gives
\[
 \operatorname{Hom}_R(M,N)=\operatorname{Ext}^1_R(M,N)=0.
\]
In elementary terms one may choose $a,b\in\operatorname{Ann}M$ regular successively on the Cohen--Macaulay ring $R/2$; their squares are regular on its Frobenius-twisted module, and the exact sequence for multiplication by $a$ kills both groups. If two is a unit, $N=0$.

\proofstep{3. Lifting the symmetric form.}
The top nonzero homotopy of $E_-[3]$ is $N$ in degree one; degree zero is zero and all other groups have degree at most $-1$. Thus
\[
 \tau_{\leq-1}(E_-[3])\longrightarrow E_-[3]\longrightarrow N[1]
\]
is a truncation triangle. Orthogonality kills degree-zero maps from $M$ to the first term. Therefore $\pi_0\operatorname{RHom}_R(M,E_-[3])$ injects into $\operatorname{Ext}^1_R(M,N)=0$. The comparison sequence \eqref{gw:linear_comparison}, shifted by three, lifts the form on $M$ to a genuine one; the unchanged bilinear part preserves nondegeneracy. The lifted form, together with the preceding surgery, gives the required supported Poincar\'e cobordism.

\proofstep{4. The image of the support group.}
The image of the support group in $L^{\mathrm{s}}_{-3}(R)$ is precisely $\Fil^3_{\mathrm{con}}$ by \eqref{gw:coniveau_quotient}. The genuine lift pushes to $L^{\mathrm{gs}}_{-3}(R)=0$ by Proposition~\ref{gw:genuine_local_negative}. Hence $\Fil^3_{\mathrm{con}}L^{\mathrm{s}}_{-3}(R)=0$.

\proofstep{5. The surviving coniveau term.}
To identify this with the desired row cohomology, use \eqref{gw:l_coniveau}. The term $H^3(C_{\mathrm W}(R))$ is $E_2^{3,0}$. Its only possible higher incoming arrows are
\[
 d_2:E_2^{1,1}\to E_2^{3,0},\qquad
 d_3:E_3^{0,2}\to E_3^{3,0}.
\]
The first source is zero by odd-field vanishing. The second is zero because its generic coefficient is $L^{\mathrm{s}}_{-2}(\Frac R)=0$. Incoming arrows of greater length have negative source column, and all higher outgoing arrows have column at least five, absent in dimension at most four. Thus $E_\infty^{3,0}=E_2^{3,0}$. In total degree $-3$ the only potential fourth filtration quotient is $E_\infty^{4,-1}$, an odd field row and therefore zero; there is no fifth filtration. Hence
\[
 H^3(C_{\mathrm W}(R))=\Fil^3_{\mathrm{con}}L^{\mathrm{s}}_{-3}(R)=0.
\]
The vanishing of this filtration piece proves the row-cohomology assertion.
\end{proof}

%% file: sections/03_motivic_local_comparisons.tex
\section{Ordinary local motivic comparisons and residue exactness}
\label{loc:section}
The goal of this section is the local exactness needed to compare the
Milnor--Witt field complex with its sheaf. The two principal results are
Theorems~\ref{loc:mw4_h2} and~\ref{loc:mw4_h1}: they give degree-two
exactness in local dimension at most three and degree-one exactness on
local surfaces. Lemma~\ref{loc:dvr_image} supplies the remaining
codimension-one image calculation.

We obtain these results by relating hermitian groups to weight-two
motivic cohomology. First we construct the local edge maps and prove
that the degree-one and degree-zero weight-two groups are unchanged
by generic restriction. A weight-three symbol surjection then gives
the required low-degree genuine hermitian comparison. Finally, a
finite-slice calculation describes the degree-four field cokernel;
its classes extend from the local ring, which allows us to replace
hermitian residues by Milnor--Witt residues.

Throughout this section, $T$ is a finite-dimensional Noetherian regular
local domain with fraction field $F=\Frac(T)$ of characteristic zero,
unless otherwise specified. We use the motivic complexes of
\cite{BouisMotivic,BouisWeibel}. The groups needed below are
$H^1_{\mot}(T,\Z(2))$ and $H^0_{\mot}(T,\Z(2))$.
The ring inclusion $T\to F$ induces their generic restriction maps.

\subsection{The local filtration and its operations}

We begin with a comparison of filtrations. Only a finite range is
needed: the lemma produces the weight-two edge in degree three and
identifies the groups from which the degree-four edge will be obtained.
Keeping the realized filtration separate from the motivic filtration
also allows us to compute the action of ordinary duality.

\begin{lemma}[Local realization and the low filtration window]
\label{loc:lisse_window}
There is a natural multiplicative filtration $\Fil^\bullet_{\mathrm{lis}}K(T)$ on $K(T)$,
with a filtered map to $\Fil^\bullet_{\mot}K(T)$, such that $\Fil^i_{\mathrm{lis}}K(T)$ is $i$-connective and, for
$i\geq2$,
\[
 \operatorname{cofib}(\Fil^i_{\mathrm{lis}}K(T)\longrightarrow \Fil^i_{\mot}K(T))
 \quad\hbox{is $(i-4)$-coconnective.}
\]
Consequently there is a natural exact sequence
\begin{equation}\label{loc:k3_window}
 H^3_{\mot}(T,\Z(3))\xrightarrow{j_T}K_3(T)
       \xrightarrow{q_{3,T}}H^1_{\mot}(T,\Z(2))\longrightarrow0.
\end{equation}
The first map sends products of unit classes to the same products in
Quillen $K$-theory. Over a field the sequence is the usual quotient by
the forward Milnor symbol map, so $H^1_{\mot}(F,\Z(2))=K_3^{\operatorname{ind}}(F)$.
Moreover
\[
 \pi_3\Fil^3_{\mathrm{lis}}K(T)\xrightarrow{\sim}\pi_3\Fil^3_{\mot}K(T),\qquad
 \pi_4\Fil^2_{\mathrm{lis}}K(T)=K_4(T),\qquad
 \pi_4\gr^2_{\mathrm{lis}}K(T)=H^0_{\mot}(T,\Z(2)).
\]
Here $a$-coconnective means that $\pi_n=0$ for $n>a$.
\end{lemma}
\begin{proof}
\proofstep{1. Realization of the local resolution.}
Use the augmented simplicial local resolution $P_\bullet\to T$ from
\cite[proof of Lemma 2.19]{BouisWeibel}. Its terms are local ind-smooth
$\Z$-algebras, its augmentations are local surjections, and its underlying
simplicial ring resolves $T$. The simplicial matrix space
$\operatorname{Mat}_n(P_\bullet)$ realizes to
$\operatorname{Mat}_n(T)$. Because the maps are local surjections, a
determinant is a unit precisely when its image in $T$ is a unit.
Thus $\GL_n(P_\bullet)$ is exactly the inverse image of the components
indexed by $\GL_n(T)$; the fibre over each such matrix is contractible. It
follows that $|\GL_n(P_\bullet)|\simeq\GL_n(T)$.

Finite projective modules over a local ring are free. Taking the
classifying spaces, disjoint union over ranks, and group completion of
the finite-free groupoids therefore gives, naturally,
\begin{equation}\label{loc:realization}
 |K^{\operatorname{conn}}(P_\bullet)|
       \simeq K^{\operatorname{conn}}(T)=K(T).
\end{equation}
The last equality follows from regularity of $T$.

\proofstep{2. The realized filtration.}
Set $\Fil^i_{\mathrm{lis}}K(T)=|\Fil^i_{\operatorname{cla}}K(P_\bullet)|$.
For a local essentially smooth $\Z$-algebra $B$, classical motivic
cohomology in weight $j$ vanishes above degree $j$
\cite[Corollary 4.4]{Geisser}. The classical filtration is eventually
connective on such a finite-dimensional $B$
\cite[Corollary 9.5; alternatively Lemma 4.8 and finite-dimensional
Nisnevich descent]{BEM}. For any fixed homotopy degree below $i$, start
in that eventually connective tail and use the filtration triangles
successively down to index $i$. Every intervening graded layer is
$j$-connective in index $j$. This proves $i$-connectivity of
$\Fil^i_{\operatorname{cla}}K(B)$. Filtered colimits and geometric
realizations preserve this bound, proving the assertion for $\Fil^i_{\mathrm{lis}}K(T)$.

\proofstep{3. Comparison of the graded layers.}
The classical-to-modern filtered map of
\cite[Definition 3.22]{BouisMotivic}, followed by the augmentation,
gives $\Fil^\bullet_{\mathrm{lis}}K(T)\to \Fil^\bullet_{\mot}K(T)$. Its underlying map is \eqref{loc:realization}. To estimate the
comparison on filtration stages, we first estimate its cofibre on each
graded layer. In index $j$ the graded map is the lisse comparison
\begin{equation}\label{loc:graded_comparison}
 \bigl(\tau^{\leq j}\Z(j)^{\mot}(T)\bigr)[2j]
       \longrightarrow\Z(j)^{\mot}(T)[2j],
\end{equation}
by \cite[Theorem 2.11, Corollary 2.12 and Lemma 2.19]{BouisWeibel}.
For $j\geq1$, the additional vanishing
$H^{j+1}_{\mot}(T,\Z(j))=0$ of Corollary 2.10 there says that the
cofibre in \eqref{loc:graded_comparison} is $(j-2)$-coconnective. At
$j=0$ the sufficient bound is $-1$.

\proofstep{4. The low finite quotients.}
Compare the finite quotients
$K_{\mathrm{lis}}^{<i}(T)=\operatorname{cofib}(\Fil^i_{\mathrm{lis}}K(T)\to K(T))$ and
$K_{\mot}^{<i}(T)=\operatorname{cofib}(\Fil^i_{\mot}K(T)\to K(T))$. Their graded indices are
$0,\ldots,i-1$. The two quotients have only finitely many graded layers. Applying the
preceding bounds successively to their triangles shows that
$\operatorname{cofib}(K_{\mathrm{lis}}^{<i}(T)\to K_{\mot}^{<i}(T))$ is
$(i-3)$-coconnective. The underlying middle map is an equivalence, so
the cofibre of the comparison on filtration stages is the loop of this
quotient cofibre. It is therefore $(i-4)$-coconnective, as required for
$\Fil^i_{\mathrm{lis}}K(T)\to \Fil^i_{\mot}K(T)$.
In particular
\[
 \begin{aligned}
 &\pi_2\Fil^3_{\mot}K(T)=0,\qquad
   \pi_3\Fil^4_{\mot}K(T)=\pi_2\Fil^4_{\mot}K(T)=0,\\
 &\pi_3\Fil^3_{\mathrm{lis}}K(T)
   \xrightarrow{\sim}\pi_3\Fil^3_{\mot}K(T).
 \end{aligned}
\]
The index-three triangle identifies $\pi_3\Fil^3_{\mot}K(T)$ with
$H^3_{\mot}(T,\Z(3))$. Weights zero and one have no relevant homotopy
in degrees three and four, so $\pi_3\Fil^2_{\mot}K(T)=K_3(T)$. The index-two triangle
gives \eqref{loc:k3_window}. Multiplicativity and the weight-one
identification with units give its stated symbol normalization.
The same weight-zero and weight-one calculation in degrees four and
five gives $\pi_4\Fil^2_{\mathrm{lis}}K(T)=K_4(T)$; the index-two comparison
\eqref{loc:graded_comparison} gives $\pi_4\gr^2_{\mathrm{lis}}K(T)=H^0_{\mot}(T,\Z(2))$.
Over a field, the classical diagonal group in weight three is
$\KM_3(F)$, proving the field assertion.
\end{proof}

Two operations will be used in the constancy arguments. Adams
operations separate the weights that occur in positive codimension,
while ordinary duality distinguishes the adjacent weights in the
integral edge sequence. We record their actions with the support
normalization used in the coniveau spectral sequence.

\begin{lemma}[Operations and support normalization]
\label{loc:operations}
Let $m\geq2$ be an integer invertible in $T$. The operation on
$K_n(T)[1/m]$ underlying the modern filtered Adams operation agrees
with the classical Adams operation $\psi^m$. On the weight-$j$ motivic
graded layer with $\Z[1/m]$ coefficients it acts integrally by $m^j$.
In particular, its action on $H^1_{\mot}(F,\Z(2))[1/m]$ is multiplication by $m^2$.
On the rational motivic summand of weight $j$ it likewise acts by $m^j$.
On a codimension-$c$ field summand of Quillen coniveau, the supported classical operation is
\begin{equation}\label{loc:support_adams}
             \psi^m_{\operatorname{supp}}=m^c\psi^m_{\operatorname{field}}.
\end{equation}
The ordinary duality $\tau$ on $K(T)$ preserves the realized classical
filtration $\Fil^\bullet_{\mathrm{lis}}K(T)$ and acts integrally by $(-1)^j$ on its
$j$th graded layer. In particular,
\begin{equation}\label{loc:duality_kernel}
               \tau|_{\ker q_{3,T}}=-1.
\end{equation}
\end{lemma}
\begin{proof}
\proofstep{1. Operations on the local resolution.}
Each term of the local resolution in Lemma~\ref{loc:lisse_window}
is a $\Z[1/m]$-algebra. On its local ind-smooth terms the classical and
modern operations coincide by
\cite[Definition 3.22, Corollary 4.7, Lemma 4.8 and Construction 4.9]{BouisMotivic}
and the classical normalization
\cite[Proposition 3.12 and Remark 3.13]{BachmannHopkins}; pass from
smooth terms by continuity. Both operations are natural pointed-space
maps. Realization and \eqref{loc:realization}, also at the level of the
underlying grouplike spaces, identify their induced operations on
$K_n(T)[1/m]$. The integral action on motivic graded layers is
\cite[Corollary 4.10]{BouisMotivic}; applying it to the weight-two
edge gives the asserted action on $H^1_{\mot}(F,\Z(2))[1/m]$.
The rational splitting and its weight action are
\cite[Theorem 4.1 and Corollary 4.60]{BouisMotivic}.

\proofstep{2. The local support group.}
Let $x$ have codimension $c$. The coniveau local-support
term is the filtered colimit of the groups with support in
$\overline{\{x\}}$ over neighborhoods of $x$. It equals $K$-theory
supported at the closed point of $\Spec\mathcal O_{T,x}$. Indeed,
write the prime of $x$ using finitely many generators. Continuity on
affine localizations applies both to perfect complexes and, using the
finite principal-open cover, to the complement of the support.
Taking the fibre defining relative $K$-theory commutes with that
filtered colimit.

\proofstep{3. The codimension factor.}
Now $\Spec\kappa(x)\to\Spec\mathcal O_{T,x}$ is a regular immersion
with trivial normal bundle of rank $c$. Soul\'e's supported
Riemann--Roch theorem, with this local scheme as the regular base,
gives \eqref{loc:support_adams}: its cannibalistic factor is
$\theta^m(\kappa(x)^c)=m^c$
\cite[\S4.5, Lemma 2(i), p.~515; \S4.6, Theorem 3, p.~517]{Soule}.
The supported operations act on the entire exact couple
\cite[Theorem 4(i), pp.~521--522]{Soule}. In particular the formula is
compatible with all coniveau differentials. The computation takes place on the regular local scheme
$\Spec\mathcal O_{T,x}$.

\proofstep{4. The action of ordinary duality.}
For duality we use the classical construction directly. In the
Snaith description, duality comes from inversion on $\mathbb G_m$,
and the operation $\psi^{-1}$ is induced by that same inversion
\cite[\S\S3.3.1--3.3.2]{BachmannHopkins}. The classical filtered
operation acts by $(-1)^j$ on weight $j$, integrally:
\cite[Appendix B, Lemma B.2 and Remark B.4]{LOCElmantoMorrow} allow
all nonzero integer indices and general quasi-compact, quasi-separated bases, so they apply in
particular to smooth $\Z$-schemes at index $-1$.
Realize this operation on $P_\bullet$. The finite-free construction
of \eqref{loc:realization} commutes with the duality functor, hence the
underlying realized operation is $\tau$ on $K(T)$.
As $\Fil^4_{\mathrm{lis}}K(T)$ is four-connective,
$\pi_3\Fil^3_{\mathrm{lis}}K(T)=\pi_3\gr^3_{\mathrm{lis}}K(T)$, on which $\tau=-1$.
Its image in $K_3(T)$ is $\ker q_{3,T}$ by
Lemma~\ref{loc:lisse_window}, proving \eqref{loc:duality_kernel}.
The resulting duality identity is integral.
\end{proof}

\subsection{The two weight-two groups}

\begin{lemma}[Field divisibility in degree zero]
\label{loc:field_d_divisible}
For every characteristic-zero field $F$ and every prime $\ell$,
$H^0_{\mot}(F,\Z(2))/\ell=0$.
\end{lemma}
\begin{proof}
The coefficient triangle and field finite-coefficient comparison give
\[
 0\longrightarrow H^0_{\mot}(F,\Z(2))/\ell
 \longrightarrow H^0_{\mathrm{\acute et}}(F,\mu_\ell^{\otimes2})
 \longrightarrow H^1_{\mot}(F,\Z(2))[\ell]\longrightarrow0.
\]
Levine's torsion theorem
\cite[Corollary 4.6, p.~331]{Levine} identifies the $\ell$-primary
torsion of $H^1_{\mot}(F,\Z(2))=K_3^{\operatorname{ind}}(F)$ with
$H^0_{\et}(F,(\Q_\ell/\Z_\ell)(2))$. Its subgroup killed by $\ell$ therefore
has the same finite order as
$H^0_{\et}(F,\mu_\ell^{\otimes2})$. Both are of order one or $\ell$.
The displayed surjection between groups of equal order is an
isomorphism. This includes $\ell=2$, because $F$ has characteristic
zero. Thus multiplication by $\ell$ is surjective on $H^0_{\mot}(F,\Z(2))$.
\end{proof}

The next theorem has both an injectivity and a surjectivity part.
For injectivity, finite coefficients remove torsion from the kernel
and the rational weight decomposition removes what remains. For
surjectivity, the constant-field theorem gives divisibility of the
cokernel, whereas supported Adams operations give a finite annihilator
after a coefficient localization.

\begin{theorem}[Generic constancy of weight-two degree one]
\label{loc:g_constancy}
For every $T$ in the standing hypotheses, the restriction
\[
                       H^1_{\mot}(T,\Z(2))\xrightarrow{\sim}H^1_{\mot}(F,\Z(2))
\]
is an isomorphism.
\end{theorem}
\begin{proof}
\proofstep{1. Torsion in the generic kernel.}
We first prove injectivity. Put $J=\ker(H^1_{\mot}(T,\Z(2))\to H^1_{\mot}(F,\Z(2)))$. If $z\in J$
is killed by a prime $\ell$, the coefficient sequence gives
$a\in H^0_{\mot}(T,\Z/\ell(2))$ with Bockstein $z$.
Its generic restriction has zero Bockstein and hence belongs to
$H^0_{\mot}(F,\Z(2))/\ell=0$ by Lemma~\ref{loc:field_d_divisible}.
The regular characteristic-zero domain $T$ is flat over $\Z$ and
is $\ell$-torsion-free and $F$-smooth in the sense of \cite{BouisMotivic}. Thus
\cite[Corollary 5.12]{BouisMotivic}, used only in degree zero in
weight two, identifies this coefficient group with
\[
 H^0_{\mathrm{\acute et}}(T[1/\ell],\mu_\ell^{\otimes2}).
\]
Its restriction to the generic field is injective: sections of a
finite \'etale sheaf on the normal integral scheme $T[1/\ell]$ are
determined generically. Hence $a=0$ and $z=0$. This proves that $J$ has no nonzero element killed by a prime. Every
nonzero torsion group contains such an element, obtained by multiplying
an element of finite order by a suitable integer. Hence $J$ is
torsion-free.

\proofstep{2. The rational generic kernel.}
We next prove $J\otimes\Q=0$. Choose an integer $m\geq2$ invertible
in $T$. The rational motivic splitting identifies $H^1_{\mot}(T,\Z(2))_\Q$ with
the weight-two summand of $K_3(T)_\Q$, compatibly with restriction
and with the classical $\psi^m$ by Lemma~\ref{loc:operations}.
The positive coniveau part
$N=\ker(K_3(T)_\Q\to K_3(F)_\Q)$ has a finite filtration whose
codimension-$c$ piece is a subquotient of sums of field $K_3$ groups.
Field $K_3\otimes\Q$ has only weights two and three; this follows
from the classical field motivic filtration and vanishing above the
Milnor range. By \eqref{loc:support_adams}, the supported weights
are $c+2,c+3$. Consequently, if $d=\dim T$, the polynomial
\[
 P(t)=\prod_{c=1}^{d}(t-m^{c+2})(t-m^{c+3})
\]
annihilates $N$. Repeated factors in this expression are retained.
The scalar $P(m^2)$ is nonzero. Since $P(\psi^m)$ annihilates $N$,
its restriction to a weight-two vector is multiplication by this
nonzero rational scalar. Such a vector must therefore be zero.
This proves rational injectivity. Consequently every element of $J$
is torsion; the preceding step proved that $J$ is torsion-free, so
$J=0$.

\proofstep{3. Divisibility of the cokernel.}
For surjectivity, use Lemma~\ref{loc:lisse_window} and put
\[
 C=H^1_{\mot}(F,\Z(2))/\im(K_3(T)\longrightarrow H^1_{\mot}(F,\Z(2))).
\]
We first show that $C$ is divisible. Represent a class over a finitely
generated subfield $E\subset F$ containing $\Q$, and let $k$ be the
algebraic closure of $\Q$ in $E$, a number field. Levine's
finite-level constant-field theorem
\cite[Theorem 4.13, p.~336]{Levine} makes
$H^1_{\mot}(k,\Z(2))/\ell\to H^1_{\mot}(E,\Z(2))/\ell$ surjective for every prime $\ell$.

\proofstep{4. Extension of constant classes.}
Every such constant class comes from $K_3(T)$. If the residue
characteristic is zero, normality gives the inclusion of $k$ in $T$. If the residue
characteristic is $p>0$, normality first gives the inclusion of the ring of integers
$\mathcal O_k$ in $T$. With
$\mathfrak p=\mathcal O_k\cap\mathfrak m_T$, locality gives the inclusion
$\mathcal O_{k,\mathfrak p}$ in $T$. The localization sequence and
$K_2(\kappa(\mathfrak p))=0$ for the finite residue field
\cite{QuillenFinite} make
$K_3(\mathcal O_{k,\mathfrak p})\to K_3(k)$ surjective.
Lift the constant class along $K_3(k)\to H^1_{\mot}(k,\Z(2))$, and then lift it from this
arithmetic DVR. Therefore $C=\ell C$ for every $\ell$.

\proofstep{5. A finite annihilator for the cokernel.}
Fix a unit integer $m\geq2$ and work with $\Z[1/m]$ coefficients.
For $g\in H^1_{\mot}(F,\Z(2))[1/m]$, lift to $x\in K_3(F)[1/m]$. In Quillen
coniveau its possible length-$s$ outgoing differential has target
\[
 E_s^{s,-s-2},
\]
a subquotient of codimension-$s$ field $K_2$. Matsumoto's theorem
and \eqref{loc:support_adams} give the scalar $m^{s+2}$ on this
target. Apply successively, for $s=1,\ldots,d$, the operators
$\psi^m-m^{s+2}$. The $s$th operator kills the $s$th obstruction of
the class already surviving to that page. No differential enters
the generic column, so each class remains a literal generic
$K_3$ class. After $d$ steps it is permanent and hence, by the
finite support filtration, lies in the image of $K_3(T)[1/m]$.
Its image in $H^1_{\mot}(F,\Z(2))[1/m]$ is
\[
 N_mg,\qquad N_m=\prod_{s=1}^{d}(m^2-m^{s+2})\ne0
 \quad(d>0).
\]
Thus $N_m C[1/m]=0$. The group $C[1/m]$ is divisible because $C$ is.
Given an element of this group, divide it by the nonzero integer
$N_m$; the displayed annihilation then shows that the original
element is zero. Hence $C[1/m]=0$.

\proofstep{6. Descent from coefficient localizations.}
There are two coprime integers $m,n\geq2$ invertible in $T$:
in residue characteristic $p>0$ one may use
$m=p+1$, $n=p(p+1)+1$, and in residue characteristic zero use $2,3$.
Every element of $C$ is killed by a power of each, hence is zero
by B\'ezout's identity. Therefore $C=0$. For $d=0$ the map was the
identity throughout. Surjectivity of $q_{3,T}$ finishes the proof.
\end{proof}

\begin{theorem}[Degree-zero constancy]
\label{loc:d_constancy}
The groups $H^0_{\mot}(T,\Z(2))$ and $H^0_{\mot}(F,\Z(2))$ are uniquely divisible, and the restriction $H^0_{\mot}(T,\Z(2))\to H^0_{\mot}(F,\Z(2))$ is an isomorphism.
\end{theorem}
\begin{proof}
We first prove unique divisibility using the finite-coefficient
comparison. This reduces the generic comparison to a rational
weight-two calculation in degree four. The supported weights then
exclude both a generic kernel and outgoing differentials from the
required weight-two classes.

For every prime $\ell$, Corollary 5.12 of \cite{BouisMotivic}
in degree $-1$ gives
$H^{-1}_{\mot}(T,\Z/\ell(2))=0$; the coefficient sequence makes
$H^0_{\mot}(T,\Z(2))[\ell]=0$. The same sequence embeds $H^0_{\mot}(T,\Z(2))/\ell$ in
$H^0_{\mot}(T,\Z/\ell(2))$. Restriction on this latter group is
injective by the degree-zero \'etale description used in
Theorem~\ref{loc:g_constancy}. The image of $H^0_{\mot}(T,\Z(2))/\ell$ restricts
into $H^0_{\mot}(F,\Z(2))/\ell=0$. Thus $H^0_{\mot}(T,\Z(2))/\ell=0$. This proves unique
divisibility, also for $T=F$.

The rational motivic splitting therefore identifies $H^0_{\mot}(T,\Z(2))$ itself
with $K_4(T)_\Q^{(2)}$ and likewise for $F$. Field $K_4\otimes\Q$
has weights two, three and four. The finite positive coniveau
filtration of the generic kernel has in codimension $c\geq1$
only supported weights $c+2,c+3,c+4$, by
Lemma~\ref{loc:operations}. The same finite-polynomial argument
as above excludes weight two from the kernel.

For surjectivity take a weight-two class in $K_4(F)_\Q$. A
length-$s$ outgoing coniveau differential lands in
$E_s^{s,-s-3}$, a subquotient of codimension-$s$ field $K_3$ with
weights $s+2,s+3$. Neither is two, so equivariance makes every
such differential zero. The generic class is permanent and lifts
to $K_4(T)_\Q$. Project the lift to its natural weight-two
summand; this preserves its generic restriction. Hence restriction
is onto on weight two, proving the assertion for $H^0_{\mot}(-,\Z(2))$.
\end{proof}

Constancy of the degree-zero group provides local motivic classes,
but the later hyperbolic construction requires lifts in integral
Quillen theory. The next proposition supplies them. Its proof uses
the opposite duality signs on the two sides of the connecting map.

\begin{proposition}[The integral degree-four edge]
\label{loc:k4_edge}
The motivic edge
\[
                     q_{4,T}:K_4(T)\longrightarrow H^0_{\mot}(T,\Z(2))
\]
is surjective.
\end{proposition}
\begin{proof}
The index-two triangle for $\Fil^\bullet_{\mathrm{lis}}K(T)$ and
Lemma~\ref{loc:lisse_window} give
\[
 K_4(T)\longrightarrow H^0_{\mot}(T,\Z(2))
       \xrightarrow{\partial}\pi_3\Fil^3_{\mathrm{lis}}K(T).
\]
Classical duality acts by $+1$ on the middle group and by $-1$
on the last group, by Lemma~\ref{loc:operations}. Equivariance of the connecting map identifies its negative with
itself, so $2\partial=0$. By Theorem~\ref{loc:d_constancy}, every
source element can be divided by two. Applying $\partial$ to such
a half shows that its value on the original element is zero.
Therefore $\partial=0$. The comparison
$\Fil^\bullet_{\mathrm{lis}}K(T)\to \Fil^\bullet_{\mot}K(T)$ identifies this map with the motivic edge, proving its surjectivity.
\end{proof}

\subsection{The needed ordinary-local unit symbols modulo two}

The genuine comparison below will leave a kernel that is killed by
two. To make that kernel vanish, we need to divide its classes while
remaining inside the kernel. For this purpose it is enough to represent
weight-three motivic classes by unit symbols modulo two. We first prove
the representation statement on a smooth local algebra.

\begin{lemma}[Smooth-local weight-three image modulo two]
\label{loc:smooth_modtwo}
Let $P$ be a local essentially smooth $\Z_{(2)}$-algebra, with
fraction field $E$. Then
\begin{equation}\label{loc:smooth_image}
 \im\bigl(\KM_3(P)/2\to\KM_3(E)/2\bigr)
 =\ker\left(\KM_3(E)/2\longrightarrow
       \bigoplus_{\mathfrak q\in(\Spec P)^{(1)}}K_2(\kappa(\mathfrak q))/2
       \right),
\end{equation}
and the ordinary unit-symbol map
\[
        \KM_3(P)/2\longrightarrow H^3_{\mot}(P,\Z(3))/2
\]
is surjective.
\end{lemma}
\begin{proof}
\proofstep{1. The integral residue complex.}
Use the integral complex
\[
 \KM_3(E)\xrightarrow{d}
 \bigoplus_{P^{(1)}}K_2(\kappa(\mathfrak q))
 \xrightarrow{d}\bigoplus_{P^{(2)}}\kappa(z)^\times
 \xrightarrow{d}\bigoplus_{P^{(3)}}\Z.
\]
Here and below a superscript on $P$ abbreviates the corresponding
stratum of $\Spec P$. Its positive differentials are the Quillen support differentials, expressed using Matsumoto's theorem.
L\"uders's integral relative Gersten theorem
\cite[Theorem 1.1 and Corollary 4.3]{LOCLudersRelative}
provides both positive exactness and image exactness at the generic
term for improved Milnor theory in this smooth-over-excellent-DVR
setting. The surjection from ordinary unit-symbol Milnor theory to
improved theory \cite[Theorem 4.3.2]{KerzThesis} gives the
same generic image from ordinary $\KM_3(P)$.

\proofstep{2. An integral lift of a mod-two class.}
We prove \eqref{loc:smooth_image} by lifting a mod-two cocycle and
correcting its integral boundary. Lift a class in its right side to $x\in\KM_3(E)$ and write
$dx=2b$. Then $db$ is killed by two. Its component in a
characteristic-two field is trivial, and its component in a
characteristic-zero field is $1$ or $-1$. On
$U=\Spec P[1/2]$ these latter signs define a finite codimension-two
cycle modulo two.

\proofstep{3. The codimension-two Chow group.}
The correction of the boundary in the previous step uses
$\operatorname{CH}^2(P)=0$. We recall its coniveau proof. In Quillen coniveau,
$E_2^{2,-2}=\operatorname{CH}^2(P)$. Its limiting term is zero
because $K_0(P)=\Z$, detected by rank. No higher differential leaves
it, since the corresponding field $K$-degrees are negative; its only
possible higher incoming differential is $d_2$ from unramified
$K_1$. Normality identifies that generic kernel with $P^\times$,
whose classes extend to $K_1(P)$ and are permanent. Thus the
$E_2$ term is already zero. Chow localization then gives $\operatorname{CH}^2(U)=0$.

\proofstep{4. Correction of the coefficient boundary.}
Choose a finite divisor cochain
$a=(a_{\mathfrak q})\in\bigoplus_{U^{(1)}}\kappa(\mathfrak q)^\times$
whose divisor has the prescribed parity cycle, and form the
$K_2$ cochain on $\Spec P$
\[
 t_{\mathfrak q}=\{-1,a_{\mathfrak q}\}
       \quad(\mathfrak q\in U^{(1)}),\qquad
 t_{\mathfrak q}=0\quad(2\in\mathfrak q).
\]
It is killed by two and satisfies $dt=db$. To check this with the
differentials, use the product formula for localization:
at a characteristic-zero specialization the residue of
$\{-1,a\}$ is $(-1)^{v(a)}$, followed by the norm; at a
characteristic-two specialization the global unit $-1$ specializes
to $1$, so the product formula gives zero. This uses the supported
boundary and projection formula even when the closure of a point
is singular.

\proofstep{5. The corrected unramified class.}
Now $b-t$ is an integral cocycle. Integral positive exactness gives
$y\in\KM_3(E)$ with $dy=b-t$. Since $2t=0$, the class $x-2y$ has
zero integral residue. Integral image exactness represents it by
ordinary $P$-unit symbols, proving \eqref{loc:smooth_image}. The
reverse inclusion follows from vanishing of residues of unit symbols.

\proofstep{6. The smooth-local motivic comparison.}
For the motivic assertion use Geisser's finite-coefficient
\emph{Zariski} Gersten resolution
\cite[Corollary 4.5, pp.~786--787]{Geisser}, on a smooth scheme
over the ordinary DVR $\Z_{(2)}$, in degree and weight three with
modulus two. At the local stalk it identifies the classical
$H^3(P,\Z/2(3))$ with the generic kernel in
\eqref{loc:smooth_image}. The residue is the tame symbol, as in
the proof of Corollary 4.4 there, including the characteristic-two
field calculation. Geisser's Proposition 3.6 identifies the
Zariski and Nisnevich complexes through the change-of-topology
map, also after the coefficient triangle. Thus this stalk group
is the classical motivic group used in the lisse comparison.

The classical-to-modern comparison is an isomorphism through
cohomological degree three on this smooth-local ring
\cite[Corollary 2.12]{BouisWeibel}. Classical degree-four
vanishing is Geisser's Corollary 4.4, and modern degree-four
vanishing is \cite[Corollary 2.10]{BouisWeibel}.
Since both degree-four integral groups vanish, the coefficient triangle
identifies the degree-three finite-coefficient groups with the integral
quotients modulo two. The comparison is compatible with these
identifications and respects unit products. Combining it with \eqref{loc:smooth_image}
proves the required surjection.
\end{proof}

\begin{proposition}[Weight-three unit-symbol surjectivity]
\label{loc:weight_three_units}
For every $T$ in the standing hypotheses,
\begin{equation}\label{loc:weight_three_symbol}
       \KM_3(T)/2\longrightarrow H^3_{\mot}(T,\Z(3))/2
\end{equation}
is surjective, where $\KM_3(T)$ is the ordinary unit-symbol group.
\end{proposition}
\begin{proof}
Suppose first that the closed residue characteristic is two.
Choose the local surjection $P\to T$ from a local
ind-smooth $\Z$-algebra constructed in
\cite[proof of Lemma 2.9; also Lemma 2.19]{BouisWeibel},
for which $(P,\ker(P\to T))$ is a henselian pair.
Corollary 2.8 of that source makes
\[
 H^3_{\mot}(P,\Z/2(3))\longrightarrow
 H^3_{\mot}(T,\Z/2(3))
\]
surjective. The local ring $P$ is a filtered colimit of local
essentially smooth $\Z_{(2)}$-algebras. Lemma~\ref{loc:smooth_modtwo}
and continuity therefore represent every source class by $P$-unit
symbols. Corollary 2.10 of
\cite{BouisWeibel} identifies the target with
$H^3_{\mot}(T,\Z(3))/2$. The local surjection $P\to T$ carries these representatives to unit
symbols of $T$.

If the closed residue characteristic is an odd prime $p$, take the
excellent base DVR $V=\Z_{(p)}$. If it is zero, take
$V=\Z_{(3)}$. In either case $T$ is a local $V$-algebra and two
is prime to the residue characteristic of $V$. By \cite[Corollary 3.8 and its proof]{Luders},
for an excellent DVR $V$, any local $V$-algebra $B$, and an
integer $n$ invertible in its base residue field, the map from
ordinary unit symbols to
$H^j_{\mathrm{\acute et}}(B,\mu_n^{\otimes j})$ is onto. The theorem
first uses improved Milnor theory, and its final paragraph passes
to ordinary symbols by their surjection onto improved theory.
Apply it with $B=T$, $j=3$, $n=2$.
The finite-coefficient comparison
\cite[Corollary 5.6]{BouisMotivic}, valid where two is invertible,
and the local degree-four vanishing
\cite[Corollary 2.10]{BouisWeibel} identify its target with that
of \eqref{loc:weight_three_symbol}, preserving unit symbols.
This proves the remaining cases.
\end{proof}

\begin{remark}\label{loc:symbol_scope}
The weight-three symbol surjection is obtained by two local
constructions. In residue characteristic two, a henselian-surjective
ind-smooth presentation reduces the argument to
Lemma~\ref{loc:smooth_modtwo}. When two is invertible in $T$, the local Galois-symbol surjectivity
of \cite[Corollary 3.8 and its proof]{Luders}, together with the
finite-coefficient motivic comparison, gives the surjection.
\end{remark}

\subsection{Genuine Grothendieck--Witt groups in low degrees}

\begin{lemma}[The local Bott calculations]
\label{loc:bott}
For every $T$ in the standing hypotheses one has
\[
 \GW_2^{[3]}(T)=0,\qquad \GW_1^{[3]}(T)=T^\times,
\]
and a natural short exact sequence
\begin{equation}\label{loc:bott_k2}
 0\longrightarrow K_2(T)\xrightarrow{h^{[4]}}\GW_2^{[4]}(T)
       \longrightarrow T^\times[2]\longrightarrow0.
\end{equation}
There is also a natural right-exact sequence
\begin{equation}\label{loc:bott_k3}
 \GW_3^{[3]}(T)\xrightarrow{f}K_3(T)
       \xrightarrow{h^{[4]}}\GW_3^{[4]}(T)\longrightarrow0.
\end{equation}
\end{lemma}
\begin{proof}
We use the genuine Bott sequence in the convention
\[
 \GW_i^{[j]}\xrightarrow{f}K_i\xrightarrow{h^{[j+1]}}
 \GW_i^{[j+1]}\longrightarrow\GW_{i-1}^{[j]}.
\]
Local symplectic elementary generation gives
$\GW_1^{[2]}=K_1^{\mathrm{Sp}}=0$. Also $\GW_0^{[2]}=\Z$, by half-rank,
and its forgetful map to $K_0=\Z$ is multiplication by two.
It follows that $h^{[3]}:K_1\to\GW_1^{[3]}$ is an isomorphism.

The group $K_2(T)$ is generated by ordinary unit symbols:
use the natural all-local comparison with improved Milnor
$K_2$ and its ordinary-unit surjection
\cite[Propositions 4.2.8(3) and 4.3.3]{KerzThesis}.
A product of the genuine degree-one unit classes in shift one
maps under $f:\GW_2^{[2]}\to K_2$ to that same symbol.
Thus $f$ is surjective; together with $\GW_1^{[2]}=0$, this
implies $\GW_2^{[3]}=0$ and proves
\eqref{loc:bott_k3}.

The remaining Bott segment begins with an injection
$K_2\to\GW_2^{[4]}$ and has cokernel the kernel of
$f:\GW_1^{[3]}\to K_1$. Under the above hyperbolic
identification, this forgetful map is $1-\tau$ on $K_1$.
Duality inverts units, so it is the squaring map on $T^\times$.
Its kernel is $T^\times[2]$, proving \eqref{loc:bott_k2}.
\end{proof}

The Bott sequence expresses the genuine group as a quotient of
Quillen degree three. The motivic edge gives another quotient of that
group. We shall compare them by first showing that the motivic edge
annihilates the forgetful image, and then proving that the induced
kernel is both two-divisible and killed by two.

\begin{theorem}[The genuine weight-two comparison]
\label{loc:gw3_shift4}
There is a natural isomorphism
\[
 \phi_T:\GW_3^{[4]}(T)\xrightarrow{\sim}H^1_{\mot}(T,\Z(2)),
 \qquad \phi_T h^{[4]}=q_{3,T}.
\]
It identifies generic restriction with the isomorphism of
Theorem~\ref{loc:g_constancy}. In particular
$\GW_3^{[4]}(T)\to\GW_3^{[4]}(F)$ is an isomorphism.
\end{theorem}
\begin{proof}
Over $F$, the forgetful image of $\GW_3^{[3]}(F)$ is exactly
the forward $\KM_3(F)$ image, by the degree-three field
symbol identification of Theorem~\ref{gw:field_symbols}.
Consequently $q_{3,T}f$ restricts to zero in $H^1_{\mot}(F,\Z(2))$. Generic
injectivity in Theorem~\ref{loc:g_constancy} makes $q_{3,T}f=0$.
Sequence \eqref{loc:bott_k3} therefore induces a unique
surjection $\phi_T$ with the stated normalization. At $F$ it
is the usual indecomposable quotient isomorphism. Uniqueness
makes all these maps natural.

To control the kernel we need the sign with which duality acts under
the hyperbolic map. For ordinary duality $\tau$,
\begin{equation}\label{loc:hyperbolic_sign}
                         h^{[4]}\tau=h^{[4]}.
\end{equation}
Indeed a Poincar\'e shift by four induces $(-1)^4\tau=\tau$
on $K$-theory. The refined hyperbolic map factors through
homotopy orbits for this shifted duality, giving
\eqref{loc:hyperbolic_sign}; see
\cite[Corollary 7.4.18]{GWCDHI} and
\cite[Proposition 3.4.12 and \S3.6, after Lemma 3.6.4]{GWCDHII}.
The sign is positive because the target shift is four.

Put $J=\ker\phi_T$. For $a\in J$, choose $y\in K_3(T)$ with
$h^{[4]}(y)=a$. Then $q_{3,T}(y)=0$, so
\eqref{loc:k3_window} gives $\zeta_0\in H^3_{\mot}(T,\Z(3))$
with $j_T(\zeta_0)=y$. Proposition~\ref{loc:weight_three_units}
gives $u\in\KM_3(T)$ and $\zeta\in H^3_{\mot}(T,\Z(3))$
such that $\zeta_0=\operatorname{symbol}(u)+2\zeta$.
The Quillen image of the unit-symbol term is the forgetful image
of a product in $\GW_3^{[3]}(T)$, hence is killed by
$h^{[4]}$. Therefore
\[
                   a=2h^{[4]}j_T(\zeta).
\]
The element obtained before multiplication by two still belongs to
$J$, because $q_{3,T}j_T=0$. Thus $J=2J$. We next show that
multiplication by two kills $J$. The identities
\eqref{loc:duality_kernel} and \eqref{loc:hyperbolic_sign} give
\[
 a=h^{[4]}(y)=h^{[4]}(\tau y)=h^{[4]}(-y)=-a.
\]
Hence $2J=0$, and $J=0$. The generic assertion now follows
from the natural comparison with $H^1_{\mot}(-,\Z(2))$.
\end{proof}

\begin{remark}
The isomorphism in Theorem~\ref{loc:gw3_shift4} is induced by
the hyperbolic quotient map in \eqref{loc:bott_k3}. Its compatibility
with generic restriction follows from the naturality of that map
and of the weight-two motivic edge.
\end{remark}

\begin{proposition}[Generic injectivity for the symmetric degree-two group]
\label{loc:ko2_injection}
If $\dim T\leq3$, then
\[
 \GW_2^{[4]}(T)\xrightarrow{\sim}\KO_2^{[4]}(T)
       \lhook\joinrel\longrightarrow\KO_2^{[4]}(F).
\]
\end{proposition}
\begin{proof}
We first prove generic injectivity for the genuine group from its
short exact Bott sequence. We then use the comparison fibre to obtain
surjectivity from genuine to symmetric theory. Generic injectivity
will supply injectivity of that comparison as well.

Use the generic diagram of \eqref{loc:bott_k2}.
Theorem~\ref{k2:thm:main} makes $K_2(T)\to K_2(F)$
injective, and $T^\times[2]\to F^\times[2]$ is injective.
A diagram chase therefore makes
$\GW_2^{[4]}(T)\to\GW_2^{[4]}(F)$ injective.

Put $\mathcal F_{\mathrm{cmp}}(T)=\operatorname{fib}(L^{\mathrm{gs}}(T)
\to L^{\mathrm{s}}(T))$. The local negative-group
vanishings of Theorem~\ref{gw:local_negative} give
$L^{\mathrm{s}}_{-2}(T)=0$
and $L^{\mathrm{gs}}_{-3}(T)=0$. The exact sequence
\[
 L^{\mathrm{gs}}_{-2}\longrightarrow L^{\mathrm{s}}_{-2}
 \longrightarrow \pi_{-3}\mathcal F_{\mathrm{cmp}}(T)\longrightarrow L^{\mathrm{gs}}_{-3}
\]
gives $\pi_{-3}\mathcal F_{\mathrm{cmp}}(T)=0$. The genuine and homotopy-symmetric fundamental
squares have the same $K$-theory-with-duality term, so their
comparison fibre in shift four is $\mathcal F_{\mathrm{cmp}}(T)[4]$. Its first homotopy
group is $\pi_{-3}\mathcal F_{\mathrm{cmp}}(T)$; hence
$\GW_2^{[4]}(T)\to\KO_2^{[4]}(T)$ is onto.

If a genuine class maps to zero under this comparison, its
restriction to $F$ is zero, since two is invertible in $F$ and
comparison there is an equivalence. Genuine generic injectivity
already proved then makes that class zero over $T$. Comparison
is thus also injective, and the displayed assertion follows.
\end{proof}

\subsection{The characteristic-zero degree-four cokernel and its edge}

The field comparison in degree four is used through its cokernel.
We compute that cokernel and identify how the hyperbolic map reaches
it. The identity between the two edge maps is essential: together with
the local degree-four edge, it will show that every cokernel class
extends from the local ring.

\begin{proposition}[A finite-slice cokernel calculation]
\label{loc:field_cokernel}
For every characteristic-zero field $F$ there is a natural exact
sequence
\begin{equation}\label{loc:field_cokernel_sequence}
 \KMW_4(F)\xrightarrow{\mu_4}\KO_4^{[4]}(F)
        \xrightarrow{c_F}H^0_{\mot}(F,\Z(2))\longrightarrow0.
\end{equation}
With the usual Bott orientation and symplectic half-rank convention,
\begin{equation}\label{loc:hyperbolic_edge}
                        c_Fh^{[4]}=q_{4,F}.
\end{equation}
Changing the common Bott orientation changes both sides by the
same sign.
\end{proposition}
\begin{proof}
The proof has two parts. A finite sequence of very-effective cover
triangles gives the exact sequence. We then place the hyperbolic map
in the corresponding source and target triangles and compute its map
on the coefficient spectrum. The symplectic half-rank convention fixes
that map as the identity.

\proofstep{1. The relevant effective slices.}
All the motivic spectra in this proof are over $F$. This is a
perfect field with two invertible, so the published very-effective
slice calculation of \cite[Theorem 16]{LOCBachmannSlices} applies.
Write $\pi_{a,b}E=[S^{a,b},E]$, put $\mathbb T=S^{2,1}$, and
write $\widetilde f_q$ for very-effective covers and
$\widetilde s_q$ for their cofibres. Let
$E_0=\widetilde s_0KQ$. The calculation gives
\begin{align*}
 \widetilde s_{-4}KQ&=\Sigma^{-8,-4}E_0,&
 \widetilde s_{-3}KQ&=\Sigma^{-6,-3}\mathrm H\Z/2,\\
 \widetilde s_{-2}KQ&=\Sigma^{-4,-2}\mathrm H\Z,&
 \widetilde s_{-1}KQ&=0.
\end{align*}
The same theorem supplies the triangle
\[
       \Sigma^{1,0}\mathrm H\Z/2\longrightarrow E_0
                     \longrightarrow\widetilde{\mathrm H}\Z,
\]
where $\widetilde{\mathrm H}\Z$ is the effective cover of the
Eilenberg--Mac Lane spectrum of the Milnor--Witt homotopy module.
It follows that $\pi_{4,0}E_0=\pi_{5,0}E_0=0$: the
$\mathrm H\Z/2$ terms are negative-degree weight-zero motivic
cohomology, and effective-cover adjunction computes the other
terms in a homotopy-module spectrum, with zero positive
homotopy groups.

\proofstep{2. The cokernel sequence.}
The sphere $S^{-4,-4}=\mathbb T^{-4}\wedge S^{4,0}$ belongs
to very-effective degree $-4$. Adjunction therefore identifies
\[
 \pi_{-4,-4}\widetilde f_{-4}KQ
       =\pi_{-4,-4}KQ=\KO_4^{[4]}(F).
\]
The adjacent groups of $\widetilde s_{-4}KQ$ are the two
groups of $E_0$ shown above to vanish. Those of
$\widetilde s_{-3}KQ$ are
$H^{-3}_{\mot}(F,\Z/2(1))$ and $H^{-2}_{\mot}(F,\Z/2(1))$, also zero.
The two finite cover triangles thus identify
\[
             \pi_{-4,-4}\widetilde f_{-2}KQ
                       =\KO_4^{[4]}(F).
\]
Since $\widetilde s_{-1}=0$, the remaining relevant triangle is
\[
 \widetilde f_0KQ\longrightarrow\widetilde f_{-2}KQ
                   \longrightarrow\Sigma^{-4,-2}\mathrm H\Z.
\]
Its homotopy sequence gives
\begin{equation}\label{loc:finite_slice_edge}
 \pi_{-4,-4}\widetilde f_0KQ\longrightarrow\KO_4^{[4]}(F)
       \longrightarrow H^0_{\mot}(F,\Z(2))\longrightarrow0.
\end{equation}
Here $\pi_{-5,-4}\widetilde f_0KQ=0$, since a
very-effective spectrum is connective in the homotopy
$t$-structure and $-5-(-4)=-1$.

\proofstep{3. The Milnor--Witt map.}
The unit induces an isomorphism
$\KMW_4(F)\simeq\pi_{-4,-4}\widetilde f_0KQ$.
Indeed \cite[Theorem 3.30(1)]{BachmannHopkins} says that the
unit from the motivic sphere to $KQ$ is an equivalence on
$\widetilde s_0$. The covers $\widetilde f_1$ are
one-connective in the homotopy $t$-structure, so the diagonal
and its adjacent preceding groups identify through this
zero-slice comparison. For the sphere the diagonal group is
Morel's $\KMW_4(F)$. This unit is the product-symbol
map $\mu_4$: it agrees in degree one with the chosen canonical
unit class and is multiplicative. In positive degrees the
presentation expresses all classes as sums of products of
degree-one units, by eliminating factors
$\eta[a][b]=[ab]-[a]-[b]$. Our convention sends the Milnor--Witt generator $\eta$ to
$-\eta_{\mathrm g}$ in Grothendieck--Witt theory. Thus \eqref{loc:finite_slice_edge} proves
\eqref{loc:field_cokernel_sequence} with its map.
This identifies both the first group and its map in the finite cover
sequence. It remains to compare the last arrow with the integral
weight-two edge.

\proofstep{4. The shifted hyperbolic map.}
To identify the hyperbolic edge, consider the Wood map of motivic spectra
\[
                        h:KGL\longrightarrow\mathbb T\wedge KQ.
\]
Evaluate it at $\pi_{-2,-3}$. Its source is
$K_{-2-2(-3)}(F)=K_4(F)$; its target is
$\pi_{-4,-4}KQ=\KO_4^{[4]}(F)$. This is the map $h^{[4]}$.

\proofstep{5. The two coefficient spectra.}
The relevant source slice is
$\widetilde s_{-1}KGL=\Sigma^{-2,-1}\mathrm H\Z$, whose
$\pi_{-2,-3}$ is $H^0_{\mot}(F,\Z(2))$. Bott translation by
three Tate suspensions identifies its edge with the ordinary
weight-two $K_4$ edge $q_{4,F}$. The earlier slices have the
required adjacent vanishing groups in negative degrees of
weights zero and one. The corresponding target slice is
\[
 \widetilde s_{-1}(\mathbb T\wedge KQ)
 =\mathbb T\wedge\widetilde s_{-2}KQ
 =\Sigma^{-2,-1}\mathrm H\Z.
\]
Its edge is $c_F$, constructed in \eqref{loc:finite_slice_edge}.
The hyperbolic map induces a map between these source and target
cover triangles. Once its map between the displayed coefficient
spectra is computed, naturality will compare the two edges.

\proofstep{6. Normalization of the hyperbolic coefficient.}
After one untwisting and the Bott identification for $KGL$,
the map on this $\mathrm H\Z$ coefficient is the degree-zero map
of the symplectic hyperbolic functor
\[
 h:\pi_0(KGL)_0=\Z\longrightarrow
            \pi_0(KO^{[2]})_0=\Z.
\]
It is an isomorphism, explicitly used in the proof of
\cite[Theorem 16, pp.~19--20 of the author manuscript]{LOCBachmannSlices}.
A rank-one module maps to the alternating hyperbolic plane,
whose half-rank is one. Lemma 13 there identifies the
corresponding effective Eilenberg--Mac Lane coefficient
with $\mathrm H\Z$. The degree-zero coefficient consequently
fixes the map on this $\mathrm H\Z$ slice as $+1$ in the chosen
orientation. This proves \eqref{loc:hyperbolic_edge} with coefficient $+1$.
\end{proof}

We now use the local results. Since the degree-zero motivic group is
unchanged by generic restriction and is reached by the integral local
edge, the field cokernel can be represented by classes defined over
the ring. Subtracting these classes will preserve residues.

\begin{corollary}[Extension of the entire field cokernel]
\label{loc:residue_images}
Every class in $\coker(\mu_4:\KMW_4(F)\to\KO_4^{[4]}(F))$
has a representative which is the restriction of a hyperbolic class over $T$. Consequently the generic residue
images of $\KMW_4(F)$ and $\KO_4^{[4]}(F)$ in
\[
          \bigoplus_{\mathfrak p\in(\Spec T)^{(1)}}
          \KMW_3(\kappa(\mathfrak p),\omega_{\mathfrak p})
\]
are equal.
\end{corollary}
\begin{proof}
Identify the cokernel with $H^0_{\mot}(F,\Z(2))$ using
Proposition~\ref{loc:field_cokernel}. Extend a class uniquely
to $H^0_{\mot}(T,\Z(2))$ by Theorem~\ref{loc:d_constancy}, lift it to
$K_4(T)$ by Proposition~\ref{loc:k4_edge}, and apply the
genuine hyperbolic map followed by comparison to $\KO$.
Equation~\eqref{loc:hyperbolic_edge} gives the prescribed
generic cokernel class. Its coniveau boundary is
zero because it extends over $T$.

Given an arbitrary generic $\KO_4^{[4]}$ class, subtract
such an extending representative of its cokernel class.
The difference belongs to $\im\mu_4$ and has the same
boundary. The reverse containment of boundary images is
immediate from the symbol factorization. This proves the equality of residue images.
\end{proof}

\subsection{The compatible Milnor--Witt row in degrees one and two}

We assemble the comparisons into the field complex used later.
The positive-codimension terms are identified by the symbol isomorphisms
through degree three. At the generic term, the preceding corollary
supplies the equality of residue images needed in place of a
degree-four isomorphism.

Let $Y$ be a regular Noetherian scheme of dimension at most
three with characteristic-zero generic fields. At $x\in Y^{(c)}$
write
\[
 \omega_x=\operatorname{Ext}^{c}_{\mathcal O_{Y,x}}
                   (\kappa(x),\mathcal O_{Y,x}).
\]
These are the determinant lines supplied by purity. The
homotopy-symmetric support spectral sequence is
\begin{equation}\label{loc:ko_coniveau}
 E_1^{c,q}=\bigoplus_{x\in Y^{(c)}}
       \KO_{-c-q}^{[4-c]}(\kappa(x),\omega_x)
       \ \Longrightarrow\ \KO_{-c-q}^{[4]}(Y),
 \qquad d_s:(c,q)\longmapsto(c+s,q-s+1).
\end{equation}
This is the finite support tower of
\cite[Proposition 4.5.1, Theorem 5.2.1, and Remark 5.1.10]{GWCHN}.
Purity is computed at the closed point of the regular
local ring $\mathcal O_{Y,x}$; its shift is $[4-c]$ because
$i^!\mathcal O$ is $\omega_x[-c]$. The comparison between
ordinary and Karoubi variants on these support categories is
\cite[Remark 5.2.4]{GWCHN}, using their bounded Noetherian
hearts. Only the finite columns $0\leq c\leq\dim Y$ occur.

\begin{definition}[The compatible weight-four Milnor--Witt complex]\label{loc:mw_row}
The compatible weight-four complex $C_{\mathrm{MW},4}(Y)$
has terms
\[
 C^c_{\mathrm{MW},4}(Y)=
 \bigoplus_{x\in Y^{(c)}}\KMW_{4-c}(\kappa(x),\omega_x).
\]
For $c\geq1$ its differentials are those of the $q=-4$ row
of \eqref{loc:ko_coniveau}, transported through the field isomorphisms $\mu_3,\mu_2,\mu_1$ of
Theorem~\ref{gw:field_symbols}. Its first differential is
the generic map $\mu_4$ followed by that row's boundary.
A coefficient line on $Y$ is included by tensoring the
purity lines with its fibres.
\end{definition}

This is a complex with the support differentials.
Its first arrow is the Milnor--Witt DVR residue. To verify the
normalization, unit classes extending over the DVR have zero
localization boundary, and the product with the uniformizer
has residue the corresponding residue-field unit product with
coefficient one, by the DVR formula \eqref{gw:dvr_coefficient_one}. The presentation calculation in
Lemma~\ref{loc:dvr_image} below generates every generic class
from these two kinds. The positive field symbol maps identify
the resulting boundaries, with their determinant lines.
The Milnor and Witt projections therefore commute with the
Quillen and Witt boundaries.

In the next two theorems, $H^i(C_{\mathrm{MW},4}(T))$ is cohomology
of the field cochain complex, with degree given by codimension.
The comparison with the cohomology of the naive Zariski sheaf is
made in Theorem~\ref{global:naive_mw4}, using stalkwise support bounds.

\begin{theorem}[Local degree-two exactness]
\label{loc:mw4_h2}
If $T$ is in the standing hypotheses and $\dim T\leq3$, then
\[
                         H^2(C_{\mathrm{MW},4}(T))=0.
\]
This holds with any invertible coefficient line.
\end{theorem}
\begin{proof}
We place the desired cohomology group in the hermitian coniveau
spectral sequence. There is only one possible higher incoming
differential. The degree-three constancy theorem makes its source
consist of extending classes, so that differential is zero. Generic
injectivity of the abutment then kills the limiting term.

The three positive-codimension terms of the $q=-4$ row in
\eqref{loc:ko_coniveau} are
\[
 \bigoplus_{T^{(1)}}\KO_3^{[3]}
 \longrightarrow\bigoplus_{T^{(2)}}\KO_2^{[2]}
 \longrightarrow\bigoplus_{T^{(3)}}\KO_1^{[1]}.
\]
Their identification with $\KMW_3,\KMW_2,\KMW_1$ gives
\[
                   B:=H^2(C_{\mathrm{MW},4}(T))=E_2^{2,-4}.
\]
Changing the generic term cannot change this cohomology group.
In dimensions at most three, no higher differential leaves
$(2,-4)$. Its only possible higher incoming differential is
\begin{equation}\label{loc:h2_differential}
                        d_2:E_2^{0,-3}\longrightarrow E_2^{2,-4}.
\end{equation}
For $s\geq3$ an incoming source would have negative first
coordinate; there is no incoming $d_3$.

The generic group in the source row is
\[
             \KO_3^{[4]}(F)=\GW_3^{[4]}(F)=H^1_{\mot}(F,\Z(2)).
\]
Theorems~\ref{loc:gw3_shift4} and \ref{loc:g_constancy}
show that every one of these generic classes extends to a
genuine class over $T$, and hence to a $\KO_3^{[4]}(T)$ class. Each is therefore a permanent cycle
of the support tower. In particular
$E_2^{0,-3}$ is the whole generic group and
\eqref{loc:h2_differential} is zero.

By Proposition~\ref{loc:ko2_injection},
$\KO_2^{[4]}(T)\to\KO_2^{[4]}(F)$ is injective. Its
positive support filtration is consequently zero, so
$E_\infty^{2,-4}=0$. Since the only possible higher
incoming differential has just been shown to be zero and no
higher differential leaves $B$, one has
$B=E_\infty^{2,-4}=0$.
An invertible line on the local ring is trivializable; the
argument applied after such a trivialization proves the
twisted assertion.
\end{proof}

\begin{theorem}[Ordinary local surface degree-one exactness]
\label{loc:mw4_h1}
If $T$ is a regular local surface in the standing hypotheses, then
\[
                         H^1(C_{\mathrm{MW},4}(T))=0.
\]
The vanishing also holds with coefficients in any invertible line.
\end{theorem}
\begin{proof}
The genuine-to-homotopy-symmetric comparison is an
isomorphism at homotopy degree three and shift four over a
regular local surface: its difference is $3-4=-1$, the
surface endpoint of Proposition~\ref{gw:comparison_window}.
Combined with Theorem~\ref{loc:gw3_shift4}, it makes
$\KO_3^{[4]}(T)\to\KO_3^{[4]}(F)$ injective.

The $q=-4$ row on this surface is
\[
 \KO_4^{[4]}(F)\longrightarrow
 \bigoplus_{T^{(1)}}\KMW_3(\kappa(\mathfrak p),\omega_{\mathfrak p})
 \longrightarrow\KMW_2(\kappa(\mathfrak m_T),\omega_{\mathfrak m_T}).
\]
Its degree-one cohomology is $E_2^{1,-4}$. Dimension two
excludes every higher incoming and outgoing differential
at this position. It is a positive associated graded piece
of $\KO_3^{[4]}(T)$ and is zero by the generic injection.
Every degree-one cocycle is therefore the boundary of a class in
$\KO_4^{[4]}(F)$.

Corollary~\ref{loc:residue_images} identifies the image of this boundary
map with the image from $\KMW_4(F)$. We may thus choose a
Milnor--Witt class with the same boundary.
This proves the assertion, with all differentials
and residue lines unchanged. Trivializing a coefficient
line over $T$ proves the last statement.
\end{proof}

\begin{remark}\label{loc:no_weight_four_symbols}
The surface calculation combines the degree-four field cokernel
sequence with constancy of $H^0_{\mot}(-,\Z(2))$ and the surjective integral $K_4$
edge. An extending hyperbolic class represents each cokernel class;
subtracting it leaves a Milnor--Witt class with the prescribed residue.
This is the mechanism by which the field comparison gives local
surface exactness.
\end{remark}

\subsection{DVR image exactness for the naive presentation}

\begin{lemma}[All-characteristic naive DVR image exactness]
\label{loc:dvr_image}
Let $V$ be any DVR, with fraction field $E$, residue field $k$,
uniformizer $\pi$, and invertible coefficient module $L$.
Write $\omega=\operatorname{Ext}^1_V(k,V)$ for the exceptional
line at its closed point.
In every degree $n$ the sequence
\begin{equation}\label{loc:dvr_sequence}
 K_n^{\mathrm{MW},\operatorname{naive}}(V,L)
 \longrightarrow\KMW_n(E,L_E)
 \xrightarrow{\partial}\KMW_{n-1}(k,\omega\otimes L_k)
 \longrightarrow0
\end{equation}
is exact at its two right-hand groups.
\end{lemma}
\begin{proof}
After trivializing the coefficient line, we quotient the generic
group by the image of the ring's unit symbols. Every class in this
quotient can be written using one uniformizer factor. The main point
is that its other factors depend only on their residues; the resulting
map from residue-field symbols is inverse to the residue map.

\proofstep{1. The DVR presentation and residue map.}
We prove the image-exactness statement of
\cite[Theorem 2.5.9(1), with Remarks 2.2.8 and 2.4.5]{DegliseMW}
using the unit-symbol presentation; compare
\cite[Theorem 3.22]{LOCMorelBook}.

\proofstep{2. The unit-generated quotient.}
Trivialize $L$ over $V$. Let $U_*$ be the subring of
$\KMW_*(E)$ generated by $\eta$ and the symbols $[u]$
for $u\in V^\times$. It is exactly the image of the
naive presentation over $V$. Let $Q_* =\KMW_*(E)/U_*$
as a graded abelian group. The presentation identities
\[
 [ab]=[a]+[b]+\eta[a][b],\qquad [a]^2=[a][-1],
 \qquad \langle a\rangle=1+\eta[a]
\]
and the inverse-symbol identity reduce every class in
$Q_*$ to sums of symbols having exactly one $[\pi]$
factor and otherwise unit entries, with possible powers
of $\eta$. This follows by writing $a=\pi^m u$ and using
the product and inverse formulas; repeated $[\pi]$
factors reduce by $[\pi]^2=[\pi][-1]$.

\proofstep{3. Independence of unit lifts.}
To show that the unit factors depend only on their residues, it
suffices to prove the principal-unit identity
\begin{equation}\label{loc:principal_unit_identity}
                    [\beta][\pi]\in U_*
                    \quad\hbox{for }\beta\in1+\pi V.
\end{equation}
If $\beta\ne1$, write $\beta=1-\pi^nv$ with
$n\geq1$ and $v\in V^\times$. For $n=1$, the
Steinberg relation for $\pi v$ gives
\[
             [\beta][\pi]\langle v\rangle=-[\beta][v].
\]
The right side belongs to $U_*$ and $\langle v\rangle$
is an invertible element of $U_0$. For $n\geq2$, set
\[
                 w=\frac{\pi^{n-1}v-1}{1-\pi}\in V^\times.
\]
Then $\beta=(1-\pi)(1-\pi w)$. Expand
$[\beta][\pi]$ using the product formula. The
$[1-\pi][\pi]$ term is zero by Steinberg. In the
$\eta$ cross term, graded commutativity and
$\eta(-\langle-1\rangle)=\eta$ move $[1-\pi]$ beside $[\pi]$,
again giving zero. The remaining term is the case
$n=1$ already proved. This establishes \eqref{loc:principal_unit_identity}. The argument
uses only the presentation identities and units in $V$, so it also
applies when $k=\mathbb F_2$.

\proofstep{4. The quotient and the residue map.}
The same identity, the power formulas, and graded
commutativity imply that multiplication on $Q_*$ by a
unit lift depends only on its residue. Together with
$\eta$, these operations satisfy the Milnor--Witt
relations over $k$. The product and $\eta$ relations
descend directly; for the Steinberg relation choose
lifts $u$ and $1-u$ when both residue entries are
nonzero. Thus, putting $[\pi]$ first and using the
corresponding right action, multiplication of its class
by residue symbols defines a surjective degree-shifted
map $\KMW_*(k)\to Q_*$. Residue sends $[\pi]$ to
$1$ and is its inverse. Hence its kernel in
$\KMW_*(E)$ is exactly $U_*$, as required.

\proofstep{5. Coefficients in an invertible line.}
In particular the coefficient-one formula is
\[
 \partial\bigl(\eta^a[\pi][u_1]\cdots[u_b]\bigr)
       =\eta^a[\bar u_1]\cdots[\bar u_b],
\]
with the fixed ordering and uniformizer orientation.
Every residue symbol lifts by this formula. Changing
uniformizer or line trivialization changes both sides
by the standard square-class action; consequently the
statement descends to the displayed determinant-line
and coefficient-line twists.
\end{proof}

\begin{corollary}[The local inputs on the original parameter open]
\label{loc:original_open}
Let $R$ be a regular local mixed-characteristic domain of
dimension four and let $r$ be a nonzero regular parameter.
Put $A=R[1/r]$ and $X=\Spec A$. At every point of $X$:
\begin{enumerate}
 \item the local ring has characteristic-zero fraction field
       and dimension at most three;
 \item its compatible weight-four row has zero degree-two
       cohomology;
 \item at a height-two point that row has zero degree-one
       cohomology;
 \item at a height-one point the naive local-to-generic
       map has image equal to the kernel of the DVR residue,
       and that residue map is surjective.
\end{enumerate}
The ordinary Zariski sheaf associated to either $H^1_{\mot}(-,\Z(2))$ or genuine
$\GW_3^{[4]}$ is naturally the generic-point sheaf
$i_{\xi*}H^1_{\mot}(\Frac R,\Z(2))$, where $\xi$ is the generic point, and is therefore flasque.
\end{corollary}
\begin{proof}
The maximal ideal of $R$ contains $r$ and is absent from
$X$, so all its remaining local rings have dimension at
most three. These local rings are the regular localizations $R_{\mathfrak p}$. Their fraction
field is $\Frac R$, of characteristic zero. The four
assertions are respectively this observation,
Theorem~\ref{loc:mw4_h2}, Theorem~\ref{loc:mw4_h1},
and Lemma~\ref{loc:dvr_image}. In dimension one the
latter lemma also gives degree-one row exactness; in
dimension zero the positive row is empty.

Motivic cohomology and the genuine groups in use commute
with these filtered localizations. Hence their Zariski
stalks are the groups just computed.
Theorems~\ref{loc:g_constancy} and \ref{loc:gw3_shift4}
identify their natural maps to the generic-point sheaf
on every stalk. This proves the sheaf identifications.
Every nonempty open of the integral scheme contains its
generic point, so the restriction maps of that sheaf
are surjective, proving flasqueness.
\end{proof}

These local results give codimension bounds for the augmentation
kernel, its cokernel, and the positive cohomology sheaves on the
small Zariski site of $X$. Part~IV uses those bounds to identify
the degree-three cohomology of the naive Milnor--Witt sheaf with
that of its field complex.

%% file: sections/04_global_mw4.tex
\section{Global Witt, Chow, and Milnor--Witt vanishing}\label{global:section}
We now prove the first vanishing in \eqref{main:arithmetic_endpoints}.
The argument begins with Chow groups and the full Witt row. We then
improve Witt cochains successively by correcting their rank parities
and discriminants. These corrections remove the two-torsion in the
Milnor cohomology group. Finally, we make the Milnor and Witt
preimages compatible in the field pullback, and pass from the field
complex to the naive sheaf by support bounds. Let $R$ and $r$ satisfy the hypotheses of Theorem~\ref{main:quillen}, and put $A=R[1/r]$ and $S=R/(r)$. The ring $S$ is regular local of dimension three; it may have characteristic zero or positive characteristic. The arguments below retain the exceptional lines and the mixed-residue maps supplied by the supported hermitian theory.

\subsection{Field complexes and their normalizations}

For a regular scheme $Y$ write $C_{\mathrm M,n}(Y)$ for the field complex whose codimension-$p$ term is
\[
 C^p_{\mathrm M,n}(Y)=\bigoplus_{x\in Y^{(p)}}K^{\mathrm M}_{n-p}(\kappa(x)),
\]
with negative Milnor degrees zero. In the low degrees used here its maps are the Quillen-supported maps, identified with tame symbols, norms and divisors by the forward field symbol comparison. Write $C_{\mathrm W}(Y)$ for the full Witt row of homotopy-symmetric coniveau, and $C_{I,j}(Y)$ for the filtered row whose codimension-$p$ coefficient is $I^{j-p}(\kappa(x),\omega_x)$, with $I^a=W$ for $a\leq0$. Finally,
\[
 C^p_{\mathrm{MW},4}(Y)=
   \bigoplus_{x\in Y^{(p)}}K_{4-p}^{\mathrm{MW}}(\kappa(x),\omega_x).
\]
Here $\omega_x$ is the exceptional-duality line defined in Section~\ref{gw:supports}. All twists use the same supported duality and the same parameter orientation. The degree-three through degree-zero field comparisons in Theorem~\ref{gw:field_symbols} identify the positive-codimension terms relevant to $R$ and $A$ with the shifted hermitian support row.

By Lemma~\ref{gw:low_residues}, rank parity, the discriminant
$I/I^2\simeq K_1^{\mathrm M}/2$, and
$e_2:I^2/I^3\simeq K_2^{\mathrm M}/2$ commute with the corresponding
supported boundaries. We use these three residue squares below.

The regular divisor has short exact field-complex sequences
\begin{equation}\label{global:divisor_complex}
 0\longrightarrow C_{\mathrm M,n-1}(S)[-1]
  \longrightarrow C_{\mathrm M,n}(R)
  \longrightarrow C_{\mathrm M,n}(A)\longrightarrow0,
\end{equation}
and the analogous full-Witt sequence. The twist on the divisor is the exceptional line trivialized by $r$. The sequences are induced by localization of the support categories.

\subsection{The elementary Chow terms}

For any regular local ring $T$ used below, $\CH^2(T)=0$. Indeed $\CH^2(T)$ occupies $(2,-2)$ in the Quillen coniveau sequence
\begin{equation}\label{global:quillen_ss}
 E_1^{p,q}=\bigoplus_{x\in(\Spec T)^{(p)}}K_{-p-q}(\kappa(x))
       \Longrightarrow K_{-p-q}(T),
 \qquad d_s:(p,q)\longmapsto(p+s,q-s+1).
\end{equation}
Its only possible incoming higher map is $d_2$ from the unramified $K_1$ edge. That edge consists of $T^\times$ and extends to $K_1(T)$, so its classes are permanent. Outgoing terms have negative field $K$-degree. The limiting positive filtration in $K_0(T)=\Z$ is zero, giving the assertion. In particular
\begin{equation}\label{global:ch2}
 \CH^2(R)=\CH^2(S)=\CH^2(A)=0,
\end{equation}
where the assertion for $A$ follows by localization.

A regular parameter DVR curve through the closed point of a regular local ring supplies every terminal divisor coefficient. Consequently $\CH^4(R)=0$ and $\CH^3(S)=0$: the valuation of a uniformizer is one. The calculation takes place on the regular parameter curve.

There is also an integral bound
\begin{equation}\label{global:ch3_bound}
 2\CH^3(R)=0.
\end{equation}
To see this, $E_2^{3,-3}=\CH^3(R)$ has zero limiting piece in $K_0(R)$. The possible $d_3$ source is the extending unit edge; hence $d_2:E_2^{1,-2}\to\CH^3(R)$ is onto. On its source the integral supported Adams weight is two, whereas the target has weight three. The supported Adams operations of Part~II are defined for every positive integer $k$. Thus the target is killed by $k^3-k^2$ for every positive $k$. The values $4$ and $18$ at $k=2,3$ have greatest common divisor two, proving \eqref{global:ch3_bound}.

\subsection{Witt cochains and the third Chow group}

The codimension-three lifting theorem, Theorem~\ref{gw:codim3_lift}, gives
\begin{equation}\label{global:witt3_R}
 H^3(C_{\mathrm W}(R))=0.
\end{equation}
The divisor sequence gives $H^3(C_{\mathrm W}(R))\to H^3(C_{\mathrm W}(A))\to H^3(C_{\mathrm W}(S))$. Its rightmost group is a terminal residue cokernel on the regular local three-dimensional ring $S$. A regular parameter DVR curve lifts each diagonal residue form with coefficient one, so this cokernel vanishes in every characteristic. Therefore
\begin{equation}\label{global:witt3_A}
 H^3(C_{\mathrm W}(A))=0.
\end{equation}

For $Y=\Spec R$ or $\Spec A$, we first improve a full-Witt preimage
so that its codimension-two coefficients lie in $I^1$. Let $\alpha$
be a degree-three $C_{I,3}(Y)$ cocycle. By
\eqref{global:witt3_R}--\eqref{global:witt3_A}, choose a full-Witt
cochain $\beta$ whose boundary is $\alpha$. Its codimension-two rank parities form a cycle modulo two in the terminal term of $C_{\mathrm M,2}(Y)$. By \eqref{global:ch2}, choose a codimension-one unit cochain whose divisor has these parities. Lift each unit to a one-fold Pfister form. The rank/discriminant boundary square shows that its Witt boundary has the required parities. Subtracting this boundary from $\beta$ makes its coefficients
$I^1$-valued. Its boundary remains $d\beta$, since the differential
squares to zero. Hence
\begin{equation}\label{global:i3}
 H^3(C_{I,3}(Y))=0.
\end{equation}

We next show $\CH^3(R)/2=0$. The preceding filtered vanishing will
provide a preimage, but first we must lift the given cycle to a Witt
cocycle. Lift a degree-three cycle modulo two of $C_{\mathrm M,3}(R)$
to a Witt degree-three cochain $\alpha$ with the same rank parities. Its Witt boundary at the closed point need not initially vanish. The terminal map from the codimension-three $I^1$ terms to the closed-point Witt group is onto: on a regular parameter DVR curve, the form $\langle -u\rangle\langle1,-\pi\rangle=\langle -u,\pi u\rangle$ lies in $I^1$ and has residue $\langle\bar u\rangle$ with the fixed line orientation. Choose an $I^1$ cochain $\gamma$ with $d\gamma=d\alpha$. Then $\alpha-\gamma$ is a full-Witt cocycle with unchanged rank parities.

By \eqref{global:i3}, it is the boundary of a codimension-two
$I^1$ cochain. The discriminants of this cochain give a
codimension-two cochain of unit square classes. The low-degree boundary square identifies its divisor modulo two with the original degree-three cycle. Thus every such cycle is a boundary. Together with \eqref{global:ch3_bound} this proves
\begin{equation}\label{global:ch3}
 \CH^3(R)=\CH^3(A)=0.
\end{equation}
The second equality follows by Chow localization.

\subsection{Low-degree Milnor cohomology}

We now turn to the integral Milnor complex. By
Theorem~\ref{k2:thm:main}, every class in the unramified $K_2$ edge
of \eqref{global:quillen_ss} extends from the local ring. Such
classes are permanent cycles. We use this observation to prove
\begin{equation}\label{global:milnor_h2}
 H^2(C_{\mathrm M,3}(R))=
 H^2(C_{\mathrm M,3}(S))=
 H^2(C_{\mathrm M,3}(A))=0.
\end{equation}
For $R$, the group in question is $E_2^{2,-3}$. Its only higher
incoming differential is from the extending $K_2$ edge, and is
therefore zero. Its only possible outgoing target is a subquotient
of $\CH^4(R)=0$. Its limiting positive filtration in $K_1(R)=R^\times$ is zero. For $S$ there is no outgoing target. The assertion for $A$ follows from \eqref{global:divisor_complex}, because the following group is $H^2(C_{\mathrm M,2}(S))=\CH^2(S)=0$. Only field $K_2,K_1,K_0$ occur here, so Matsumoto's degree-two comparison suffices.

Set $B_Y=H^3(C_{\mathrm M,4}(Y))$. Then
\begin{equation}\label{global:milnor_bound}
 B_R\simeq B_A,\qquad 2B_R=2B_A=0.
\end{equation}
Indeed $B_R=E_2^{3,-4}$. The $d_3$ source is the extending unramified $K_2$ edge; the limit is zero in positive-coniveau $K_1(R)$. Hence $d_2:E_2^{1,-3}\to B_R$ is onto. Its source is a subquotient of codimension-one field $K_2$ cochains, of supported weight three; its target is a subquotient of codimension-three field $K_1$ cochains, of supported weight four. Consequently $k^3(k-1)B_R=0$. At $k=2,3$ the annihilators are $8$ and $54$, again with greatest common divisor two. The divisor sequence gives
\[
 H^2(C_{\mathrm M,3}(S))\longrightarrow B_R\longrightarrow B_A
                    \longrightarrow H^3(C_{\mathrm M,3}(S)),
\]
whose outer groups vanish by \eqref{global:milnor_h2} and $\CH^3(S)=0$. This proves the isomorphism.

To strengthen this exponent-two bound to vanishing, it remains to prove
that every class of $B_A$ is divisible by two. The filtered Witt
corrections below provide that divisibility.

We will also use
\begin{equation}\label{global:modtwo_h2}
 H^2(C_{\mathrm M,3}(Y)/2)=0\qquad(Y=\Spec R,\Spec A).
\end{equation}
Lift a degree-two cocycle modulo two to $c$ and write $dc=2a$ in the free degree-three cycle group. By \eqref{global:ch3}, choose $b$ with $db=a$. Then $c-2b$ is an integral cocycle, filled by \eqref{global:milnor_h2}. Reduction of its filler proves \eqref{global:modtwo_h2}. The point of the correction $c\mapsto c-2b$ is that its boundary
vanishes integrally. We can therefore apply the integral exactness
before reducing the resulting preimage modulo two.

\subsection{Filtered Witt cochains and Milnor two-torsion}

The next correction moves the degree-two preimage one step further
into the fundamental-ideal filtration. Let $\alpha$ be a degree-three
$C_{I,4}(Y)$ cocycle. Its coefficients lie in $I^1$, and
\eqref{global:i3} gives a degree-two $I^1$ cochain $\beta$ with
boundary $\alpha$. Its discriminants form a degree-two cocycle of $C_{\mathrm M,3}(Y)/2$, because the next boundary is the rank parity of $\alpha$ and is zero. By \eqref{global:modtwo_h2}, these discriminants are the boundary
of a degree-one $K_2/2$ cochain. Lift that cochain fieldwise to
$I^2$ forms. The $e_2$/discriminant residue square makes the boundary of this lift have the same discriminant as $\beta$. Subtract it. The resulting degree-two cochain has zero discriminant and hence
lies in $I^2$. Its boundary is still $\alpha$. Thus
\begin{equation}\label{global:i4}
 H^3(C_{I,4}(R))=H^3(C_{I,4}(A))=0.
\end{equation}
The residue squares for $e_1$ and $e_2$ give the two successive corrections.

On $A$ a degree-three cochain of $C_{\mathrm M,4}/2$ is a tuple of unit square classes and has no further boundary. Lift it to a tuple of $I^1$ forms. By \eqref{global:i4}, this tuple is the boundary of a degree-two
$I^2$ cochain. Its $e_2$ invariant has the prescribed unit-square-class boundary. Therefore
\[
 H^3(C_{\mathrm M,4}(A)/2)=0.
\]
As this is the final term of the complex, its cokernel is $B_A/2B_A$. Hence $B_A=2B_A$, and the exponent-two bound \eqref{global:milnor_bound} gives
\begin{equation}\label{global:milnor_h3}
 H^3(C_{\mathrm M,4}(A))=H^3(C_{\mathrm M,4}(R))=0.
\end{equation}
The equality for $R$ follows from \eqref{global:milnor_bound}.

\subsection{Compatibility of the Milnor and Witt preimages}

\begin{proposition}[The top Milnor--Witt field-complex vanishing]\label{global:mw4_field}
For $Y=\Spec R$ and $Y=\Spec A$, the twisted field complex satisfies $H^3(C_{\mathrm{MW},4}(Y))=0$.
\end{proposition}
\begin{proof}
The two projections of a Milnor--Witt cocycle can already be filled
separately. To obtain a Milnor--Witt preimage, their chosen preimages
must agree modulo two in the field pullback. We shall correct the
Milnor preimage by an integral cocycle, so that its boundary stays
fixed while its reduction acquires the required value.

Let $\alpha$ be a degree-three Milnor--Witt cocycle. Its Witt and Milnor projections have degree-two preimages under the
respective boundary maps,
\[
 \beta_I\in\bigoplus_{Y^{(2)}}I^2(\kappa(x),\omega_x),
 \qquad
 \beta_M\in\bigoplus_{Y^{(2)}}K_2(\kappa(x))
\]
by \eqref{global:i4} and \eqref{global:milnor_h3}. Their discrepancy modulo two is a degree-two cocycle $\gamma$ of $C_{\mathrm M,4}(Y)/2$, since their boundaries are the two coordinates of the same $\alpha$.

Choose an integral $K_2$ cochain $c$ lifting $\gamma$ and write $dc=2a$. On $R$, the equality $2da=0$ takes place in the free final group of codimension-four integers, so $da=0$. On $A$ the next group is absent. Equation~\eqref{global:milnor_h3} gives $b$ with $db=a$. Thus $c-2b$ is an integral degree-two cocycle with reduction $\gamma$. Subtract the resulting discrepancy cocycle from $\beta_M$.
Because the correction is an integral cocycle, the boundary of
$\beta_M$ does not change; its reduction now agrees with $e_2\beta_I$.

The field pullback, including characteristic two and the purity line,
\[
 K_2^{\mathrm{MW}}(E,\omega)
       =K_2^{\mathrm M}(E)\mathbin{\times}_{K_2^{\mathrm M}(E)/2}I^2(E,\omega),
\]
now supplies a degree-two Milnor--Witt cochain $\beta$. Its two boundary projections agree with those of $\alpha$. The corresponding degree-one field pullback detects equality, so $d\beta=\alpha$. Every boundary used is the compatible supported map of Lemma~\ref{gw:low_residues}.
\end{proof}

\subsection{From the field complex to the naive Zariski sheaf}

\begin{theorem}[The top cohomology of the naive sheaf]\label{global:naive_mw4}
For $A=R[1/r]$ as in Theorem~\ref{main:quillen},
\[
 H^3_{\Zar}(\Spec A,\mathcal K_4^{\mathrm{MW}})=0.
\]
\end{theorem}
\begin{proof}
We compare the naive sheaf with the point-sheaf complex whose global
sections were studied in Proposition~\ref{global:mw4_field}. The
augmentation need not be an isomorphism in every cohomological degree.
What matters here is that its error sheaves are supported in
codimensions large enough to make the comparison an isomorphism in
third cohomology.

\proofstep{1. Residue fields on the parameter open.}
First, every point of $X=\Spec A$ has infinite residue field. Its prime $\mathfrak p$ in the local ring $R$ is properly below $\fm$, which contains $r$. The local domain $R/\mathfrak p$ is therefore not a field. If its fraction field were finite, the domain itself would be a finite domain and hence a field, a contradiction. Also $\dim X\leq3$.

\proofstep{2. The point-sheaf complex.}
On the small Zariski site form the point-sheaf complex
\[
 \mathcal C^p=\bigoplus_{x\in X^{(p)}}i_{x*}
                   K_{4-p}^{\mathrm{MW}}(\kappa(x),\omega_x),\qquad0\leq p\leq3.
\]
Its global sections are the complex of Proposition~\ref{global:mw4_field}. Unit symbols have zero first residue, giving the natural augmentation $\mathcal K_4^{\mathrm{MW}}[0]\to\mathcal C^\bullet$. Its first differential is the DVR residue: the DVR image theorem, Lemma~\ref{loc:dvr_image}, expresses classes using unit symbols and uniformizer-times-unit symbols; the latter have the expected coefficient-one boundary by the unit-product formula and the degree-three field comparison.

\proofstep{3. Local cohomology of the complex.}
At the ordinary local stalk $T=\mathcal O_{X,x}$, the complex is
$C_{\mathrm{MW},4}(T)$. To see this, represent a cochain and any
relations involving it by finitely many field coefficients. They
descend to one common neighborhood of $x$. Passing to the stalk
removes the primes whose closures do not contain $x$, while the
purity lines and their boundary maps localize to those of $T$.

\proofstep{4. Supports of the comparison sheaves.}
Let $\varepsilon:\mathcal K_4^{\mathrm{MW}}\to\mathcal H^0(\mathcal C^\bullet)$ be the induced augmentation. The local results of Part~III give the following support bounds:
\begin{itemize}
\item $\ker\varepsilon$ is zero at the generic point;
\item $\coker\varepsilon$ is zero through codimension one, by the ordinary DVR \emph{image} theorem;
\item $\mathcal H^1(\mathcal C^\bullet)$ is zero through codimension two, by Corollary~\ref{loc:original_open}, which combines the surface and DVR results with the empty positive-degree complex at the generic point;
\item $\mathcal H^2(\mathcal C^\bullet)=0$, by Theorem~\ref{loc:mw4_h2} on every local ring of dimension at most three;
\item $\mathcal H^3(\mathcal C^\bullet)=0$: in a dimension-three
regular local ring $T$, a regular system of parameters $t_1,t_2,t_3$ gives the
DVR $T/(t_1,t_2)$, with the image of $t_3$ as uniformizer
\cite[Corollary~10.15 and Proposition~11.1]{Eisenbud1995}.
Its divisor is the closed point with coefficient one.
The supported unit-product residue formula of
Lemma~\ref{gw:low_residues}, with the compatible purity-line orientation,
therefore lifts every terminal unit-symbol coefficient.
Lower dimensions have no such term.
\end{itemize}
These stalkwise bounds concern the augmentation kernel and cokernel
and the cohomology sheaves of the point-sheaf complex.
They control the augmentation cone, while
Proposition~\ref{global:mw4_field} supplies the vanishing of the third
cohomology of the global field cochain complex. Both enter the final
comparison in Step~6.

\proofstep{5. The augmentation cone.}
By Lemma~\ref{bridge:support}, the cohomological dimension of $\ker\varepsilon$ is at most $2$, that of $\coker\varepsilon$ at most $1$, and that of $\mathcal H^1(\mathcal C^\bullet)$ at most $0$. Let $\mathcal E$ be the cone of the augmentation. Its only possibly nonzero cohomology sheaves are these three, in degrees $-1,0,1$.
For this bounded complex, the hypercohomology spectral sequence has
$E_2^{a,b}=H^a(X,\mathcal H^b(\mathcal E))$ and total degree $a+b$;
see \cite[Application~5.7.10]{WeibelHomological1994}.
The sheaf in degree
$-1$ has cohomology only through degree two, the sheaf in degree zero
only through degree one, and the sheaf in degree one only in degree
zero. Thus every possibly nonzero term has total degree at most one.
In particular,
\[
 \mathbb H^2(X,\mathcal E)=\mathbb H^3(X,\mathcal E)=0.
\]

\proofstep{6. Hypercohomology and global cochains.}
Each point pushforward in $\mathcal C^\bullet$ is flasque. The direct sums remain flasque here: a section is locally supported in finitely many summands, and a finite subcover of a Noetherian open makes that support finite. Restriction merely omits components. Hence $\mathbb H^3(X,\mathcal C^\bullet)$ is computed by its global field complex and is zero by Proposition~\ref{global:mw4_field}.
By the augmentation triangle and the two cone vanishings above,
$H^3(X,\mathcal K_4^{\mathrm{MW}})
\simeq\mathbb H^3(X,\mathcal C^\bullet)$.
This proves the asserted vanishing.
\end{proof}

%% file: sections/05_projective_conclusion.tex
\section{Symplectic framing and freeness in every rank}\label{main:conclusion}
We now assemble the proof. Stable freeness and symplectic completion
settle ranks at least three. For rank two, alternating Witt vanishing
first gives a finite hyperbolic stabilization. We cancel its additional
planes in the stable range, leaving the rank-four frame required by
Theorem~\ref{bridge:criterion}.

Throughout this section, $R$, $r$, and $A$ satisfy the hypotheses of Theorem~\ref{main:quillen}. In particular, $A$ is a domain of dimension at most three. Theorem~\ref{global:naive_mw4} and Theorem~\ref{bridge:criterion} already give
\begin{equation}\label{main:sp4_transitive}
 \Um_4(A)=\Sp_4(A)e_1.
\end{equation}
Thus every unimodular column of length four has a symplectic completion over $A$.

\subsection{Stable freeness}

\begin{lemma}[Stable freeness and freeness in ranks at least three]\label{main:stable_free}
Every projective $A$-module is stably free, and $K_0(A)=\Z$ by rank. Every projective of rank at least three is free.
\end{lemma}
\begin{proof}
Let $P$ be projective. Clear denominators in a finite presentation to obtain a finite $R$-module $M$ with $M[1/r]\simeq P$.
Since $R$ is regular local of dimension four, $M$ has a finite
projective resolution of length at most four. Its projective terms
are finite free because $R$ is local; see \cite[Theorem~19.2 and
Corollary~19.6]{Eisenbud1995}. Localization gives a finite free
$A$-resolution of $P$.
The surjection onto $P$ splits because $P$ is projective. Its kernel
is therefore projective, so the preceding surjection splits in turn.
Proceeding through the finite resolution gives an alternating
direct-sum identity, hence an isomorphism $P\oplus A^a\simeq A^b$. This proves stable freeness and the asserted $K_0$ calculation.

Bass cancellation for a Noetherian ring of dimension at most three makes every stably free module of rank at least four free. If $Q$ has rank three, then $Q\oplus A$ has rank four and is free. A chosen isomorphism $Q\oplus A\simeq A^4$ presents $Q$ as the kernel of a unimodular row of length four. The transpose of a completion in \eqref{main:sp4_transitive} is an invertible completion of that row, and its kernel is free of rank three.
\end{proof}

We will also use Bass's elementary transitivity on rows of length at least $\dim A+2$ \cite[Theorems~9.3 and~11.1]{BassStable}.

\subsection{Metabolic alternating forms split integrally}

To pass from a Witt-class relation to an isometry, we need an explicit
splitting of a metabolic form. The following elementary construction
adjusts a chosen complement by a bilinear form with prescribed
antisymmetrization.

\begin{lemma}[Hyperbolicity of a metabolic alternating form]\label{main:metabolic}
Over an arbitrary commutative ring, a perfect alternating space with a projective Lagrangian is hyperbolic.
\end{lemma}
\begin{proof}
Let $L$ be the Lagrangian. Split the underlying exact sequence with quotient $L^*$. In this decomposition the pairing is
\[
 \langle(l,\alpha),(m,\beta)\rangle
    =\beta(l)-\alpha(m)+c(\alpha,\beta),
\]
where $c$ is an alternating form on $L^*$. There is a bilinear form $b$ on $L^*$ whose antisymmetrization is $-c$. Indeed choose a projective complement so that $L^*\oplus Q$ is free, extend $c$ by zero, choose a strictly triangular matrix with antisymmetrization $-c$, and restrict it to $L^*$.

View $b$ as a map $t:L^*\to L$ by $\beta(t\alpha)=b(\alpha,\beta)$. Replacing the complement by the graph $\alpha\mapsto(t\alpha,\alpha)$ changes its self-pairing by $b-b^t=-c$. The new complement is isotropic and the cross-pairing is unchanged. This is the hyperbolic form on $L\oplus L^*$.
\end{proof}

\begin{lemma}[A rank-four symplectic frame]\label{main:actual_frame}
Every oriented rank-two projective $P$ over $A$ admits an oriented isometry $P\perp \Hyp\simeq \Hyp^2$.
\end{lemma}
\begin{proof}
We first obtain a hyperbolic stabilization from the vanishing of the
alternating Witt group. We then cancel free hyperbolic planes while
the standard target remains in the stable range. The final isometry
retains the given rank-two space as the perpendicular complement.

\proofstep{1. The alternating form on the projective module.}
Write $\theta:\det P\xrightarrow{\sim}A$ for the chosen orientation.
For $p,q\in P$, it defines the pairing
$\langle p,q\rangle=\theta(p\wedge q)$. In an oriented
local basis this pairing is the standard perfect alternating form.
The frame constructed below is an isometry for this form.
By Theorem~\ref{gw:global_walt}, the alternating Witt group is zero. The hyperbolic map $h:K_0(A)\to K_0^{\mathrm{Sp}}(A)$ is therefore onto. Since $K_0(A)=\Z$ by Lemma~\ref{main:stable_free}, rank comparison gives $[P]=[\Hyp]$ in the Grothendieck group of alternating spaces.

\proofstep{2. A finite hyperbolic stabilization.}
Group completion supplies an alternating space $C$ with $P\perp C\simeq \Hyp\perp C$. Add $-C$ to both sides. The diagonal in $C\perp(-C)$ is a projective Lagrangian, so Lemma~\ref{main:metabolic} identifies this sum with the hyperbolic space on the underlying projective module $Q$ of $C$. Choose a finite projective complement $Q'$ with $Q\oplus Q'\simeq A^m$ and add $\Hyp(Q')$. Thus
\begin{equation}\label{main:finite_hyperbolic}
 P\perp \Hyp^m\simeq \Hyp^{m+1}
\end{equation}
for some finite $m$. This also gives the group-completion interpretation when alternating
Witt theory is defined by metabolic relations: each such relation is
hyperbolic by Lemma~\ref{main:metabolic}, and adding a projective
complement turns it into a finite free hyperbolic stabilization.

\proofstep{3. Cancellation in the stable range.}
If $m=0$, the desired conclusion is immediate, and $m=1$ is already the required frame. If $m\geq2$, choose a free hyperbolic pair on the left of \eqref{main:finite_hyperbolic}. Its image is a hyperbolic pair $(u,v)$ in the standard free space of rank $2m+2\geq6$. Bass elementary transitivity and the characteristic-free equality of elementary linear and symplectic column orbits give an elementary symplectic matrix taking $u$ to $e_1$. The latter equality is \cite[Lemma~2.11 and Theorem~3.9]{Gupta}.

\proofstep{4. The partners of the standard columns.}
After that change write $v=e_2+a e_1+w$, with $w$ in the standard perpendicular complement of $Ae_1+Ae_2$. The transformation fixing $e_1$, sending
\[
 e_2\longmapsto e_2+w+a e_1,
 \qquad z\longmapsto z+\langle w,z\rangle e_1
\]
on that complement is symplectic, with inverse obtained by replacing $(w,a)$ by $(-w,-a)$. Its inverse sends $v$ to $e_2$. Restricting the resulting isometry to the perpendicular complements gives
\[
 P\perp \Hyp^{m-1}\simeq \Hyp^m.
\]
Repeat until $m=1$. At each step the standard target has rank at least six, so the stable-range theorem applies. The final isometry is the required rank-four frame.
\end{proof}

\subsection{Completion of the theorem}

Stable freeness trivializes $\det P$ for every rank-two projective. Fix any of its orientations and apply Lemma~\ref{main:actual_frame}. Theorem~\ref{bridge:criterion}, with Theorem~\ref{global:naive_mw4}, shows that $P$ is free. The orientation can be realized by its basis as in the criterion's proof. Rank-one projectives are free because taking determinants in a stable trivialization gives $P\simeq A$. Rank-zero projectives are zero, since their localizations at all primes are zero. Ranks at least three are covered by Lemma~\ref{main:stable_free}. Together with \eqref{main:sp4_transitive}, this proves every assertion of Theorem~\ref{main:quillen}.

\begin{remark}
For a rank-two projective module, the proof constructs an isometry
$P\perp \Hyp\simeq \Hyp^2$ and applies the completion criterion to the
corresponding framed complement. The resulting basis can be rescaled
to induce any chosen determinant orientation.
\end{remark}

%% file: frontmatter/ai_disclosure.tex
\section*{Use of AI tools}
\addcontentsline{toc}{section}{Use of AI tools}

OpenAI's ChatGPT and Codex (GPT-6) were used to develop the proof and
prepare the manuscript. The human author takes full responsibility for
the mathematical arguments and the final manuscript.

%% file: main.bbl
\newcommand{\etalchar}[1]{$^{#1}$}
\begin{thebibliography}{CDH{\etalchar{+}}26}

\bibitem[AF17]{AsokFaselKO}
Aravind Asok and Jean Fasel.
\newblock An explicit {$KO$}-degree map and applications.
\newblock {\em Journal of Topology}, 10(1):268--300, 2017.
\newblock \href{https://doi.org/10.1112/topo.12007}{DOI
  \nolinkurl{10.1112/topo.12007}};
  \href{https://arxiv.org/abs/1403.4588}{Text}.

\bibitem[AGV72]{SGA4}
Michael Artin, Alexander Grothendieck, and Jean-Louis Verdier.
\newblock {\em Th{\'e}orie des topos et cohomologie {\'e}tale des sch{\'e}mas},
  volume 269--270 of {\em Lecture Notes in Mathematics}.
\newblock Springer, 1972.
\newblock SGA 4, Expos{\'e} VI.

\bibitem[Bac17]{LOCBachmannSlices}
Tom Bachmann.
\newblock The generalized slices of {Hermitian} {$K$}-theory.
\newblock {\em Journal of Topology}, 10(4):1124--1144, 2017.
\newblock Cited version:
  \href{https://arxiv.org/abs/1610.01346v3}{arXiv:1610.01346v3}.

\bibitem[Bac24]{GWBachmann}
Tom Bachmann.
\newblock Strongly {$\mathbb A^1$}-invariant sheaves (after {F. Morel}), 2024.
\newblock Version 1, 17 June 2024.
  \href{https://arxiv.org/abs/2406.11526v1}{arXiv:2406.11526v1}.

\bibitem[Bas64]{BassStable}
Hyman Bass.
\newblock {$K$}-theory and stable algebra.
\newblock {\em Publications Math{\'e}matiques de l'IH{\'E}S}, 22:5--60, 1964.
\newblock \href{https://doi.org/10.1007/BF02684689}{DOI
  \nolinkurl{10.1007/BF02684689}};
  \href{https://www.numdam.org/article/PMIHES_1964__22__5_0.pdf}{Text}.

\bibitem[BEM25]{BEM}
Tom Bachmann, Elden Elmanto, and Matthew Morrow.
\newblock {$\mathbb A^1$}-invariant motivic cohomology of schemes, 2025.
\newblock Version 1, 13 August 2025.
  \href{https://arxiv.org/abs/2508.09915v1}{arXiv:2508.09915v1}.

\bibitem[BH21]{BachmannHopkins}
Tom Bachmann and Michael~J. Hopkins.
\newblock {$\eta$}-periodic motivic stable homotopy theory over fields, 2021.
\newblock Version 4, 3 May 2021.
  \href{https://arxiv.org/abs/2005.06778v4}{arXiv:2005.06778v4}.

\bibitem[BM23]{BhattMathew}
Bhargav Bhatt and Akhil Mathew.
\newblock Syntomic complexes and {$p$}-adic {\'e}tale {Tate} twists.
\newblock {\em Forum of Mathematics, Pi}, 11:e1, 2023.
\newblock \href{https://doi.org/10.1017/fmp.2022.21}{Published article}.

\bibitem[Bou25]{BouisWeibel}
Tess Bouis.
\newblock Weibel vanishing and the projective bundle formula for mixed
  characteristic motivic cohomology, 2025.
\newblock Version 1, 22 July 2025.
  \href{https://arxiv.org/abs/2507.16501v1}{arXiv:2507.16501v1}.

\bibitem[Bou26]{BouisMotivic}
Tess Bouis.
\newblock Motivic cohomology of mixed characteristic schemes, 2026.
\newblock Version 3, 7 September 2026.
  \href{https://arxiv.org/abs/2412.06635v3}{arXiv:2412.06635v3}.

\bibitem[Car23]{GWCarlier}
Robin Carlier.
\newblock {Milnor--Witt} {$K$}-theory and {Witt} {$K$}-theory of a field, 2023.
\newblock Version 1, 29 June 2023.
  \href{https://arxiv.org/abs/2306.16985v1}{arXiv:2306.16985v1}.

\bibitem[CDH{\etalchar{+}}21]{GWCDHI}
Baptiste Calm{\`e}s, Emanuele Dotto, Yonatan Harpaz, Fabian Hebestreit, Markus
  Land, Kristian Moi, Denis Nardin, Thomas Nikolaus, and Wolfgang Steimle.
\newblock Hermitian {$K$}-theory for stable {$\infty$}-categories {I}:
  Foundations, 2021.
\newblock Version 3, 27 December 2021.
  \href{https://arxiv.org/abs/2009.07223v3}{arXiv:2009.07223v3}.

\bibitem[CDH{\etalchar{+}}25]{GWCDHII}
Baptiste Calm{\`e}s, Emanuele Dotto, Yonatan Harpaz, Fabian Hebestreit, Markus
  Land, Kristian Moi, Denis Nardin, Thomas Nikolaus, and Wolfgang Steimle.
\newblock Hermitian {$K$}-theory for stable {$\infty$}-categories {II}:
  Cobordism categories and additivity, 2025.
\newblock Version 5, 11 April 2025.
  \href{https://arxiv.org/abs/2009.07224v5}{arXiv:2009.07224v5}.

\bibitem[CDH{\etalchar{+}}26]{GWCDHIII}
Baptiste Calm{\`e}s, Emanuele Dotto, Yonatan Harpaz, Fabian Hebestreit, Markus
  Land, Kristian Moi, Denis Nardin, Thomas Nikolaus, and Wolfgang Steimle.
\newblock Hermitian {$K$}-theory for stable {$\infty$}-categories {III}:
  {Grothendieck--Witt} groups of rings, 2026.
\newblock Version 4, 27 April 2026.
  \href{https://arxiv.org/abs/2009.07225v4}{arXiv:2009.07225v4}.

\bibitem[{\v C}es19]{Cesnavicius}
K{\k e}stutis {\v C}esnavi{\v c}ius.
\newblock Purity for the {Brauer} group.
\newblock {\em Duke Mathematical Journal}, 168(8):1461--1486, 2019.
\newblock \href{https://doi.org/10.1215/00127094-2018-0057}{DOI
  \nolinkurl{10.1215/00127094-2018-0057}};
  \href{https://arxiv.org/abs/1711.06456}{Text}.

\bibitem[{\v C}es22]{CesnaviciusTorsorsSurvey}
K{\k e}stutis {\v C}esnavi{\v c}ius.
\newblock Problems about torsors over regular rings.
\newblock {\em Acta Mathematica Vietnamica}, 47(1):39--107, 2022.
\newblock With an appendix by Yifei Zhao. Cited version:
  \href{https://arxiv.org/abs/2201.06424v4}{arXiv:2201.06424v4}.

\bibitem[{\v C}es24]{CesnaviciusSmoothDivisor}
K{\k e}stutis {\v C}esnavi{\v c}ius.
\newblock Torsors on the complement of a smooth divisor.
\newblock {\em Cambridge Journal of Mathematics}, 12(4):903--936, 2024.
\newblock Cited version:
  \href{https://arxiv.org/abs/2204.08233v6}{arXiv:2204.08233v6}.

\bibitem[CHN25]{GWCHN}
Baptiste Calm{\`e}s, Yonatan Harpaz, and Denis Nardin.
\newblock A motivic spectrum representing hermitian {$K$}-theory, 2025.
\newblock Version 2, 24 January 2025.
  \href{https://arxiv.org/abs/2402.15136v2}{arXiv:2402.15136v2}.

\bibitem[CM21]{ClausenMathew}
Dustin Clausen and Akhil Mathew.
\newblock Hyperdescent and {\'e}tale {$K$}-theory, 2021.
\newblock Version 3, 18 March 2021.
  \href{https://arxiv.org/abs/1905.06611v3}{arXiv:1905.06611v3}.

\bibitem[D{\'e}g23]{DegliseMW}
Fr{\'e}d{\'e}ric D{\'e}glise.
\newblock Notes on {Milnor--Witt} {$K$}-theory, 2023.
\newblock Version 1, 29 May 2023.
  \href{https://arxiv.org/abs/2305.18609v1}{arXiv:2305.18609v1}.

\bibitem[Eis95]{Eisenbud1995}
David Eisenbud.
\newblock {\em Commutative Algebra: With a View Toward Algebraic Geometry},
  volume 150 of {\em Graduate Texts in Mathematics}.
\newblock Springer-Verlag, New York, 1995.

\bibitem[EM26]{LOCElmantoMorrow}
Elden Elmanto and Matthew Morrow.
\newblock Motivic cohomology of equicharacteristic schemes, 2026.
\newblock Version 2, 27 March 2026.
  \href{https://arxiv.org/abs/2309.08463v2}{arXiv:2309.08463v2}.

\bibitem[FGV22]{GWFGV}
Tony Feng, S{\o}ren Galatius, and Akshay Venkatesh.
\newblock The {Galois} action on symplectic {$K$}-theory.
\newblock {\em Inventiones Mathematicae}, 230(1):225--319, 2022.
\newblock \href{https://doi.org/10.1007/s00222-022-01127-8}{DOI
  \nolinkurl{10.1007/s00222-022-01127-8}}.

\bibitem[Gei04]{Geisser}
Thomas Geisser.
\newblock Motivic cohomology over {Dedekind} rings.
\newblock {\em Mathematische Zeitschrift}, 248:773--794, 2004.
\newblock \href{https://www2.rikkyo.ac.jp/web/geisser/Dedekind.pdf}{Article}.

\bibitem[Gup15]{Gupta}
Anjan Gupta.
\newblock Optimal injective stability for the symplectic {$K_1Sp$} group.
\newblock {\em Journal of Pure and Applied Algebra}, 219(5):1336--1348, 2015.
\newblock Cited version:
  \href{https://arxiv.org/abs/1511.09419v1}{arXiv:1511.09419v1}, 30 November
  2015. \href{https://doi.org/10.1016/j.jpaa.2014.06.008}{DOI
  \nolinkurl{10.1016/j.jpaa.2014.06.008}}.

\bibitem[Har77]{Hartshorne1977}
Robin Hartshorne.
\newblock {\em Algebraic Geometry}, volume~52 of {\em Graduate Texts in
  Mathematics}.
\newblock Springer-Verlag, New York, 1977.

\bibitem[Hoo06]{Hoobler}
Raymond~T. Hoobler.
\newblock The {Merkuriev--Suslin} theorem for any semi-local ring.
\newblock {\em Journal of Pure and Applied Algebra}, 207(3):537--552, 2006.
\newblock Cited version:
  \href{https://arxiv.org/abs/math/0503282v2}{arXiv:math/0503282v2}, 15 March
  2005. \href{https://doi.org/10.1016/j.jpaa.2005.10.020}{DOI
  \nolinkurl{10.1016/j.jpaa.2005.10.020}}.

\bibitem[HT10]{BridgeHT}
Kevin Hutchinson and Liqun Tao.
\newblock Homology stability for the special linear group of a field and
  {Milnor--Witt} {$K$}-theory.
\newblock {\em Documenta Mathematica}, Extra Volume: Andrei A. Suslin's
  Sixtieth Birthday:267--315, 2010.
\newblock Published version, revised 1 April 2010.
  \href{https://emis.de/ft/9818}{Article}.

\bibitem[HW15]{BridgeHW}
Kevin Hutchinson and Matthias Wendt.
\newblock On the third homology of {$SL_2$} and weak homotopy invariance.
\newblock {\em Transactions of the American Mathematical Society},
  367(10):7481--7513, 2015.
\newblock Cited version:
  \href{https://arxiv.org/abs/1307.3069v3}{arXiv:1307.3069v3}, 25 April 2014.
  \href{https://doi.org/10.1090/S0002-9947-2014-06495-7}{DOI
  \nolinkurl{10.1090/S0002-9947-2014-06495-7}}.

\bibitem[Ker08]{KerzThesis}
Moritz Kerz.
\newblock {\em Milnor {$K$}-theory of local rings}.
\newblock PhD thesis, Universit{\"a}t Regensburg, 2008.
\newblock \href{https://d-nb.info/98935329x/34}{Dissertation}.

\bibitem[KSW16]{GWKSW}
Max Karoubi, Marco Schlichting, and Charles Weibel.
\newblock The {Witt} group of real algebraic varieties.
\newblock {\em Journal of Topology}, 9(4):1257--1302, 2016.
\newblock Cited version:
  \href{https://arxiv.org/abs/1506.03862v3}{arXiv:1506.03862v3}, 16 June 2016.
  \href{https://doi.org/10.1112/jtopol/jtw024}{DOI
  \nolinkurl{10.1112/jtopol/jtw024}}.

\bibitem[Lan02]{LangAlgebra2002}
Serge Lang.
\newblock {\em Algebra}, volume 211 of {\em Graduate Texts in Mathematics}.
\newblock Springer-Verlag, New York, third revised edition, 2002.

\bibitem[Lev89]{Levine}
Marc Levine.
\newblock The indecomposable {$K_3$} of fields.
\newblock {\em Annales scientifiques de l'{\'E}cole Normale Sup{\'e}rieure},
  22(2):255--344, 1989.
\newblock \href{https://doi.org/10.24033/asens.1585}{DOI
  \nolinkurl{10.24033/asens.1585}};
  \href{https://numdam.org/item/10.24033/asens.1585.pdf}{Text}.

\bibitem[LM]{LudersMorrowK2}
Morten L{\"u}ders and Matthew Morrow.
\newblock Milnor {$K$}-theory of {$p$}-adic rings.
\newblock Author manuscript.
\newblock
  \href{https://www.imo.universite-paris-saclay.fr/~matthew.morrow/Lueders,%20Morrow,%20Milnor%20K-theory%20of%20p-adic%20rings.pdf}{Author
  version}.

\bibitem[L{\"u}d24]{LOCLudersRelative}
Morten L{\"u}ders.
\newblock On the relative {Gersten} conjecture for {Milnor} {$K$}-theory in the
  smooth case, 2024.
\newblock Version 4, 5 July 2024.
  \href{https://arxiv.org/abs/2010.02622v4}{arXiv:2010.02622v4}.

\bibitem[L{\"u}d26]{Luders}
Morten L{\"u}ders.
\newblock {Bloch--Ogus} theory for smooth and semi-stable schemes in mixed
  characteristic.
\newblock {\em {\'E}pijournal de G{\'e}om{\'e}trie Alg{\'e}brique}, 10:Article
  19, 2026.
\newblock Cited version:
  \href{https://arxiv.org/abs/2201.11553v4}{arXiv:2201.11553v4}.

\bibitem[Mor12]{LOCMorelBook}
Fabien Morel.
\newblock {\em {$\mathbb A^1$}-Algebraic Topology over a Field}, volume 2052 of
  {\em Lecture Notes in Mathematics}.
\newblock Springer, 2012.
\newblock Theorem numbering follows the published book.

\bibitem[MS83]{MerkurjevSuslin}
A.~S. Merkur'ev and A.~A. Suslin.
\newblock {$K$}-cohomology of {Severi--Brauer} varieties and the norm residue
  homomorphism.
\newblock {\em Mathematics of the USSR-Izvestiya}, 21(2):307--340, 1983.
\newblock English translation of the Russian original, Izv. Akad. Nauk SSSR
  Ser. Mat. 46 (1982), 1011--1046.
  \href{https://doi.org/10.1070/IM1983v021n02ABEH001793}{DOI
  \nolinkurl{10.1070/IM1983v021n02ABEH001793}};
  \href{https://www.mathnet.ru/eng/im1657}{Article}.

\bibitem[NS90]{GWNS}
Yu.~P. Nesterenko and A.~A. Suslin.
\newblock Homology of the full linear group over a local ring, and {Milnor}'s
  {$K$}-theory.
\newblock {\em Mathematics of the USSR-Izvestiya}, 34(1):121--145, 1990.
\newblock English translation of the 1989 Russian original.
  \href{https://www.mathnet.ru/eng/im1233}{Article}.

\bibitem[Pan03]{Panin}
Ivan~A. Panin.
\newblock The equicharacteristic case of the {Gersten} conjecture.
\newblock {\em Trudy Matematicheskogo Instituta imeni V. A. Steklova},
  241:169--178, 2003.
\newblock English translation: Proceedings of the Steklov Institute of
  Mathematics 241 (2003), 154--163. Citations use the original pagination.
  \href{https://www.mathnet.ru/eng/tm394}{Article}.

\bibitem[Qui72]{QuillenFinite}
Daniel Quillen.
\newblock On the cohomology and {$K$}-theory of the general linear groups over
  a finite field.
\newblock {\em Annals of Mathematics}, 96:552--586, 1972.

\bibitem[Ran73]{GWRanickiFoundations}
A.~A. Ranicki.
\newblock Algebraic {$L$}-theory, {I}: Foundations.
\newblock {\em Proceedings of the London Mathematical Society}, 27(1):101--125,
  1973.
\newblock
  \href{https://www.maths.ed.ac.uk/~v1ranick/papers/ltheory1.pdf}{Author
  version}.

\bibitem[Ran80]{GWRanickiSurgery}
Andrew Ranicki.
\newblock The algebraic theory of surgery. {I}. foundations.
\newblock {\em Proceedings of the London Mathematical Society}, 40(1):87--192,
  1980.
\newblock \href{https://doi.org/10.1112/plms/s3-40.1.87}{DOI
  \nolinkurl{10.1112/plms/s3-40.1.87}};
  \href{https://webhomes.maths.ed.ac.uk/~v1ranick/papers/ats1.pdf}{Author
  version}.

\bibitem[Rao88]{Rao1988}
Ravi~A. Rao.
\newblock The {Bass--Quillen} conjecture in dimension three but characteristic
  {$\ne 2,3$} via a question of {A. Suslin}.
\newblock {\em Inventiones Mathematicae}, 93(3):609--618, 1988.
\newblock \href{https://doi.org/10.1007/BF01410201}{DOI
  \nolinkurl{10.1007/BF01410201}};
  \href{https://link.springer.com/article/10.1007/BF01410201}{Text}.

\bibitem[Sch88]{Schneider}
Peter Schneider.
\newblock Introduction to the {Beilinson} conjectures.
\newblock In Michael Rapoport, Norbert Schappacher, and Peter Schneider,
  editors, {\em Beilinson's Conjectures on Special Values of {$L$}-Functions},
  volume~4 of {\em Perspectives in Mathematics}. Academic Press, 1988.
\newblock
  \href{https://virtualmath1.stanford.edu/~conrad/BSDseminar/refs/BeilinsonintroI.pdf}{Chapter
  text}.

\bibitem[Sch17]{SchlichtingEuler}
Marco Schlichting.
\newblock Euler class groups and the homology of elementary and special linear
  groups.
\newblock {\em Advances in Mathematics}, 320:1--81, 2017.
\newblock Cited version:
  \href{https://warwick.ac.uk/fac/sci/maths/people/staff/marco_schlichting/research/eulerclassfinal.pdf}{accepted
  author manuscript}, 10 August 2017.
  \href{https://doi.org/10.1016/j.aim.2017.08.034}{DOI
  \nolinkurl{10.1016/j.aim.2017.08.034}}.

\bibitem[Sch25a]{SchlichtingSpStability}
Marco Schlichting.
\newblock On the homology stability range for symplectic groups.
\newblock {\em Compositio Mathematica}, 161:2724--2753, 2025.
\newblock \href{https://doi.org/10.1112/S0010437X25102856}{DOI
  \nolinkurl{10.1112/S0010437X25102856}}.

\bibitem[Sch25b]{GWSchlichtingComparison}
Marco Schlichting.
\newblock Symmetric versus genuine symmetric forms in {Hermitian} {$K$}-theory,
  2025.
\newblock Version 3, 19 June 2025.
  \href{https://arxiv.org/abs/2503.14288v3}{arXiv:2503.14288v3}.

\bibitem[Sou85]{Soule}
Christophe Soul{\'e}.
\newblock Op{\'e}rations en {$K$}-th{\'e}orie alg{\'e}brique.
\newblock {\em Canadian Journal of Mathematics}, 37(3):488--550, 1985.
\newblock \href{https://doi.org/10.4153/CJM-1985-029-x}{DOI
  \nolinkurl{10.4153/CJM-1985-029-x}};
  \href{https://doi.org/10.4153/CJM-1985-029-x}{Text}.

\bibitem[SS25]{SchlichtingSarwar}
Husney~Parvez Sarwar and Marco Schlichting.
\newblock The third homology of symplectic groups and algebraic {$K$}-theory.
\newblock {\em American Journal of Mathematics}, 147(3):819--867, 2025.
\newblock Cited version:
  \href{https://preprint.press.jhu.edu/ajm/sites/default/files/AJM-sarwar-schlichting-FINAL.pdf}{final
  author manuscript}. \href{https://doi.org/10.1353/ajm.2025.a961348}{DOI
  \nolinkurl{10.1353/ajm.2025.a961348}}.

\bibitem[Swa88]{Swan1988}
Richard~G. Swan.
\newblock A simple proof of {Gabber}'s theorem on projective modules over a
  localized local ring.
\newblock {\em Proceedings of the American Mathematical Society},
  103(4):1025--1030, 1988.

\bibitem[{The}26]{Stacks}
{The Stacks Project Authors}.
\newblock The {Stacks Project}, 2026.
\newblock \url{https://stacks.math.columbia.edu/}. Accessed 5 October 2026.

\bibitem[TT90]{TT}
Robert~W. Thomason and Thomas Trobaugh.
\newblock Higher algebraic {$K$}-theory of schemes and of derived categories.
\newblock In Pierre Cartier, Luc Illusie, Nicholas~M. Katz, G{\'e}rard Laumon,
  Yuri~I. Manin, and Kenneth~A. Ribet, editors, {\em The Grothendieck
  Festschrift, Volume III}, volume~88 of {\em Progress in Mathematics}, pages
  247--435. Birkh{\"a}user, Boston, 1990.
\newblock \href{https://doi.org/10.1007/978-0-8176-4576-2_10}{DOI
  \nolinkurl{10.1007/978-0-8176-4576-2_10}};
  \href{https://www.sas.rochester.edu/mth/sites/doug-ravenel/otherpapers/Thomason-Trobaugh.pdf}{Text}.

\bibitem[Wei93]{WeibelChern}
Charles Weibel.
\newblock {\'E}tale {Chern} classes at the prime 2.
\newblock In {\em Algebraic {$K$}-Theory and Algebraic Topology}, volume 407 of
  {\em NATO ASI Series C}, pages 249--286. Kluwer, 1993.
\newblock Cited version:
  \href{https://sites.math.rutgers.edu/~weibel/archive/papers-dir/chernclass.pdf}{author
  manuscript}.

\bibitem[Wei94]{WeibelHomological1994}
Charles~A. Weibel.
\newblock {\em An Introduction to Homological Algebra}, volume~38 of {\em
  Cambridge Studies in Advanced Mathematics}.
\newblock Cambridge University Press, Cambridge, 1994.

\end{thebibliography}
